\documentclass[matprg]{svjour3}
\usepackage{subcaption}
\usepackage{amsmath,amsxtra,amsfonts,amscd,amssymb,bm}
\usepackage{algorithm}
\usepackage{algpseudocode}
\usepackage[top=3.5cm,bottom=3.5cm,right=3.5cm,left=3.5cm]{geometry}
\usepackage[mathscr]{eucal}
\usepackage{color}
\usepackage{cite}
\numberwithin{equation}{section}
\usepackage{hyperref}
\usepackage{multirow}

\def\d{\delta}

\def\sl{\textstyle{\sum}_{\ell=1}^k}

\newcommand{\tsum}{\textstyle\sum}
\newcommand{\bbe}{\mathbb{E}}

\def\eqref#1{(\ref{#1})}
\usepackage{epsfig}
\usepackage{calc}
\usepackage{amstext}
\usepackage{amsmath}
\usepackage{multicol}
\usepackage{amsfonts}
\usepackage{amssymb}
\usepackage{mathrsfs}
\usepackage{bm}
\usepackage{xcolor}
\usepackage{comment}
\usepackage[scientific-notation=true]{siunitx}

\newcommand{\bbr}{\Bbb{R}}
\newcommand{\beq}{\begin{equation}}
\newcommand{\eeq}{\end{equation}}
\newcommand{\beqa}{\begin{eqnarray}}
\newcommand{\eeqa}{\end{eqnarray}}
\newcommand{\beqas}{\begin{eqnarray*}}
\newcommand{\eeqas}{\end{eqnarray*}}

\newcommand{\revision}[2]{{\color{black}#2}}

\usepackage{algorithm}
\usepackage{algpseudocode}

\def\bbe{{\mathbb{E}}}
\def\vgap{\vspace*{.1in}}
\newcommand{\nn}{\nonumber}

\newcommand{\rf}[1]{(\ref{#1})}
\renewcommand{\top}{{T}}

\def\vgap{\vspace*{.1in}}

\title{A Simple Uniformly Optimal Method  without Line Search for Convex Optimization
}
\author{
  Tianjiao Li
\and
  Guanghui Lan
  \thanks{TL and GL were partially supported by Air Force Office of Scientific Research grant FA9550-22-1-0447.}
}
\institute{
Tianjiao Li\\
  School of Industrial and Systems
    Engineering, Georgia Institute of Technology, Atlanta, GA, 30332.\\
    Email: {\tt tli432@gatech.edu}.
    \vspace{0.2cm}\\
    Guanghui Lan\\
  School of Industrial and Systems
    Engineering, Georgia Institute of Technology, Atlanta, GA, 30332.\\
    Email: {\tt george.lan@isye.gatech.edu}.\\
}

\date{the date of receipt and acceptance should be inserted later}

\begin{document}

\maketitle

\begin{abstract}{
Line search (or backtracking) procedures
have been widely employed into first-order methods for solving convex optimization problems, especially those with unknown problem parameters (e.g., Lipschitz constant). In this paper, we show that line search is superfluous in attaining the optimal rate of convergence for solving a convex optimization problem whose parameters are not given a priori. In particular, we present a novel accelerated gradient descent type algorithm called auto-conditioned fast gradient method (AC-FGM) that can achieve an optimal $\mathcal{O}(1/k^2)$ rate of convergence for smooth convex optimization without requiring the estimate of a global Lipschitz constant or the employment of line search procedures.
    We then extend AC-FGM to solve convex optimization problems with H\"{o}lder continuous gradients and show that it automatically achieves the optimal rates of convergence uniformly for all problem classes with the desired accuracy of the solution as the only input. 
    Finally, we report some encouraging numerical results that demonstrate the advantages of AC-FGM over the previously developed parameter-free methods for convex optimization.}
 \end{abstract}
\keywords{convex optimization \and complexity bounds \and uniformly optimal methods \and parameter-free methods \and line search free methods \and  accelerated gradient descent }
 
\subclass{90C25 \and 90C06 \and 68Q25 }

\section{Introduction}
In this paper, we consider the \revision{basic convex optimization problem}{convex composite optimization problem}
{\color{black}
\begin{align}\label{main_prob}
\Psi^* := \min_{x \in X} \{\Psi(x) := f(x) + h(x)\},
\end{align}
where $X \subseteq \bbr^n$ is a closed convex set and $h: X \rightarrow \mathbb{R}$ is a simple closed convex function.
We assume $f: X \rightarrow \mathbb{R}$ is a closed convex function, represented by a first-order oracle which returns $f(x)$ and $g(x) \in \partial f(x)$ upon request at $x$.} Here, $\partial f(x)$ denotes the subdifferential of $f(\cdot)$ at $x\in X$, and $g(x)$ denotes one of its subgradients. Throughout the paper we assume that the set of optimal solutions $X^*$ is nonempty, and let $x^*$ denote an arbitrary point in $X^*$.
 
Different problem classes of convex optimization have been studied in the literature. 
A problem is said to be nonsmooth if \revision{its objective}{the} function $f$, although not necessarily differentiable, is Lipschtiz continuous, i.e., 
\begin{align}\label{eq_nonsmooth}
|f(x) -f(y)| \leq M\|x-y\|, ~~\forall x, y\in X.
\end{align}
It is well-known that the classic subgradient descent method exhibits an $\mathcal{O}(1/\sqrt{k})$ rate of convergence for nonsmooth convex optimization~\cite{nemirovski1983problem}. Here $k$ denotes the number of calls to the first-order oracle.
This convergence rate can be significantly improved for smooth problems which have a differentiable objective function $f$ with Lipschtiz continuous gradient $g$ such that
 \begin{align}\label{eq:smooth_1}
\|g(x) - g(y)\| \leq L \|x-y\|, ~~\forall x, y \in X. 
 \end{align}
Specifically, while the gradient descent method
exhibits only an $\mathcal{O}(1/k)$ rate of convergence
for smooth problems, the celebrated accelerated gradient descent (AGD) method developed by Nesterov \cite{nesterov1983method}  
has an $\mathcal{O}(1/k^2)$ rate of convergence (see also 
\cite{NemirovskiYudin83-smooth,nemirovski1983problem}
for Nemirovski and Yudin's earlier developments on AGD). 
Between nonsmooth and smooth problems there exists a class of weakly smooth problems with H\"{o}lder continuous gradient, i.e.,
\begin{align}
\|g(x) - g(y)\| \leq L_\nu \|x-y\|^\nu, ~~\text{for some } \nu \in (0,1).
\end{align}
Clearly, when $\nu = 0$ this condition also indicates \eqref{eq_nonsmooth}, and when $\nu = 1$ it reduces to \eqref{eq:smooth_1}. 
Nemirovski and Nesterov~\cite{nemirovskii1985optimal} showed that
the AGD method has an 
\[
\mathcal{O} \left(\left( \tfrac{L_\nu \|x_0-x^*\|^{1+\nu}}{\epsilon}\right)^{\frac{2}{1+3\nu}}\right)
\]
\revision{rate of convergence}{oracle complexity} for solving weakly smooth problems.
In a more recent work, Lan \cite{lan2012optimal} further 
showed that AGD with a properly defined stepsize policy also
exhibits an $\mathcal{O}(1/\sqrt{k})$ rate of convergence
for nonsmooth (and stochastic) problems. Therefore,
 AGD is considered to be a universally optimal method~\cite{lan2012optimal,nesterov2015universal} for solving different classes of convex optimization problems, since its rates of convergence match well 
the lower complexity bounds 
for smooth, nonsmooth, and weakly smooth problems established by Nemirvoski and Yudin in \cite{nemirovski1983problem}.

However, in order to achieve these optimal rates of convergence,
one needs to first know which problem class $f$ belongs to, and then supply AGD with prior knowledge about $f$, including parameters like $L$, $\nu$, and $L_\nu$, as the input of the algorithm. These constants can be unavailable or difficult to compute in many practice scenarios. Additionally, problem classification and parameter estimation over a global scope of $f$ may lead to overly conservative stepsize choices and as a consequence can slow down convergence of the algorithm. To alleviate these concerns, Lan~\cite{Lan13-1}
suggested studying the so-called {\em uniformly optimal} methods, signifying that they can achieve the optimal convergence guarantees across all classes of smooth, weakly smooth, and nonsmooth convex optimization problems without requiring much information on problem parameters. Moreover,
by noticing that the classic bundle-level method \cite{LNN} does not require any problem parameters for nonsmooth optimization, Lan~\cite{Lan13-1} introduced a few novel accelerated bundle-level type methods that are uniformly optimal for convex optimization. These methods can terminate based on the gap between computable upper and lower bounds without using the target accuracy in the updating formula for the iterates. However, these bundle-level type methods require the solution of a more complicated subproblem than AGD in each iteration. In addition, the analysis in \cite{Lan13-1} requires the feasible region $X$ to be bounded, which may not hold in some applications.
In an effort to address these issues, Nesterov in an important work \cite{nesterov2015universal} presented a new fast gradient method (FGM) obtained by incorporating a novel line search (or backtracking) procedure \cite{armijo1966minimization} into the AGD method.
He showed that FGM achieves uniformly (or universally) the optimal rate of convergence for solving smooth, weakly smooth, and nonsmooth convex optimization problems with the accuracy of the solution as the only input. While the cost per iteration of FGM might be higher than AGD due to the line search procedure,
the total number of calls to the first-order oracle will still be in the same order of magnitude as AGD up to \revision{a}{an additive} constant factor.
Further developments for adaptive AGD methods with line search can be found in \cite{roulet2017sharpness,liang2023unified, renegar2022simple,grimmer2023optimal,lu2023accelerated}, among others. 

Due to the wide deployment of first-order methods
in data science and machine learning,
the past few years have witnessed a growing interest in the development of easily implementable
parameter-free first-order methods that have guaranteed fast rates of convergence.  
One notable line of research is to eliminate the incorporation of line search procedures in first-order methods to further speed up their convergence and reduce the cost per iteration.
Specifically, in some interesting works, \cite{chaudhuri2009parameter, streeter2012no, orabona2016coin, cutkosky2017online, cutkosky2018black, orabona2020tutorial, mhammedi2020lipschitz} studied the problem of parameter-free \revision{}{online} regret minimization that pertains to nonsmooth convex optimization, \cite{carmon2022making,defazio2023learning, ivgi2023dog} developed adaptive line-search free subgradient type methods for nonsmooth objectives \revision{}{in deterministic or stochastic settings}, while \cite{levy2017online, li2019convergence, malitsky2019adaptive, malitsky2023adaptive, khaled2023dowg,latafat2023adaptive,orabona2023normalized} established the convergence guarantees for this type of methods for \revision{}{deterministic} smooth objectives. However, all the aforementioned methods at most match the convergence rate of non-accelerated (sub)gradient descent in the worst case, thus failing to achieve the optimal rates of convergence for smooth and weakly smooth problems with $\nu \in (0,1]$. It is worth noting that there exist a few accelerated variants of adaptive gradient type methods with $\mathcal{O}(1/k^2)$ rate for smooth objectives \cite{levy2018online, ene2021adaptive}. However, these methods require the stepsize to either depend on the diameter of a bounded feasible region or be fine-tuned according to the unknown initial distance $\|x_0-x^*\|$.
To the best of our knowledge, no optimal first-order method currently exists that meets the following criteria for solving smooth optimization problems with an unknown Lipschitz constant \(L\): it offers simple subproblems, does not assume the feasible region to be bounded, and does not require line search procedures. \revision{}{As pointed out by Malitsky and Mishchenko~\cite{malitsky2019adaptive} and Orabona~\cite{orabona2023normalized}, this remains an open problem.} Furthermore, there is no such method that is uniformly optimal for smooth, weakly smooth, and nonsmooth convex optimization.


In this paper, we attempt to address the aforementioned unresolved issue in the development of optimal methods for convex optimization. Our major contributions are briefly summarized as follows.
    Firstly, we introduce a novel first-order algorithm, called  {\em Auto-Conditioned Fast Gradient Method (AC-FGM)}, which can achieve the optimal $\mathcal{O}(1/k^2)$ convergence rate for smooth convex optimization without knowing any problem parameters or resorting to any line search procedures. The algorithmic framework, being as simple as the AGD method, can completely adapt to the problem structure.
    The major novelty of AC-FGM lies in the management of two extrapolated sequences of search points, with one serving as the ``prox-centers'' and the other used for gradient computation, respectively. This design plays an essential role in getting rid of line search required by \cite{nesterov2015universal} 
    and removing the assumption of a bounded feasible region 
    used in \cite{Lan13-1}. Moreover, we propose a novel stepsize policy that is highly adaptive to the local smoothness level, which can lead to more efficient empirical performance compared with other existing adaptive methods.
    Secondly, we show that AC-FGM can be extended to solve convex problems with H\"{o}lder continuous gradients,  requiring no prior knowledge of problem parameters except for the desired accuracy of the solution. Our analysis demonstrates that AC-FGM is uniformly optimal for all smooth, weakly smooth, and nonsmooth objectives with $\nu 
\in [0,1]$. 
Thirdly, we confirm the theoretical advantages of AC-FGM through a series of numerical experiments. Our results show that AC-FGM outperforms the previously developed parameter-free methods, such as the adaptive gradient descent and Nesterov's FGM, across a wide range of testing problems.

It is worth noting that we focus on convex problems with the function optimality gap as the termination criterion for our algorithms in this paper. 
We do not consider strongly convex or nonconvex problems. However, Lan, Ouyang and Zhang in a concurrent paper~\cite{LanOuyangZhang2023} show that by using AC-FGM as 
a subroutine in their new algorithmic framework, one can achieve tight complexity bounds to minimize (projected) gradients for convex, strongly convex and nonconvex problems in a parameter-free manner. Therefore, the results reported in this work may have an impact in the broader area of parameter-free optimization methods.

The remaining part of the paper is organized as follows. In Section~\ref{sec_smooth}, we propose AC-FGM and prove its optimal convergence guarantees \revision{for solving smooth convex problems}{when $f$ is a smooth convex function}. In Section~\ref{sec_holder}, we extend AC-FGM to solve problems \revision{with H\"{o}lder continuous gradients}{where $f$ has H\"{o}lder continuous gradients} and establish its uniform optimality. 
Finally, in Section~\ref{sec_numerical}, we provide numerical experiments that demonstrate the empirical advantages of AC-FGM. 

\subsection{Notation \revision{}{and preliminaries}}
In this paper, we use the convention that $\tfrac{0}{0} = 0$, and $\tfrac{a}{0} = + \infty$ if $a > 0$. Unless stated otherwise, we let $\langle \cdot, \cdot\rangle$ denote the Euclidean inner product and $\|\cdot\|$ denote the corresponding Euclidean norm ($\ell_2$-norm). Given a vector $x \in \bbr^n$, we denote its $i$-th entry by $x^{(i)}$. Given a matrix $A$, we denote its $(i,j)$-th entry by $A_{i,j}$. We use $\|A\|$ to denote the spectral norm of matrix $A$. We let $\lambda_{\max}(A)$ and $\lambda_{\min}(A)$ denote the largest and smallest eigenvalue of a self-adjoint matrix $A$.

\revision{}{Throughout the paper, we assume the function $h$ is ``prox-friendly'', meaning that the proximal-mapping problem
\begin{align*}
\arg \min_{z \in X} \left\{\langle \xi, z \rangle +  h(z) + \tfrac{1}{2\eta}\|y - z\|^2\right\}, \quad \text{where $\eta > 0$, $\xi \in \bbr^n$, $y\in X$,}
\end{align*}
is easily computable either in closed form or by some efficient computational procedure.
}

\section{AC-FGM for smooth convex problems}\label{sec_smooth}
In this section, 
we first consider \revision{the case}{solving problem \eqref{main_prob}} when $f$ is smooth, satisfying condition~\eqref{eq:smooth_1}
or equivalently,
\begin{align}\label{eq:smooth_2}
    \tfrac{1}{2L} \|g(y) - g(x)\|^2 \leq f(y) - f(x) - \langle g(x), y-x\rangle \leq \tfrac{L}{2}\|y - x\|^2, \quad \forall x, y \in X.
\end{align}
Our goal is to present a new parameter-free AGD type optimal algorithm that does not 
require line search for solving this class of problems.

We start to introduce the basic algorithmic framework of the Auto-Conditioned Fast Gradient Method (AC-FGM). 
\begin{algorithm}[H]\caption{Auto-Conditioned Fast Gradient Method (AC-FGM)}\label{alg1}
	\begin{algorithmic}
		\State{\textbf{Input}: initial point $z_0 = y_0 = x_0$, nonnegative parameters $\beta_t \in (0, 1)$, $\eta_t \in \bbr_+$ and $\tau_t \in \bbr_+$.}
		\For{$t=1,\cdots, k$}
		\State{
  \begin{align}
  z_t &= \arg \min_{z \in X} \left\{\revision{\eta_t \langle g(x_{t-1}), z \rangle}{\eta_t[\langle g(x_{t-1}), z \rangle + h(z)]} + \tfrac{1}{2}\|y_{t-1} - z\|^2\right\}, \label{prox-mapping}\\
  y_t &= (1-\beta_t) y_{t-1} + \beta_t z_t, \label{prox-center}\\
  x_t &= (z_{t} + \tau_t x_{t-1})/({1 + \tau_t}). \label{output_series_2}
  \end{align}
  }
		\EndFor
	\end{algorithmic}
\end{algorithm}
As shown in Algorithm~\ref{alg1}, the basic algorithmic scheme is conceptually simple. It involves three intertwined sequences of search points, $\{x_t\}$, $\{y_t\}$, and $\{z_t\}$, where $\{x_t\}$ and $\{y_t\}$ are two weighted average sequences of $\{z_t\}$. The sequence $\{x_t\}$ is used for gradient computation and also represents the output solutions, while $\{y_t\}$ is the sequence of ``prox-centers'' of the prox-mapping problems in \eqref{prox-mapping}. 
The search point $z_t$ is updated according to \eqref{prox-mapping}
by minimizing \revision{}{the summation of $h$ and} a linear approximation of $f$ at $x_{t-1}$, but not moving too far away from the prox-center $y_{t-1}$ with $\eta_t$
as the stepsize. \revision{}{Notice that in \eqref{prox-center}-\eqref{output_series_2}, we present the two weighted average sequences, $\{x_t\}$ and $\{y_t\}$, in different formats to simplify the algorithm analysis. Additionally, the parameter $\tau_t$ in \eqref{output_series_2} has a specific meaning when expressing AC-FGM in its primal-dual form (see \eqref{game_step_3} below)}.


To illustrate some basic ideas behind the design of AC-FGM,
let us compare it with a few other well-known first-order algorithms in the literature.
\begin{itemize}
    \item \emph{Nesterov's accelerated gradient method}: Nesterov's accelerated gradient method  can be expressed using three sequences of search points (see, e.g., Section 3.3 of \cite{LanBook2020}):
    \begin{align*}
    \revision{y_t}{x_t} &= (1-q_t) \revision{x_{t-1}}{y_{t-1}} + q_t z_{t-1},\\
    z_t &=\arg \min_{z \in X} \left\{\revision{\eta_t \langle g(y_{t}), z \rangle}{\eta_t [\langle g(x_{t}), z \rangle + h(z)]} + \tfrac{1}{2}\|z_{t-1} - z\|^2\right\},\\
    \revision{x_t}{y_t} &= (1-\alpha_t) \revision{x_{t-1}}{y_{t-1}} + \alpha_t z_t.
    \end{align*}
    \revision{Notably, the gradients are computed at $\{y_t\}$ rather than the output solutions $\{x_t\}$.}{Notably, the gradients are computed at $\{x_t\}$, which is also a weighted average sequence of $\{z_t\}$. However, the method uses} $\{z_t\}$ as the prox-centers, while AC-FGM uses the sequence $\{y_t\}$, \revision{}{another dynamically weighted average sequence} of $\{z_t\}$, as the prox-centers. 
    \item \emph{Gradient extrapolation method (GEM)}: GEM \cite{lan2018random, ilandarideva2023accelerated} is another method related to AC-FGM, which carries a particular duality relationship with Nesterov's accelerated gradient method and also achieves the optimal rate of convergence for smooth convex optimization. The algorithm follows the steps
    \begin{align*}
    \tilde g_t &= g(x_{t-1}) + \alpha_t(g(x_{t-1}) - g(x_{t-2})) ,\\
    z_t &=\arg \min_{z \in X} \left\{\revision{\eta_t \langle \tilde g_t, z \rangle}{\eta_t [\langle \tilde g_t, z \rangle + h(z)]} + \tfrac{1}{2}\|z_{t-1} - z\|^2\right\},\\
    x_t &= (z_{t} + \tau_t x_{t-1})/({1 + \tau_t}).
    \end{align*}
    Notice that the last step of GEM is the same as AC-FGM, which generates search points for both gradient computation and solution output. The difference between these two methods exists in that GEM utilizes extrapolation in the dual space as in the first step. It is noteworthy that this dual extrapolation idea is also used in the operator extrapolation (OE) method \cite{kotsalis2022simple, kotsalis2022simple1} for solving variational inequalities (VIs).
    \item \emph{Golden ratio method for variational inequalities}: The idea of using an extrapolated sequence $\{y_t\}$ as the prox-centers was introduced in the Golden ratio method \cite{malitsky2020golden, alacaoglu2023beyond} for VIs:
    \begin{align*}
    z_t &= \arg \min_{z \in X} \left\{\eta_t \langle F(z_{t-1}), z \rangle + \tfrac{1}{2}\|y_{t-1} - z\|^2\right\}, \\
  y_t &= (1-\beta_t) y_{t-1} + \beta_t z_t,
    \end{align*}
    where $F(\cdot)$ is the VI operator. It is shown in \cite{malitsky2020golden, alacaoglu2023beyond} that the Golden ratio method is a fully adaptive approach for smooth VIs without requiring the line search procedures. However, one can only achieve an $\mathcal{O}(1/k)$ rate for smooth VIs, which differs from the optimal convergence rate for smooth convex optimization. Therefore, we construct another sequence of search points $\{x_t\}$ as in GEM and shift the gradient estimations to accelerate the convergence rate to $\mathcal{O}(1/k^2)$. In summary, the relationship between AC-FGM and the golden ratio method resembles the relationship between GEM and the OE method in that the former approaches (i.e., AC-FGM and GEM) accelerate the rate of convergence
    for the latter methods (i.e., Golden ratio method and OE)
    applied to smooth convex optimization.
\end{itemize}

\revision{}{To further illustrate the intuition behind AC-FGM, we consider the following saddle-point reformulation of \eqref{main_prob}, i.e., 
\begin{align*}
\Psi^* := \min_{x\in X}\left\{ \max_{\varsigma \in \bbr^n} \left\{\langle x, \varsigma\rangle - J_f(\varsigma) \right\} + h(x)\right\}.
\end{align*}
Here we assume for similicity $f$ is well-defined in $\bbr^n$ and $J_f: \bbr^n \rightarrow \bbr$ is the Fenchel conjugate function of $f$. Due to the $L$-smoothness of $f$, it is clear that $J_f$ is strongly convex with modulus $1/L$. Therefore, we can define the Bregman divergence associated with $J_f$ according to
\begin{align*}
D_f(\tilde \varsigma, \varsigma)= J_f(\varsigma)- [J_f(\tilde \varsigma) + \langle J_f'(\tilde \varsigma), \varsigma - \tilde \varsigma \rangle],\;\;\varsigma,\tilde \varsigma\in \bbr^n, 
\end{align*}
where $J_f'(\tilde \varsigma) \in \partial J_f(\tilde \varsigma)$. Then we consider the (generalized) prox-mapping associated with this Bregman divergence,
\begin{align}\label{gendprox}
\arg\max_{\varsigma \in \bbr^n} \Big\{
\langle  \tilde x, \varsigma \rangle - J_f(\varsigma) - \tau D_f(\tilde \varsigma, \varsigma)\Big\},\;\;  \tilde x, \tilde \varsigma\in \bbr^n.
\end{align}
As shown in Lemma 1 of \cite{lan2015optimal} (see also Lemma 3.6 of \cite{LanBook2020}),
the maximizer of \rf{gendprox} is the gradient of $f$ at a certain  $\bar x\in \bbr^n$, specifically,
\[
\arg\max_{\varsigma \in \bbr^n} \Big\{
\langle \tilde x, \varsigma \rangle - J_f(\varsigma) - \tau D_f(\tilde \varsigma, \varsigma)\Big\} = g(\bar x), ~~\text{where $\bar x = \tfrac{\tilde x + \tau J_f'(\tilde \varsigma)}{1 + \tau}$}.
\]
Using the above observation, we can rewrite AC-FGM in Algorithm~\ref{alg1}
in a primal-dual form. Starting from $x_0 = y_0 = z_0$ and $\varsigma_0=g(x_0)$, AC-FGM iteratively updates $(z_t,y_t,\varsigma_t)$ by
\begin{align}
z_t &= \arg \min_{z \in X} \left\{\revision{\eta_t \langle \varsigma_{t-1}, z \rangle}{\eta_t[\langle \varsigma_{t-1}, z \rangle + h(z)]} + \tfrac{1}{2}\|y_{t-1} - z\|^2\right\}, \label{game_step_1}\\
  y_t &= (1-\beta_t) y_{t-1} + \beta_t z_t, \label{game_step_2}\\
\varsigma_{t} &= \arg \max_{\varsigma\in \bbr^n} \left\{\langle z_t, \varsigma \rangle - J_f(\varsigma) - \tau_t D_f(\varsigma_{t-1}, \varsigma)\right\}. \label{game_step_3}
\end{align}
With a selection of $J_f'(\varsigma_{t-1}) = x_{t-1}$ in $D_f(\varsigma_{t-1}, \varsigma)$, 
we have $\varsigma_t = g(x_t)$ with $x_t = (z_t + \tau_t x_{t-1}) / (1 + \tau_t)$ which is the definition of $x_t$ in \eqref{output_series_2}. From a game point of view, the above expression can be intuitively interpreted as follows: In \eqref{game_step_1}-\eqref{game_step_2}, the ``smart'' primal player determines the next action $z_t$ based on the last action of the dual player $\varsigma_{t-1}$, but not too far away from the weighted average $y_{t-1}$ of the past decisions; in \eqref{game_step_3}, the dual player determines the next action $\varsigma_t$ based on the latest action of the primal player $z_t$, but not too far away from the last decision $\varsigma_{t-1}$.

Similarly, one can show that Nesterov's accelerated gradient method discussed above
has the following primal-dual update (see \cite[Section 2.2]{lan2015optimal} and \cite[Section 3.4]{LanBook2020}):
\begin{align}
\tilde z_{t}&= z_{t-1} + \lambda_t (z_{t-1}-z_{t-2}), \label{def_txt}\\
\varsigma_{t} &= \arg\max_{\varsigma\in \bbr^n} \left\{\langle \tilde z_t, \varsigma \rangle - J_f(\varsigma) - \tau_t D_f(\varsigma_{t-1}, \varsigma)\right\},
\label{def_yt}\\
z_{t} &=  \arg\min_{z \in X}\left\{\eta_t[\langle \varsigma_t,x\rangle + h(z)] +\tfrac{1}{2}\|z_{t-1}-z\|^2 \right\},
\label{def_xt_GEM}
 \end{align}
 with $\tau_t = (1-\alpha_t)/\alpha_t$ and $\lambda_t=\alpha_{t-1}(1-\alpha_t)/\alpha_t$ when $\alpha_t = q_t$. The principle difference between this method and AC-FGM exists in that here, the dual player is consider to be ``smart'' by determining the next action $\varsigma_t$ based on $\tilde z_t$, the past decision of the primal player with a ``correction'' using extrapolation. Moreover, the GEM algorithm can also be expressed in a similar primal-dual form, by switching the role of primal and dual variables in \eqref{def_txt}-\eqref{def_xt_GEM}.}


In the remaining part of this section, we establish the convergence guarantees for AC-FGM and specify the selection of algorithmic parameters $\eta_t$, $\beta_t$ and $\tau_t$. To ensure that AC-FGM operates in a \emph{problem-parameter-free} manner, we suppose that we do not possess prior knowledge of the smoothness parameter $L$ in conditions~\eqref{eq:smooth_1} and \eqref{eq:smooth_2}. Instead, we estimate the local smoothness level in each iteration. In the first iteration ($t=1$), we select an initial stepsize $\eta_1$ such that $x_1 \neq x_0$ and calculate $L_1$ as
    \begin{align}\label{def_L_1}
L_1 := \frac{\|g(x_1) - g(x_0)\|}{\|x_1-x_0\|}.
\end{align}
For subsequent iterations ($t\geq 2$), we calculate
\begin{align}\label{def_L_t}
L_t:=  \begin{cases}
      ~0, &\text{if}~~f(x_{t-1}) - f(x_t) - \langle g(x_t), x_{t-1} - x_t\rangle = 0,\\
       \frac{\|g(x_t) - g(x_{t-1})\|\revision{_*}{}^2}{2[f(x_{t-1}) - f(x_t) - \langle g(x_t), x_{t-1} - x_t\rangle]}, & \text{if}~~f(x_{t-1}) - f(x_t) - \langle g(x_t), x_{t-1} - x_t\rangle > 0.
      \end{cases}
\end{align}
Notice that this approach allows us to calculate $L_t$ whenever a new search point $x_t$ is obtained in each iteration. Because the search point $x_t$ is calculated using the parameters $\eta_t$ and $\tau_t$, to obtain the line-search free property, the parameters $\eta_t$ and $\tau_t$ should be chosen according to the local estimations $L_1, ..., L_{t-1}$ but not $L_t$.

The following two results (i.e., Proposition~\ref{prop_1} and Theorem~\ref{main_theorem}) characterize the convergence behavior and the line-search free property of Algorithm~\ref{alg1}. It should be noted that our analysis differs significantly from the existing analyses of the golden ratio method and gradient extrapolation method, thus appears to be nontrivial. Proposition~\ref{prop_1} below provides an important recursive relationship for the iterates of the AC-FGM algorithm. 
\begin{proposition}\label{prop_1}
Assume the parameters $\{\tau_t\}$, $\{\eta_t\}$ and $\{\beta_t\}$ satisfy
\begin{align}
\tau_1 &= 0, ~\beta_1 = 0,~\beta_t = \beta > 0,~t\geq 2 \label{cond_3}\\
    \eta_2 &\leq \min\left\{\revision{2(1-\beta)\eta_1}{(1-\beta)\eta_1}, \revision{\tfrac{\beta}{2L_1}}{\tfrac{1}{4L_1}} \right\},\label{cond_1}\\ 
    \eta_t &\leq \min\left\{2(1-\beta)^2 \eta_{t-1}, \revision{\tfrac{\beta \tau_{t-1}}{4L_{t-1}}}{\tfrac{ \tau_{t-1}}{4L_{t-1}}}, \tfrac{\tau_{t-2}+1}{\tau_{t-1}}\cdot\eta_{t-1}\right\}, ~t \geq 3
    \label{cond_4}
\end{align}
where $L_1$ and $L_t, ~t\geq 2$ are defined in \eqref{def_L_1} and \eqref{def_L_t}, respectively. We have for any $z \in X$,
\begin{align}\label{eq9_3}
&\tsum_{t=1}^k \eta_{t+1}\left[\tau_t f(x_{t})+ \langle g(x_t), x_t - z\rangle - \tau_t f(x_{t-1})\revision{}{+ h(z_{t}) - h(z)} \right] + \tfrac{\eta_2}{2\eta_{1}} \left[\|z_1 - y_1\|^2 + \|z_2 - z_1\|^2 \right]\nn\\
& \leq \tfrac{1}{2\beta}\|y_1 - z\|^2 - \tfrac{1}{2\beta}\|y_{k+1}-z\|^2 + \tsum_{t=2}^{k+1}\Delta_{t},
\end{align}
where \revision{}{for $t\geq 2$,}
\begin{align*}
\color{black}
\Delta_t :=  
       \eta_t \langle g(x_{t-1}) - g(x_{t-2}), z_{t-1} - z_t \rangle - \tfrac{\eta_t \tau_{t-1}\|g(x_{t-1}) - g(x_{t-2})\|^2}{2L_{t-1}} - \tfrac{\|z_t - y_{t-1}\|^2}{2}.
\end{align*}
\end{proposition}
\begin{proof}
First, by the optimality condition of step \eqref{prox-mapping} \revision{}{and the convexity of $h$}, we have for $t \geq 1$,
\begin{align}\label{eq1}
\langle \eta_t g(x_{t-1})+ z_t - y_{t-1}, z - z_t\rangle \geq \revision{0}{\eta_t[h(z_t) - h(z)]}, \quad \forall z \in X,
\end{align}
and consequently for $t \geq 2$,
\begin{align}\label{eq2}
\langle \eta_{t-1} g(x_{t-2})+ z_{t-1} - y_{t-2}, z_t - z_{t-1}\rangle \geq \revision{0}{\eta_{t-1}[h(z_{t-1}) - h(z_t)]}.
\end{align}
By \eqref{prox-center}, we have the relationship $z_{t-1} - y_{t-2} = \tfrac{1}{1-\beta_{t-1}} (z_{t-1}-y_{t-1})$, so we can rewrite \eqref{eq2} as
\begin{align}\label{eq3}
\langle \eta_t g(x_{t-2}) + \tfrac{\eta_t}{\eta_{t-1}(1-\beta_{t-1})}(z_{t-1} - y_{t-1}), z_t - z_{t-1}\rangle \geq \revision{0}{\eta_t[h(z_{t-1}) - h(z_t)]}. 
\end{align}
Summing up \eqref{eq1} and \eqref{eq3} and rearranging the terms give us for $t \geq 2$,
\begin{align*}
&\eta_t \langle g(x_{t-1}), z - z_{t-1}\rangle + \eta_t \langle g(x_{t-1}) - g(x_{t-2}), z_{t-1}-z_t \rangle\\ 
&+ \langle z_t - y_{t-1}, z - z_t\rangle + \tfrac{\eta_t}{\eta_{t-1}(1-\beta_{t-1})} \langle z_{t-1}-y_{t-1}, z_t - z_{t-1} \rangle \geq \revision{0}{\eta_t[h(z_{t-1}) - h(z)]}.
\end{align*}
By utilizing the fact that $2\langle x - y, z - x\rangle = \|y-z\|^2 - \|x-y\|^2 - \|x-z\|^2$ on the last two terms of the above inequality, we obtain
\begin{align}\label{eq4}
&\eta_t \langle g(x_{t-1}), z - z_{t-1}\rangle + \eta_t \langle g(x_{t-1}) - g(x_{t-2}), z_{t-1}-z_t \rangle\nn\\ 
&+ \tfrac{1}{2}\|y_{t-1}-z\|^2 - \tfrac{1}{2}\|z_{t}-y_{t-1}\|^2 -\tfrac{1}{2}\|z_t - z\|^2\nn\\
&+ \tfrac{\eta_t}{2\eta_{t-1}(1-\beta_{t-1})} \left[\|z_t - y _{t-1}\|^2 - \|z_{t-1} - y_{t-1}\|^2 - \|z_t - z_{t-1}\|^2\right] \geq \revision{0}{\eta_t[h(z_{t-1}) - h(z)]},
\end{align}
Meanwhile, \eqref{prox-center} also indicates that
\begin{align}\label{eq5}
    \|z_t - z\|^2 &= \|\tfrac{1}{\beta_t} (y_t - z) - \tfrac{1-\beta_t}{\beta_t}(y_{t-1} - z)\|^2\nn\\
    &\overset{(i)}=\tfrac{1}{\beta_t}\|y_t -z\|^2 - \tfrac{1-\beta_t}{\beta_t}\|y_{t-1} - z\|^2 + \tfrac{1}{\beta_t} \cdot \tfrac{1-\beta_t}{\beta_t} \|y_t - y_{t-1}\|^2\nn\\
    &= \tfrac{1}{\beta_t}\|y_t -z\|^2 - \tfrac{1-\beta_t}{\beta_t}\|y_{t-1} - z\|^2 + (1-\beta_t) \|z_t - y_{t-1}\|^2,
\end{align}
where step (i) follows from the fact that $\|\alpha a + (1-\alpha) b\|^2 = \alpha \|a\|^2 + (1-\alpha)\|b\|^2 - \alpha (1-\alpha)\|a-b\|^2,~\forall \revision{a}{\alpha} \in \bbr$. By combining \revision{Ineqs.}{Ineq.} \eqref{eq4} and \revision{}{Eq.} \eqref{eq5} and rearranging the terms, we obtain
\begin{align}\label{eq6}
&\eta_t \langle g(x_{t-1}), z_{t-1} - z\rangle+ \revision{}{\eta_t[h(z_{t-1}) - h(z)]} + \tfrac{1}{2\beta_t} \|y_t - z\|^2 + \tfrac{\eta_t}{2\eta_{t-1}(1-\beta_{t-1})} \left[\|z_{t-1} - y_{t-1}\|^2 + \|z_t - z_{t-1}\|^2 \right]\nn\\
&\leq \tfrac{1}{2\beta_t}\|y_{t-1} - z\|^2 +  \eta_t \langle g(x_{t-1}) - g(x_{t-2}), z_{t-1} - z_t \rangle- \left[ \tfrac{1}{2} + \tfrac{1 - \beta_t}{2} - \tfrac{\eta_t}{2 \eta_{t-1}(1-\beta_{t-1})} \right]\|z_t - y_{t-1}\|^2.
\end{align}
\revision{}{When $t \geq 3$, by utilizing the fact that $\|z_{t-1} - y_{t-1}\|^2 + \|z_t - z_{t-1}\|^2 \geq \tfrac{1}{2}\|z_t - y_{t-1}\|^2$ and rearranging the terms in Ineq.~\eqref{eq6}, we obtain
\begin{align}\label{eq6_prime}
&\eta_t \langle g(x_{t-1}), z_{t-1} - z\rangle+ \revision{}{\eta_t[h(z_{t-1}) - h(z)]} + \tfrac{1}{2\beta_t} \|y_t - z\|^2 \nn\\
&\leq \tfrac{1}{2\beta_t}\|y_{t-1} - z\|^2 +  \eta_t \langle g(x_{t-1}) - g(x_{t-2}), z_{t-1} - z_t \rangle- \left[ \tfrac{1}{2} + \tfrac{1 - \beta_t}{2} - \tfrac{\eta_t}{4 \eta_{t-1}(1-\beta_{t-1})} \right]\|z_t - y_{t-1}\|^2\nn\\
&\overset{(i)}\leq \tfrac{1}{2\beta_t}\|y_{t-1} - z\|^2 +  \eta_t \langle g(x_{t-1}) - g(x_{t-2}), z_{t-1} - z_t \rangle- \tfrac{1}{2}\|z_t - y_{t-1}\|^2,
\end{align}
where step (i) follows from $\eta_t \leq 2(1-\beta)^2 \eta_{t-1}$ in \eqref{cond_4}. 
On the other hand, recalling the definition of $L_t$ in \eqref{def_L_t} and the convention that $0/0=0$, we can write}
\begin{align}\label{eq7}
    f(x_{t-1}) - f(x_t) - \langle g(x_t), x_{t-1} - x_t\rangle = \tfrac{1}{2L_t}\|g(x_t) - g(x_{t-1})\|^2, ~\forall t \geq 2.
\end{align}
Then for any $z \in X$ and $t \geq 3$, using the above equality and the fact that $x_{t-1} - z = z_{t-1} - z + \tau_{t-1}(x_{t-2} - x_{t-1})$ due to Eq.~\eqref{output_series_2}, we have
\begin{align}\label{eq8}
\tau_{t-1} f(x_{t-1}) + \langle g(x_{t-1}), x_{t-1} - z\rangle
&= \tau_{t-1} [f(x_{t-1}) + \langle g(x_{t-1}), x_{t-2} - x_{t-1}\rangle] + \langle g(x_{t-1}), z_{t-1} - z\rangle\nn\\
& =\tau_{t-1} [f(x_{t-2}) - \tfrac{\|g(x_{t-1}) - g(x_{t-2})\|^2}{2L_{t-1}}] + \langle g(x_{t-1}), z_{t-1} - z\rangle.
\end{align}
Combining \revision{Ineqs.}{Ineq.~\eqref{eq6_prime}} and \revision{}{Eq.} \eqref{eq8} and rearranging the terms, we obtain that for $t \geq 3$
\begin{align}\label{eq9}
&\eta_t\left[\tau_{t-1} f(x_{t-1}) + \langle g(x_{t-1}), x_{t-1} - z\rangle - \tau_{t-1}f(x_{t-2}) + \revision{}{h(z_{t-1}) - h(z)}\right]  \nn\\
&\leq \tfrac{1}{2\beta_t}\|y_{t-1} - z\|^2 - \tfrac{1}{2\beta_t}\|y_{t}- z\|^2 + \eta_t \langle g(x_{t-1}) - g(x_{t-2}), z_{t-1}- z_t \rangle\nn\\
&\quad  - \tfrac{\eta_t \tau_{t-1}}{2L_{t-1}}\|g(x_{t-1}) - g(x_{t-2})\|^2 - \tfrac{\revision{\beta_t}{1}}{2} \|z_t - y_{t-1}\|^2.
\end{align}
Meanwhile, notice that $\beta_1 = 0$ and $x_1 = z_1$ due to $\tau_1=0$, we can rewrite \eqref{eq6} for the case $t=2$ as
\begin{align}\label{eq9_2}
&\eta_2 [\langle g(x_{1}), x_1 - z\rangle + \revision{}{h(z_{1}) - h(z)} ]    + \tfrac{\eta_2}{2\eta_{1}} \left[\|z_1 - y_1\|^2 + \|z_2 - z_1\|^2 \right]\nn\\
&\leq \tfrac{1}{2\beta_2}\|y_{1} - z\|^2 - \tfrac{1}{2\beta_2} \|y_2 - z\|^2+  \eta_2 \langle g(x_{1}) - g(x_{0}), z_{1} - z_2 \rangle \revision{- \tfrac{\beta_2}{2}\|z_2 - y_{1}\|^2.}{- \left[ \tfrac{1}{2} + \tfrac{1 - \beta_2}{2} - \tfrac{\eta_2}{2 \eta_1} \right]\|z_2 - y_{1}\|^2}\nn\\
&\overset{(ii)}\leq \revision{- \tfrac{\beta_2}{2}\|z_2 - y_{1}\|^2.}{\tfrac{1}{2\beta_2}\|y_{1} - z\|^2 - \tfrac{1}{2\beta_2} \|y_2 - z\|^2+  \eta_2 \langle g(x_{1}) - g(x_{0}), z_{1} - z_2 \rangle -  \tfrac{1}{2}\|z_2 - y_{1}\|^2},
\end{align}
\revision{}{where step (ii) follows from the condition $\eta_2 \leq (1-\beta)\eta_1$ in \eqref{cond_1}.}
By taking the summation of Ineq.~\eqref{eq9_2} and the telescope sum of Ineq.~\eqref{eq9} for $t=3,..., k+1$, and noting that $\tau_1=0$ and $\beta_t = \beta$ for $t \geq 2$, we obtain
\begin{align*}
&\tsum_{t=1}^k \eta_{t+1}\left[\tau_t f(x_{t})+ \langle g(x_t), x_t - z\rangle - \tau_t f(x_{t-1}) \revision{}{+ h(z_{t}) - h(z)}\right] + \tfrac{\eta_2}{2\eta_{1}} \left[\|z_1 - y_1\|^2 + \|z_2 - z_1\|^2 \right]\nn\\
& \leq \tfrac{1}{2\beta}\|y_1 - z\|^2 - \tfrac{1}{2\beta}\|y_{k+1}-z\|^2 + \tsum_{t=2}^{k+1}\Delta_{t},
\end{align*}
where $\Delta_t$ is defined in the statement of this proposition, and we complete the proof.
\end{proof}
\vgap

Utilizing Proposition~\ref{prop_1}, we further bound the extra terms $\tsum_{t=2}^{k+1}\Delta_{t}$ and provide the convergence guarantees for AC-FGM in terms of the objective values in the following theorem.
\begin{theorem}\label{main_theorem}
 Assume the parameters $\{\tau_t\}$, $\{\eta_t\}$ and $\{\beta_t\}$ satisfy conditions~\eqref{cond_3}, \eqref{cond_1} and \eqref{cond_4}. Then the sequences $\{z_t\}, \{y_t\}$ and $\{x_t\}$ generated by Algorithm~\ref{alg1} are bounded. Moreover, we have
\begin{align}\label{main_thm_eq1}
\revision{f}{\Psi}(x_{k}) - \revision{f}{\Psi}(x^*) \leq \tfrac{1}{\revision{\tau_k \eta_{k+1}}{(\tau_k+1) \eta_{k+1}}}\cdot \left[\tfrac{1}{2\beta} \|z_0 - x^*\|^2 + \left(\revision{\tfrac{3\eta_2 L_1}{2}}{\tfrac{5\eta_2 L_1}{4}} - \tfrac{\eta_2}{2\eta_1} \right)\|z_1 - z_0\|^2 \right],
\end{align}
and
\begin{align}\label{main_thm_eq2}
\revision{f}{\Psi}(\bar x_{k}) - \revision{f}{\Psi}(x^*) \leq \tfrac{1}{\sum_{t=2}^{k+1}\eta_t }\cdot \left[\tfrac{1}{2\beta} \|z_0 - x^*\|^2 + \left(\revision{\tfrac{3\eta_2 L_1}{2}}{\tfrac{5\eta_2 L_1}{4}} - \tfrac{\eta_2}{2\eta_1} \right)\|z_1 - z_0\|^2 \right],
\end{align}
where 
\begin{align}\label{def_bar_x_k}
\bar x_{k} := 
 \tfrac{\sum_{t=1}^{k-1}[(\tau_{t}+1)\eta_{t+1} - \tau_{t+1} \eta_{t+2}]x_{t}+ (\tau_{k}+1)\eta_{k+1} x_{k} }{\sum_{t=2}^{k+1}\eta_t }.
 \end{align}

\end{theorem}
\begin{proof}
First, notice that Ineq.~\eqref{eq9_3} in Proposition~\ref{prop_1} holds. Now we work on the upper bounds of $\Delta_t$. 
For $t \geq 3$, we can bound $\Delta_t$ as
\begin{align}\label{eq10}
\Delta_t 
 &\overset{(i)}\leq  \eta_t \|g(x_{t-1}) - g(x_{t-2})\|\|z_{t-1} - z_t\| - \tfrac{\eta_t \tau_{t-1}\|g(x_{t-1}) - g(x_{t-2})\|^2}{2L_{t-1}} - \tfrac{\revision{\beta}{1}}{2} \|z_t - y_{t-1}\|^2 \nn\\
 &\overset{(ii)}\leq  \tfrac{\eta_t L_{t-1}}{2 \tau_{t-1}} \|z_{t-1} - z_t\|^2 - \tfrac{\revision{\beta}{1}}{2} \|z_t - y_{t-1}\|^2\nn\\
 &\overset{(iii)}\leq  \tfrac{\eta_t L_{t-1}(1-\beta)^2}{ \tau_{t-1}} \|z_{t-1} - y_{t-2}\|^2 - \left(\tfrac{\revision{\beta}{1}}{2} - \tfrac{\eta_t L_{t-1}}{\tau_{t-1}} \right) \|z_t - y_{t-1}\|^2\nn\\
 &\overset{(iv)}\leq  \tfrac{\revision{\beta}{1}}{4} \|z_{t-1} - y_{t-2}\|^2 - \tfrac{\revision{\beta}{1}}{4}  \|z_t - y_{t-1}\|^2,
\end{align}
Here, step (i) follows from Cauchy-Schwarz inequality, step (ii) follows from Young's inequality, step (iii) follows from the fact that 
\begin{align*}
\|z_{t-1}- z_t\|^2 &= \|z_{t-1} - y_{t-1}+ y_{t-1} - z_t\|^2 = \|(1-\beta)(z_{t-1} - y_{t-2}) + y_{t-1} - z_t\|^2\\
&\leq 2(1-\beta)^2\|z_{t-1} - y_{t-2}\|^2 + 2\|z_t - y_{t-1}\|^2,
\end{align*}
and step (iv) follows from condition $\eta_t \leq \tfrac{\revision{\beta}{}\tau_{t-1}}{4L_{t-1}}$ in \eqref{cond_4}.
Meanwhile, we can bound $\Delta_2$ by
\begin{align}\label{eq10_2}
\Delta_2 
& \overset{(i)}\leq \eta_2 L_1 \|x_1 - x_0\|\|z_1 - z_2\| - \tfrac{\revision{\beta}{1}}{2}\|z_2 - y_1\|^2\nn\\
&\overset{(ii)}\leq \eta_2 L_1 \|z_1 - y_0\| \left( \|z_1 - y_0\|+ \|z_2 - y_0\| \right) - \tfrac{\revision{\beta}{1}}{2}\|z_2 - y_0\|^2\nn\\
&\overset{(iii)}\leq \revision{\tfrac{3\eta_2 L_1}{2}}{\tfrac{5\eta_2 L_1}{4}}\|z_1 - z_0\|^2 +\left(\revision{\tfrac{\eta_2 L_1}{2}}{\eta_2 L_1} - \tfrac{\revision{\beta}{1}}{2}\right)\| z_2 - y_0\|^2\nn\\
&\overset{(iv)}\leq \revision{\tfrac{3\eta_2 L_1}{2}}{\tfrac{5\eta_2 L_1}{4}}\|z_1 - z_0\|^2 - \tfrac{\revision{\beta}{1}}{4}\|z_2 - y_1\|^2,
\end{align}
where step (i) follows from Cauchy-Schwarz inequality and the definition of $L_1$; step (ii) follows from triangle inequality and $x_1 = z_1$, $x_0 = y_0 = y_1=z_0$ due to $\tau_1 = 0$, $\beta_1 = 0$; step (iii) follows from Young's inequality; step (iv) follows from the condition $\eta_2 \leq \revision{\tfrac{\beta}{2L_1}}{\tfrac{1}{4L_1}}$. 

By substituting the bounds of \eqref{eq10} and \eqref{eq10_2} in Ineq.~\eqref{eq9_3}, we obtain
\begin{align}\label{eq11}
&\tsum_{t=1}^k \eta_{t+1}\left[\tau_t f(x_{t})+ \langle g(x_t), x_t - z\rangle - \tau_t f(x_{t-1}) \revision{}{+ h(z_{t}) - h(z)}\right] \nn\\
& \leq \tfrac{1}{2\beta}\|y_1 - z\|^2 - \tfrac{1}{2\beta}\|y_{k+1}-z\|^2 + \left(\revision{\tfrac{3\eta_2 L_1}{2}}{\tfrac{5\eta_2 L_1}{4}} - \tfrac{\eta_2}{2\eta_1}\right)\|z_1-z_0\|^2 - \tfrac{\revision{\beta}{1}}{4}\|z_{k+1}-y_k\|^2,
\end{align}
Set $z=x^*$ in \eqref{eq11}. Then, \revision{relationship}{by utilizing the relationships} $\langle g(x_t), x_t - x^*\rangle \geq f(x_t) - f(x^*)$ due to the convexity of $f$, \revision{}{and 
\begin{align}\label{label_convex_h}
\eta_{t+1} (\tau_{t}+1) [h(x_{t}) - h(x^*)] \leq \eta_{t+1}\tau_t [h(x_{t-1})-h(x^*)] + \eta_{t+1}[h(z_{t})-h(x^*)]
\end{align}
due to the convexity of $h$ and step \eqref{output_series_2}, and recalling the definition $\Psi(x)=f(x)+h(x)$},
we obtain
\begin{align}\label{eq12}
&\tsum_{t=1}^{k-1}[(\tau_{t}+1)\eta_{t+1} - \tau_{t+1} \eta_{t+2}][\revision{f}{\Psi}(x_{t})-\revision{f}{\Psi}(x^*)] + (\tau_{k}+1)\eta_{k+1} [\revision{f}{\Psi}(x_{k})-\revision{f}{\Psi}(x^*)]\nn \\
&\leq \tau_1 \eta_2 [\revision{f}{\Psi}(x_0) - \revision{f}{\Psi}(x^*)] + \tfrac{1}{2\beta} \|y_1 - x^*\|^2 +\left(\revision{\tfrac{3\eta_2 L_1}{2}}{\tfrac{5\eta_2 L_1}{4}} - \tfrac{\eta_2}{2\eta_1}\right)\|z_1-z_0\|^2 - \tfrac{1}{2\beta}\|y_{k+1}- x^*\|^2 - \tfrac{\revision{\beta}{1}}{4}\|z_{k+1}-y_k\|^2\nn\\
& \overset{(i)}= \tfrac{1}{2\beta} \|z_0 - x^*\|^2 + \left(\revision{\tfrac{3\eta_2 L_1}{2}}{\tfrac{5\eta_2 L_1}{4}} - \tfrac{\eta_2}{2\eta_1} \right)\|z_1 - z_0\|^2 - \tfrac{1}{2\beta}\|y_{k+1}- x^*\|^2 - \tfrac{\revision{\beta}{1}}{4}\|z_{k+1}-y_k\|^2,
\end{align}
where step (i) follows from the condition that $\tau_1 = 0$ \revision{}{and $y_1 = z_0$ due to $\beta_1 = 0$}.
Then by invoking the condition $\eta_t \leq \frac{\tau_{t-2}+1}{\tau_{t-1}}\cdot\eta_{t-1}$ in \eqref{cond_4}, we obtain \eqref{main_thm_eq1}.  By further invoking the definition of $\bar x_k$, we arrive at
\begin{align*}
    \revision{f}{\Psi}(\bar x_{k}) - \revision{f}{\Psi}(x^*) \leq  \tfrac{1}{\sum_{t=2}^{k+1}\eta_t }\cdot \left[\tfrac{1}{2\beta} \|z_0 - x^*\|^2 + \left(\revision{\tfrac{3\eta_2 L_1}{2}}{\tfrac{5\eta_2 L_1}{4}} - \tfrac{\eta_2}{2\eta_1} \right)\|z_1 - z_0\|^2 \right],
\end{align*}
which completes the proof of \eqref{main_thm_eq2}. Moreover, recalling  Ineq.~\eqref{eq5},
we have
\begin{align*}
\tfrac{1}{2\beta}\|y_{k+1}- x^*\|^2 + \tfrac{\revision{\beta}{1}}{4}\|z_{k+1}-y_k\|^2 &\geq \min\big\{\tfrac{\revision{\beta}{1}}{4(1-\beta)}, \tfrac{1}{2}\big\}\cdot \big[\tfrac{1}{\beta}\|y_{k+1}-x^*\|^2 + (1-\beta)\|z_{k+1}-y_k\|^2 \big]\\
&= \min\big\{\tfrac{\revision{\beta}{1}}{4(1-\beta)}, \tfrac{1}{2}\big\}\cdot \left[\|z_{k+1} - x^*\|^2 + \tfrac{1-\beta}{\beta}\|y_k - x^*\|^2\right],
\end{align*}
which, together with Ineq.~\eqref{eq12}, indicates that 
$$\|z_{k+1} - x^*\|^2 \leq \left(\min\big\{\tfrac{\revision{\beta}{1}}{4(1-\beta)}, \tfrac{1}{2}\big\}\right)^{-1}\cdot \left[\tfrac{1}{2\beta} \|z_0 - x^*\|^2 + \left(\revision{\tfrac{3\eta_2 L_1}{2}}{\tfrac{5\eta_2 L_1}{4}} - \tfrac{\eta_2}{2\eta_1} \right)\|z_1 - z_0\|^2\right].$$
Therefore, the sequence $\{z_t\}$ and its weighted average sequences $\{x_t\}$ and $\{y_t\}$ are bounded, and we complete the entire proof.
\end{proof}
\vgap

Clearly, the conditions \eqref{cond_1} and \eqref{cond_4} indicate that it is possible to determine the stepsize $\eta_t$ without the line search procedure. 
Specifically, we propose the following adaptive stepsize policy with optimal convergence guarantees.
\begin{corollary} \label{main_corollary_1}
In the premise of Theorem~\ref{main_theorem}, suppose $\tau_1=0, ~\tau_t =\tfrac{t}{2}$, for $t \geq 2$ and $\beta \in (0,  1 -\revision{\tfrac{\sqrt{3}}{2}}{\tfrac{\sqrt{6}}{3}}]$, and the stepsize $\eta_t$ follows the rule:
\begin{align*}
\eta_2 &= \min \{\revision{2(1 - \beta) \eta_1}{(1 - \beta) \eta_1}, \revision{\tfrac{\beta}{2L_1}}{\tfrac{1}{4L_1}} \},\\
\eta_3 &= \min \{\eta_2, \tfrac{\revision{\beta}{1}}{4 L_2} \},\\
\eta_t &= \min\{\tfrac{t}{t-1} \eta_{t-1}, \tfrac{\revision{\beta (t-1)}{t-1}}{8 L_{t-1}} \}, ~t \geq 4.
\end{align*}
Then we have for $t \geq 2$,
\begin{align}\label{eta_lower_bound}
\eta_t \geq \tfrac{t \revision{\beta}{}}{12 \hat L_{t-1}},~~ \text{where} ~~\hat L_{t} := \max\{\tfrac{\revision{\beta}{1}}{4(1-\beta)\eta_1}, L_1,...,L_t\}.
\end{align}
Consequently, we have
\begin{align*}
\revision{f}{\Psi}(x_{k}) - \revision{f}{\Psi}(x^*) \leq \tfrac{12 \hat L_{k} }{k (k+1)\revision{\beta}{}}\cdot \left[\tfrac{1}{\beta} \|z_0 - x^*\|^2 + \eta_2\left( \revision{3L_1}{\tfrac{5L_1}{2}} - \tfrac{1}{\eta_1} \right)\|z_1 - z_0\|^2 \right],
\end{align*}
and
\begin{align*}
\revision{f}{\Psi}(\bar x_{k}) - \revision{f}{\Psi}(x^*) \leq \tfrac{1}{\sum_{t=2}^{k+1}\frac{t\revision{\beta}{}}{6 \hat L_{t-1}}}\cdot \left[\tfrac{1}{\beta} \|z_0 - x^*\|^2 + \eta_2\left(\revision{3L_1}{\tfrac{5L_1}{2}} - \tfrac{1}{\eta_1} \right)\|z_1 - z_0\|^2 \right].
\end{align*}
\end{corollary}
\begin{proof}
First of all, from the definition of $\tau_t$ and $\beta$, it is easy to see that the stepsize rule satisfies conditions \eqref{cond_3}-\eqref{cond_4}. Then, we prove Ineq.~\eqref{eta_lower_bound} using inductive arguments. For $t=2$, invoking the definition $\hat L_1 := \max\{\tfrac{\revision{\beta}{1}}{4(1-\beta)\eta_1}, L_1\}$, we have $\eta_2 = \revision{\tfrac{\beta}{2 \hat L_1}}{\tfrac{1}{4 \hat L_1}}$. Next for $t=3$, we have 
\begin{align*}
\eta_3 = \min\{\eta_2, \tfrac{\revision{\beta}{1}}{4 L_2}\} =\min\{\tfrac{\revision{\beta}{1}}{4\hat L_1}, \tfrac{\revision{\beta}{1}}{4 L_2}\} \geq \tfrac{3}{12 \hat L_2},
\end{align*}
which satisfies \eqref{eta_lower_bound}. Now we assume that for a $t \geq 3$, $\eta_t \geq \tfrac{t }{12\hat L_{t -1}}$, then
\begin{align*}
\eta_{t  + 1} = \min\{\tfrac{t +1}{t }\eta_{t }, \tfrac{ t }{8 L_{t }}\} \geq \min\{\tfrac{t +1}{12 \hat L_{t -1}}, \tfrac{ t }{8 L_{t }}\} \geq \min\{\tfrac{t +1}{12}, \tfrac{t }{8}\}\cdot \tfrac{\revision{\beta}{1}}{ \hat L_t } = \tfrac{t+1}{12 \hat L_t },
\end{align*}
where the last equality follows from $t \geq 3$. Therefore, we proved Ineq.~\eqref{eta_lower_bound}. By utilizing it in Ineqs.\eqref{main_thm_eq1} and \eqref{main_thm_eq2}, we obtain the desired convergence guarantees.
\end{proof}
\vgap


The bounds in Corollary~\ref{main_corollary_1} merit some comments. First, we notice that this stepsize rule is fully problem-parameter-free, with no line search required. Specifically, the stepsize $\eta_t$ is chosen based on the previous stepsize $\eta_{t-1}$ and the local Lipschitz constant $L_{t-1}$ from the previous iteration. The initial stepsize of the first iteration, $\eta_1$, can be chosen arbitrarily, but this choice will influence the subsequent stepsize selections. 
\revision{In practice}{A} desirable choice of $\eta_1$ should be neither too large nor too small. A too-large $\eta_1$  will introduce an extra term, $\eta_2(\revision{3L_1}{\tfrac{5L_1}{2}} - \tfrac{1}{\eta_1})\|z_1 - z_0\|^2$, on the right-hand side of the bounds in Corollary~\ref{main_corollary_1}. However, an excessively small $\eta_1$ can make $\tfrac{\revision{\beta}{1}}{4(1-\beta)\eta_1}$ dominate $\hat L_t$. 
\revision{}{In practice, a simple strategy is to find an $z_{-1} \in X$ (e.g., a perturbation of $z_0$) and compute 
\begin{align}\label{def_L_0}
\eta_1 := \tfrac{\zeta}{4(1-\beta)L_0}, \text{ where } L_0 := \tfrac{\|g(z_{-1}) - g(z_0)\|}{\|z_{-1}-z_0\|} \leq L,~\text{for some} ~\zeta > 0.
\end{align}
Then we have $\hat L_t \leq \mathcal{O}(L)$ for all $t \in \mathbb{Z}_+$. If the condition $L_1 \leq \tfrac{8(1-\beta)}{5\zeta}L_0$ is satisfied, both $x_k$ and $\bar x_k$ achieve the optimal convergence rate
\begin{align*}
\mathcal{O}\left(\tfrac{L}{k^2} \|z_0-x^*\|^2\right).
\end{align*}
In the case that $L_1 > \tfrac{8(1-\beta)}{5\zeta}L_0$, the convergence rate will be bounded by
\begin{align*}
\mathcal{O}\left(\tfrac{L}{k^2} [\|z_0-x^*\|^2+\|z_0 - z_1\|^2]\right),
\end{align*}
depending on both iterates $z_1$ and $z_0$. 

Another strategy that can fully get rid of the term $\|z_0 -z_1\|^2$ is to incorporate a simple line search procedure only in the first iteration:
\begin{itemize}
    \item[a)] Choose $L_0 \leq L$, e.g., as in \eqref{def_L_0}. Choose $\gamma \in (1, +\infty)$.
    \item[b)] Find the smallest $i \geq 0$ such that for $\eta_1^{(i)} = \tfrac{1}{4(1-\beta)L_0 \gamma^i}$ and 
    \begin{align*}
    z_1^{(i)} &= \arg \min_{z \in X} \left\{\revision{\eta_1^{(i)} \langle g(z_{0}), z \rangle}{\eta_1^{(i)}[\langle g(z_{0}), z \rangle + h(z)]} + \tfrac{1}{2}\|z_{0} - z\|^2\right\}
    \end{align*}
    we have
    \begin{align*}
    \eta_1^{(i)} \leq \tfrac{2}{5 L_1^{(i)}}, \text{ where } L_1^{(i)} := \tfrac{\|g(z_{1}^i) - g(z_0)\|}{\|z_{1}^i-z_0\|}.
    \end{align*}
    \item[c)] Set $\eta_1 = \eta_1^{(i)}$, $x_1 = z_1 = z_1^{(i)}$, $L_1 = L^{(i)}_1$.
\end{itemize}
\vspace{0.1in}

As a result, when the above line search procedure terminates for some  $i \geq 1$, we have
\begin{align*}
   \tfrac{2}{5 \gamma L} \leq \tfrac{2}{5 \gamma L_1^{(i-1)}} < \eta_1 \leq \tfrac{2}{5 L_1},
\end{align*}
indicating that $\hat L_t := \max\{\tfrac{1}{4(1-\beta)\eta_1}, L_1, ..., L_t\} \leq \tfrac{5 \gamma L}{8(1-\beta)}$, and 
\begin{align*}
\revision{f}{\Psi}(x_k) - \revision{f}{\Psi}(x^*) \leq \tfrac{15 \gamma L}{2k(k+1)\beta(1-\beta)} \|z_0-x^*\|^2, ~~\text{and}~~ \revision{f}{\Psi}(\bar x_k) - \revision{f}{\Psi}(x^*) \leq \tfrac{15\gamma L}{2k(k+3)\beta(1-\beta)} \|z_0-x^*\|^2.
\end{align*}
Therefore, in order to compute an $\epsilon$-optimal solution $\bar x$ satisfying $\revision{f}{\Psi}(\bar x) - \revision{f}{\Psi}(x^*) \leq \epsilon$, we need  at most 
\begin{align*}
\mathcal{O}\left(\sqrt{\tfrac{L \|z_0 - x^*\|^2}{\epsilon}}  + \log_\gamma (\tfrac{L}{L_0})\right)
\end{align*}
calls to the first-order oracle, which matches the lower bound in general complexity theorem \cite{nemirovski1983problem} up to only an additive factor of the initial line search steps. 
}

\paragraph{Remark.}Even when $f$ is only \emph{locally smooth}, there exists an upper bound $\bar L$ for $\hat L_t$ since all search points $\{x_t\}$ are bounded (as proved in Theorem~\ref{main_theorem}), thus we can obtain similar convergence guarantees with $L$ replaced by $\bar L$.

Although the worst-case guarantee is theoretically sound, the stepsize policy in Corollary~\ref{main_corollary_1} can sometimes be conservative in practice. In particular, if the initial estimation of the Lipschitz constant $L_1$ happens to be close to $L$, the subsequent stepsizes $\eta_t$ will be chosen as $\eta_t \sim \mathcal{O}(t/L)$ due to the update rule $\eta_t \leq \tfrac{t}{t-1}\eta_{t-1}$. Consequently, regardless of how small the local Lipschitz constant $L_t$ is, the stepsizes will not be much increased to reflect the local curvature. Therefore, to further improve the adaptivity of AC-FGM, we propose the following novel stepsize rule, which makes the selection of $\tau_t$ more adaptive to allow larger stepsizes $\eta_t$.

\begin{corollary}\label{main_corollary_2}
In the premise of Theorem~\ref{main_theorem}, suppose $\tau_1=0, ~\revision{\tau_2 = 2}{\tau_2=1}$,  and $\beta \in (0,  1 -\revision{\tfrac{\sqrt{6}}{3}}{\tfrac{\sqrt{6}}{3}}]$. 
Let $\alpha$ denote an absolute constant in $[0,1]$.  We set \revision{$\eta_2 = \tfrac{\beta}{2L_1}$}{$\eta_2 = \min\left\{(1-\beta)\eta_1, \tfrac{1}{4L_1} \right\}$}, and for $t\geq 3$,
\begin{align}
    \eta_t &= \min\left\{\revision{}{\tfrac{4}{3}\eta_{t-1}}, \tfrac{\tau_{t-2}+1}{\tau_{t-1}}\cdot \eta_{t-1}, \tfrac{\tau_{t-1}}{4 L_{t-1}}\right\}, \label{def_eta_t}\\
    \tau_t &= \tau_{t-1}+ \tfrac{\alpha}{2} + \tfrac{2 (1-\alpha) \eta_t L_{t-1}}{\tau_{t-1}}. \label{def_tau_t}
\end{align}
Then we have
\begin{align}\label{coro_2_eq0}
\revision{}{\Psi(x_k) - \Psi(x^*) \leq \tfrac{12 \hat L_k}{(\alpha k+4-2\alpha)(\alpha k + 3 - 2\alpha)}\cdot \left[\tfrac{1}{\beta} \|z_0 - x^*\|^2 + \eta_2\left( \revision{3L_1}{\tfrac{5L_1}{2}} - \tfrac{1}{\eta_1} \right)\|z_1 - z_0\|^2 \right],}
\end{align}
\revision{}{and}
\begin{align}\label{coro_2_eq1}
\revision{f}{\Psi}(\bar x_k) - \revision{f}{\Psi}(x^*) &\leq \tfrac{6}{ \sum_{t=1}^k (3+\alpha(t-2))/\hat L_t} \cdot \left[\tfrac{1}{\beta} \|z_0 - x^*\|^2 + \eta_2\left( \revision{3L_1}{\tfrac{5L_1}{2}} - \tfrac{1}{\eta_1} \right)\|z_1 - z_0\|^2 \right]\nn\\
&\leq \tfrac{12 \hat L_k }{6 k + \alpha k (k-3)}\cdot \left[\tfrac{1}{\beta} \|z_0 - x^*\|^2 + \eta_2\left( \revision{3L_1}{\tfrac{5L_1}{2}} - \tfrac{1}{\eta_1} \right)\|z_1 - z_0\|^2 \right],
\end{align}
where $\hat L_t:=\max\{\revision{}{\tfrac{1}{4(1-\beta)\eta_1}}, L_1, ..., L_t\}$ and $\bar x_k$ is defined in \eqref{def_bar_x_k}.
\end{corollary}

\begin{proof}
\revision{}{First, by using the condition $\beta\leq 1 - \tfrac{\sqrt{6}}{3}$, we have $2(1-\beta)^2 \geq \tfrac{4}{3}$, which together with \eqref{def_eta_t} ensures the stepsize rule \eqref{cond_4} is satisfied.} 
\revision{Next, since $\tau_2 \geq 2$, we have $\tau_t \geq 2$ for all $t \geq 2$. Consequently, for $t\geq 3$
\begin{align*}
\tfrac{\tau_{t-1}+1}{\tau_{t}} \leq\tfrac{\tau_{t-1} + 1}{\tau_{t-1}} \leq \tfrac{3}{2}.
\end{align*}
Also, we have $\tfrac{\tau_1 + 1}{\tau_2} = \tfrac{1}{2}$.
Invoking that $\beta\leq 1 - \tfrac{\sqrt{3}}{2}$, we can always ensure that $\tfrac{\tau_{t-2}+1}{\tau_{t-1}} \leq 2(1-\beta)^2$, which indicates that the condition \eqref{cond_4} is satisfied.}{} 
Next, by the definition of $\tau_t$ in \eqref{def_tau_t}, we have for $t \geq 3$
\begin{align}\label{tau_bound}
\tau_t = \tau_2 + \tsum_{j=3}^t [\tfrac{\alpha}{2} + \tfrac{2(1-\alpha)\eta_j L_{j-1}}{\tau_{j-1}}] \overset{(i)}\leq 1 + \tsum_{j=3}^t (\tfrac{1}{2}) = \tfrac{t}{2},
\end{align}
where step (i) follows from the condition $\eta_t \leq \tfrac{\tau_{t-1}}{4 L_{t-1}}$. Therefore, for $t \geq 3$, we always have
\begin{align}\label{bound_tau_ratio}
\tfrac{\tau_{t-1}+1}{\tau_{t}} \overset{(i)}\geq \tfrac{\tau_{t} +\frac{1}{2}}{\tau_{t}} \overset{(ii)}\geq \tfrac{t+1}{t},
\end{align} 
where step (i) follows from  $\tau_{t} \leq \tau_{t-1} + \tfrac{1}{2}$ due to \eqref{def_eta_t} and \eqref{def_tau_t}, and step (ii) follows from Ineq.~\eqref{tau_bound}.
Meanwhile, we have that 
\begin{align*}
\tau_t = \tau_2 + \tsum_{j=3}^t [\tfrac{\alpha}{2} + \tfrac{2(1-\alpha)\eta_j L_{j-1}}{\tau_{j-1}}] \geq 1 +\tfrac{\alpha (t-2)}{2}.
\end{align*}
Next, we use inductive arguments to prove that $\eta_{t} \geq \tfrac{3 + \alpha(t-3)}{12\hat L_{t-1}}$ for $t \geq 2$. First, it is easy to see that $\eta_2 = \min\{\revision{2(1-\beta)\eta_1}{(1-\beta)\eta_1}, \revision{\tfrac{\beta}{2L_1}}{\tfrac{1}{4L_1}} \} = \revision{\tfrac{\beta}{2 \revision{L_1}{\hat L_1}}}{\tfrac{1}{4\hat L_1}}$ \revision{}{and $\eta_3 = \min\{\eta_2, \tfrac{1}{4 L_2}\} = \tfrac{1}{4\hat L_2}$} satisfies this condition. Then, we assume that for a $t \geq 3$, we have $\eta_{t} \geq \tfrac{3 + \alpha(t-3)}{12 \hat L_{t-1}}$. Based on the stepsize rule \eqref{def_eta_t}, we have
\begin{align*}
\eta_{t+1} &= \min\big\{\tfrac{4}{3}\eta_{t}, \tfrac{\tau_{t-1}+1}{\tau_t} \eta_t, \tfrac{ \tau_t}{4 L_t}\big\}\\
&\overset{(i)}\geq \min\big\{\tfrac{t+1}{t}\eta_t, \tfrac{1+ \alpha(t-2)/2}{4 L_t}\big\}\\
&\geq \min\big\{\tfrac{t+1}{t} \cdot \tfrac{3 + \alpha(t-3)}{12 \hat L_{t-1}}, \tfrac{2+ \alpha(t-2)}{8 L_t}  \big\}\\
& \geq \tfrac{3 + \alpha(t-2)}{12\hat L_t},
\end{align*}
where step (i) follows from Ineq.~\eqref{bound_tau_ratio}, and we complete the inductive arguments. By utilizing \revision{$\eta_{t} \geq \tfrac{4 + \alpha(t-2)}{\revision{48}{40} \hat L_{t-1}}$}{the lower bounds on $\eta_t$ and $\tau_t$} in Ineqs. \eqref{main_thm_eq1} and \eqref{main_thm_eq2} in Theorem \ref{main_theorem}, we obtain the desired convergence guarantees.
\end{proof}
\vgap

Corollary~\ref{main_corollary_2} introduces a stepsize policy that allows AC-FGM to take larger stepsizes $\eta_t$ during execution. Remarkably, this policy remains line-search free, \revision{}{except for the first iteration where one can either consider using an initial line search stated after Corollary~\ref{main_corollary_1} to determine the initial stepsize or run without line search but possibly pay the price of $\|z_0-z_1\|^2$ in the convergence rate. On the other hand, the stepsize rule contains}
 one extra constant $\alpha$, which can be set as any number in $[0,1]$. When $\alpha=1$, this policy reduces to the policy in Corollary~\ref{main_corollary_1}. It should be noted that, although the worst-case upper bound in Ineq.~\eqref{coro_2_eq1} suggests us to choose a larger $\alpha$, a smaller $\alpha$ may allow the algorithm to accommodate larger stepsizes and consequently achieve faster convergence in practice. 
Here we explain how it works in more detail.
\begin{itemize}
    \item Let us consider the calculation of the stepsize $\eta_t$ in \eqref{def_eta_t}. If the previous stepsize $\eta_{t-1}$ is relatively conservative regarding the local smoothness level $L_{t-1}$, the \revision{first option}{option $\tfrac{\tau_{t-2}+1}{\tau_{t-1}}\cdot \eta_{t-1}$} in the \eqref{def_eta_t} will be activated. Consequently, together with the rule in \eqref{def_tau_t}, the parameter $\tau_t$ in \eqref{def_tau_t} will grow more slowly than $\tau_t = \tau_{t-1} + \tfrac{1}{2}$ in Corollary~\ref{main_corollary_1}, thus making $\tfrac{\tau_{t-1} + 1}{\tau_t}$ larger and allowing the stepsizes to grow faster in subsequent iterations. 
    \item Conversely, when the stepsize $\eta_{t-1}$ is too large regarding the local smoothness level $L_{t-1}$, the \revision{second option for the stepsize}{option $\tfrac{\tau_{t-1}}{4 L_{t-1}}$}  in \eqref{def_eta_t} will be activated for $\eta_t$, and $\tau_t$ will grow faster, as $\tau_t = \tau_{t-1} + \tfrac{1}{2}$, consequently slowing down the growth of the stepsizes in subsequent iterations. 
\end{itemize}
In summary, if we choose a small $\alpha$, the interplay between $\eta_t$ and $\tau_t$ can enhance the adaptivity of the AC-FGM updates to the local smoothness level, thus potentially leading to more efficient algorithmic behavior.

From the complexity point of view, when $\alpha$ is chosen as a constant in $(0,1]$ \revision{}{and the initial stepsize $\eta_1$ is chosen properly, i.e., $\eta_1 \leq \tfrac{2}{5 L_1}$ and $\eta_1^{-1} \leq \mathcal{O}(L)$ (see the discussion after Corollary~\ref{main_corollary_1} on how to meet these conditions)}, AC-FGM need at most
$
\mathcal{O}\left(\sqrt{\tfrac{L \|z_0 - x^*\|^2}{\epsilon}} \right)
$
iterations to find an $\epsilon$-optimal solution, also matching the optimal complexity in \cite{nemirovski1983problem}. On the other hand, when $\alpha = 0$, \revision{}{the output sequence $\{\bar x_k\}$ of} AC-FGM can guarantee an $\mathcal{O}\left(\tfrac{L\|z_0-x^*\|}{\epsilon} \right)$ complexity for finding an $\epsilon$-optimal solution, which matches the convergence rate of vanilla gradient descent or adaptive gradient descent methods. However, this is a worst-case guarantee determined by edge cases. In the upcoming section on numerical experiments (Section \ref{sec_numerical}), we will show that AC-FGM with $\alpha$ selected across a certain range of $[0,0.5]$ (including $\alpha=0$) can significantly outperform the non-accelerated adaptive algorithms, as well as other accelerated algorithms with line search. \revision{}{Meanwhile, we perform an ablation study in the experiments, showing that performing the initial line search to choose $\eta_1$ does not have a significant impact on the algorithm performance. Therefore, AC-FGM can be efficiently implemented in a completely line search-free manner.}


\section{AC-FGM for convex problems with H\"{o}lder continuous gradients}\label{sec_holder}
In this section, we consider the weakly smooth setting where the convex function $f: \bbr^n \rightarrow \bbr$ has H\"{o}lder continuous (sub)gradients, i.e., for some $\nu \in [0,1]$,
\begin{align}\label{eq:holder}
    \|g(x) - g(y)\| \leq L_\nu \|x-y\|^\nu, \quad x, y \in \bbr^n.
\end{align}
This inequality ensures that 
\begin{align}\label{eq:holder_2}
f(y) - f(x) - \langle g(x), y-x\rangle \leq \tfrac{L_\nu}{1+\nu} \|x-y\|^{1+\nu}.
\end{align}
It is noteworthy that \revision{}{Lemma 2 in} Nesterov \cite{nesterov2015universal}
established a relation that bridges this H\"{o}lder continuous setting and the smooth setting by showing that for any $\delta > 0$ and $\tilde L \geq \left[\tfrac{1-\nu}{(1+\nu)\delta} \right]^{\frac{1-\nu}{1+\nu}}\cdot L_\nu^{\frac{2}{1+\nu}}$,
\begin{align}\label{holder_condition_nesterov}
f(y) - f(x) - \langle g(x), y-x\rangle \leq \tfrac{\tilde L}{2} \|y-x\|^2 + \revision{\delta}{\tfrac{\delta}{2}}, \quad \forall x, y \in X.
\end{align}
Motivated by this result, in Lemma~\ref{lemma:holder_1}, we derive the lower bound for $f(y) - f(x) - \langle g(x), y-x\rangle$ similar to the first inequality of \eqref{eq:smooth_2}, with a tolerance $\delta$. 
This lower bound is crucial for our analysis of the AC-FGM method
applied to the H\"{o}lder continuous setting. \textcolor{black}{Furthermore, we find that this lower bound is an equivalent argument of the  H\"{o}lder continuous condition \eqref{eq:holder} and the relaxed quadratic upper bound \eqref{holder_condition_nesterov}.}

\begin{lemma}\label{lemma:holder_1}
Let function $f$ satisfy condition \eqref{eq:holder}. Then for any $\delta >0$ and
\begin{align}\label{condition_tilde_L}
\tilde L \geq \left[\tfrac{1-\nu}{(1+\nu)\delta} \right]^{\frac{1-\nu}{1+\nu}}\cdot L_\nu^{\frac{2}{1+\nu}}, 
\end{align}
we have
\begin{align}\label{eq:holder_3}
    f(y) \geq f(x) + \langle g(x), y -x\rangle + \tfrac{1}{2\tilde L}\|g(x)-g(y)\|^2 - \tfrac{\delta}{2}, \quad \forall x, y \in X.
\end{align}
\revision{}{Consequently, we have
}
\begin{align}\label{eq:holder_4}
\color{black}
    \tfrac{\|g(x) - g(y)\|^2}{2\tilde L} - \tfrac{\tilde L \|x - y\|^2}{2} \leq \delta, \quad \forall x, y \in X.
\end{align}
{\color{black} Moreover, the arguments in \eqref{eq:holder}, \eqref{holder_condition_nesterov}, and \eqref{eq:holder_3} are equivalent up to absolute constants. }
\end{lemma}
\begin{proof}\revision{}{
Since $f$ has H\"{o}lder continuous (sub)gradients, it satisfies condition \eqref{holder_condition_nesterov} that follows from Lemma 2 in \cite{nesterov2015universal}. 
Let us define $\phi(y):= f(y) - \langle g(x), y\rangle$. Then clearly, $\phi$  also satisfies condition \eqref{holder_condition_nesterov}. We have $\nabla \phi(y) = g(y) - g(x)$, which indicates that 
\begin{align*}
\min_{y} \phi(y) = \phi(x).
\end{align*}}
\revision{Consequently, we have for any $\tilde L > 0$,
\begin{align}\label{lemma_eq_1}
    \phi(x)&\leq \phi(y - \tfrac{1}{\tilde L} \nabla \phi(y))\nn\\
    &\leq \phi(y) + \langle \nabla \phi(y), -\tfrac{1}{\tilde L}\nabla \phi(y)\rangle + \tfrac{L_\nu}{1+\nu} \|\tfrac{1}{\tilde L}\nabla \phi(y)\|^{1+\nu}\nn\\
    & = \phi(y) - \tfrac{1}{\tilde L}\|\nabla \phi(y)\|^2 + \tfrac{L_\nu}{(1+\nu)\tilde L^{1+\nu}} \|\nabla \phi(y)\|^{1+\nu}.
\end{align}}{
Consequently, we have for any $\tilde L$ that satisfies \eqref{condition_tilde_L},
\begin{align}\label{lemma_eq_1}
    \phi(x)&\leq \phi(y - \tfrac{1}{\tilde L} \nabla \phi(y))\nn\\
    &\overset{(i)}\leq \phi(y) + \langle \nabla \phi(y), -\tfrac{1}{\tilde L}\nabla \phi(y)\rangle + \tfrac{\tilde L}{2}\|\tfrac{1}{\tilde L} \nabla \phi(y)\|^2 + \tfrac{\delta}{2}\nn\\
    & = \phi(y) - \tfrac{1}{2\tilde L}\|\nabla \phi(y)\|^2 + \tfrac{\delta}{2},
    \end{align}
    where step (i) follows from Ineq.~\eqref{holder_condition_nesterov}.}
\revision{
On the other hand, it follows from
the Young's inequality that for all $\tau\ge 0$ and $s \ge 0$  
\begin{align*}
\tfrac{1}{p} \tau^p + \tfrac{1}{q}s^q \geq \tau s, \quad \text{where} ~p, q \geq 1,~ \tfrac{1}{p} + \tfrac{1}{q} = 1.
\end{align*}
By taking $p = \tfrac{2}{1+\nu}$, $q = \tfrac{2}{1-\nu}$ and $\tau = t^{1 + \nu}$, we get
\begin{align*}
t^{1+\nu} \leq \tfrac{1+\nu}{2s}\cdot t^2 + \tfrac{1-\nu}{2} \cdot s^{\frac{1+\nu}{1-\nu}}.
\end{align*}
If we choose $s = \left[\tfrac{1+\nu}{1-\nu} \cdot \tfrac{\delta}{L_\nu} \right]^{\frac{1-\nu}{1+\nu}}$, then $\tfrac{1-\nu}{1+\nu} \cdot L_\nu s^{\frac{1-\nu}{1+\nu}} = \delta$. Therefore,
\begin{align*}
\tfrac{L_\nu}{1+\nu} \cdot t^{1+\nu} 
\leq \tfrac{1}{2s} \cdot L_\nu t^2 + \tfrac{1-\nu}{2(1+\nu)} \cdot L_\nu s^{\frac{1-\nu}{1+\nu}}
= \tfrac{1}{2s} \cdot L_\nu t^2 + \tfrac{\delta}{2}.
\end{align*}
Now let us take $t = \tfrac{\|\nabla \phi(y)\|}{\tilde L}$ and assume $\tilde L \geq \tfrac{L_\nu}{s}$, then we have
\begin{align}\label{lemma_eq_2}
\tfrac{L_\nu}{(1+\nu)\tilde L^{1+\nu}} \|\nabla \phi(y)\|^{1+\nu} \leq \tfrac{L_\nu}{2s} \cdot \tfrac{\|\nabla \phi(y)\|^2}{\tilde L^2} + \tfrac{\delta}{2}\leq \tfrac{1}{2 \tilde L} \|\nabla \phi(y)\|^2 + \tfrac{\delta}{2}.
\end{align}}{}
\revision{Combining Ineqs. \eqref{lemma_eq_1} and \eqref{lemma_eq_2}, and}{}Recalling the expressions of $\phi$, $\nabla \phi$, we obtain that
\begin{align*}
f(y) \geq f(x) + \langle g(x), y -x\rangle + \tfrac{1}{2\tilde L}\|g(x) - g(y)\|^2 - \tfrac{\delta}{2},
\end{align*}
which completes the proof \revision{}{of Ineq.~\eqref{eq:holder_3}. Moreovoer, Ineq.~\eqref{eq:holder_4} follows directly by summing up Ineqs.~\eqref{holder_condition_nesterov} and \eqref{eq:holder_3}.}

{\color{black} Next, we establish the equivalence of \eqref{eq:holder}, \eqref{holder_condition_nesterov}, and \eqref{eq:holder_3}. First, we know that \eqref{eq:holder} indicates the statement in \eqref{holder_condition_nesterov} thanks to Lemma 2 in \cite{nesterov2015universal}. Second, we have that \eqref{holder_condition_nesterov} indicates \eqref{eq:holder_3} given the proof of \eqref{eq:holder_3} above. Therefore, it remains to show that \eqref{eq:holder_3} implies \eqref{eq:holder}. To demonstrate this, we first add two copies of Ineq. \eqref{eq:holder_3} with $x$ and $y$ interchanged to obtain 
\begin{align*}
    \langle g(x) - g(y), x -y\rangle \geq \tfrac{1}{\bar L(\delta)}\|g(x)-g(y)\|^2 - \delta, ~ \text{where}~ \bar L(\delta):=\left[\tfrac{1-\nu}{(1+\nu)\delta} \right]^{\frac{1-\nu}{1+\nu}}\cdot L_\nu^{\frac{2}{1+\nu}}.
\end{align*}
Next, we maximize the RHS of the above inequality with respect to $\delta$ and obtain
\begin{align*}
    \langle g(x) - g(y), x -y\rangle \geq \max_{\delta > 0} \left\{ \tfrac{1}{\bar L(\delta)}\|g(x)-g(y)\|^2 - \delta\right\} = \tfrac{2\nu}{1+\nu} L_\nu^{-1/\nu}\|g(x)-g(y)\|^{1 + 1/\nu}.
\end{align*}
Then, by applying Cauchy-Schwarz inequality on the LHS of the above inequality and rearranging the terms, we have
\begin{align*}
\|g(x)-g(y)\| \leq \left( \tfrac{1+\nu}{2\nu}\right)^\nu L_\nu \|x-y\|^\nu \leq 1.262 L_\nu \|x-y\|^\nu.
\end{align*}
Thus \eqref{eq:holder_3} indicates \eqref{eq:holder} up an absolute constant, and we complete the proof.}
\end{proof}
\vgap

\revision{Lemma~\ref{lemma:holder_1}}{Clearly, Ineq.~\eqref{eq:holder_3}} provides a similar bound to the first inequality of \eqref{eq:smooth_2} with a tolerance parameter $\delta$.
Similarly, \revision{the forthcoming lemma }{Ineq.~\eqref{eq:holder_4}} serves as another bridge between the smooth case and the Hölder continuous case, but with a focus on Ineq. \eqref{eq:smooth_1}.
\revision{The two lemmas above}{These inequalities} indicate that, with the knowledge of the tolerance parameter $\delta$, we can treat a problem with H\"{o}lder continuous gradients as a smooth problem with $\tilde L = \mathcal{O}\left(\left(\tfrac{1}{\delta} \right)^{\frac{1-\nu} {1+\nu}}\cdot L_\nu^{\frac{2}{1+\nu}}\right)$. 

We have established in Theorem~\ref{main_theorem} and Corollary~\ref{main_corollary_1}  that AC-FGM can solve the smooth problem \eqref{main_prob} without knowing the Lipschitz constant $L$. Similarly, we can show AC-FGM
does not require prior knowledge of $L_\nu$ and $\nu$ when applied
to solve H\"{o}lder continuous problems. Specifically, for a given accuracy $\epsilon >0$, we need to define the local constants:
\begin{align}\label{define_tilde_L}
\tilde L_{t}(\epsilon) := \begin{cases}
      \frac{\sqrt{\|x_1 - x_0\|^2\|g(x_1)-g(x_0)\|^2 + (\epsilon/4)^2} - \epsilon/4}{\|x_1 - x_0\|^2}, &\text{if}~~t =1,\\
       \frac{\|g(x_{t}) - g(x_{t-1})\|^2}{2[f(x_{t-1}) - f(x_t) - \langle g(x_t), x_{t-1} - x_t\rangle] + \epsilon / \tau_{t}}, & \text{if}~~t \geq 2,
      \end{cases}
\end{align}
where $\tau_t$ is an algorithm parameter in Algorithm~\ref{alg1}.

Now we are in the position to present the convergence properties for AC-FGM under the H\"{o}lder continuous setting, and consequently establish the uniformly optimal convergence guarantees. The next proposition, serving as an analog of Propsition~\ref{prop_1}, provides a characterization of the recursive relationship for the iterates of AC-FGM.

\begin{proposition}\label{prop_2}
For any given accuracy $\epsilon > 0$, assume the parameters $\{\tau_t\}$, $\{\eta_t\}$ and $\{\beta_t\}$ satisfies conditions \eqref{cond_3} and
\begin{align}
    \eta_2 &\leq \min\left\{\revision{2(1-\beta)\eta_1}{(1-\beta)\eta_1}, \revision{\tfrac{\beta}{2\tilde L_1(\epsilon)}}{\tfrac{1}{4\tilde L_1(\epsilon)}} \right\},\label{cond_2_2_0}\\ 
    \eta_t &\leq \min\left\{2(1-\beta)^2 \eta_{t-1}, \tfrac{\tau_{t-1}}{4\tilde L_{t-1}(\epsilon)}, \tfrac{\tau_{t-2}+1}{\tau_{t-1}}\cdot\eta_{t-1}\right\}, ~t \geq 3,
   \label{cond_2_2}
\end{align}
where $\tilde L_t(\epsilon)$ is defined in \eqref{define_tilde_L}. We have for any $z \in X$,
    \begin{align}\label{theorem_2_eq4}
&\tsum_{t=1}^k \eta_{t+1}\left[\tau_t f(x_{t})+ \langle g(x_t), x_t - z\rangle - \tau_t f(x_{t-1})\revision{}{+ h(z_{t}) - h(z)} \right] + \tfrac{\eta_2}{2\eta_{1}} \left[\|z_1 - y_1\|^2 + \|z_2 - z_1\|^2 \right]\nn\\
& \leq \tfrac{1}{2\beta}\|y_1 - z\|^2 - \tfrac{1}{2\beta}\|y_{k+1}-z\|^2 + \tsum_{t=2}^{k+1}\tilde \Delta_{t} + \tsum_{t=3}^{k+1}\tfrac{\eta_t \epsilon}{2},
\end{align}
where
{\begin{align*}
\color{black}
\tilde \Delta_t :=  
       \eta_t \langle g(x_{t-1}) - g(x_{t-2}), z_{t-1} - z_t \rangle - \tfrac{\eta_t \tau_{t-1}\|g(x_{t-1}) - g(x_{t-2})\|^2}{2\tilde L_{t-1}(\epsilon)} - \tfrac{\|z_t - y_{t-1}\|^2}{2}.
\end{align*}}
\end{proposition}
\begin{proof}
First, we recall that in the proof of Proposition \ref{prop_1}, the first few steps up to Ineq.~\revision{\eqref{eq6}}{\eqref{eq6_prime}} were obtained via the optimality conditions of step \eqref{prox-mapping} \revision{}{and the conditions on the stepsizes}. Therefore, Ineq.~\revision{\eqref{eq6}}{\eqref{eq6_prime}} is still valid in this proof. On the other hand, based on the definition of $\tilde L_t(\epsilon)$, we have that, for $t\geq 2$,
\begin{align}\label{theorem_2_eq1}
    f(x_{t-1}) - f(x_t) - \langle g(x_t), x_{t-1} - x_t\rangle = \tfrac{1}{2\tilde L_t(\epsilon)}\|g(x_t) - g(x_{t-1})\|^2 - \tfrac{\epsilon}{2\tau_{t}}.
\end{align}
Then for any $z \in X$, we have for $t \geq 3$
\begin{align}\label{theorem_2_eq2}
&\tau_{t-1} f(x_{t-1}) + \langle g(x_{t-1}), x_{t-1} - z\rangle \nn\\
&\overset{(i)}= \tau_{t-1} [f(x_{t-1}) + \langle g(x_{t-1}), x_{t-2} - x_{t-1}\rangle] + \langle g(x_{t-1}), z_{t-1} - z\rangle\nn\\
& \overset{(ii)}=\tau_{t-1} [f(x_{t-2}) - \tfrac{1}{2\tilde L_{t-1}(\epsilon)}\|g(x_{t-1}) - g(x_{t-2})\|^2 + \tfrac{\epsilon}{2\tau_{t-1}}] + \langle g(x_{t-1}), z_{t-1} - z\rangle,
\end{align}
where step (i) follows from the fact that $x - x_{t-1} = x - z_{t-1} - \tau_t(x_{t-2} - x_{t-1})$ due to Eq.~\eqref{output_series_2}, and step (ii) follows from Eq.~\eqref{theorem_2_eq1}. 
Combining Ineqs.~\eqref{eq6} and \eqref{theorem_2_eq2} and rearranging the terms, we obtain an analog of \eqref{eq9} in the proof of Proposition \ref{prop_1}, 
\begin{align}\label{theorem_2_eq3}
&\eta_t\left[\tau_{t-1} f(x_{t-1}) + \langle g(x_{t-1}), x_{t-1} - z\rangle - \tau_{t-1}f(x_{t-2}) \revision{}{+ h(z_{t-1}) - h(z)}\right] \nn\\
&\leq \tfrac{1}{2\beta_t}\|y_{t-1} - z\|^2 - \tfrac{1}{2\beta_t}\|y_{t}- z\|^2 + \eta_t \langle g(x_{t-1}) - g(x_{t-2}), z_{t-1}- z_t \rangle\nn\\
&\quad  - \tfrac{\eta_t \tau_{t-1}}{2\tilde L_{t-1}(\epsilon)}\|g(x_{t-1}) - g(x_{t-2})\|^2 - \tfrac{\revision{\beta_t}{1}}{2} \|z_t - y_{t-1}\|^2 + \tfrac{\eta_t \epsilon}{2}.
\end{align}
Notice that Ineq.~\eqref{eq9_2} in the proof of Proposition \ref{prop_1} also holds for this proof. Therefore, by taking the summation of Ineq.~\eqref{eq9_2} and the telescope sum of Ineq.~\eqref{theorem_2_eq3} for $t=3,..., k+1$, and recalling that $\tau_1=0$, $\beta_t = \beta$ for $t \geq 2$, we obtain
\begin{align*}
&\tsum_{t=1}^k \eta_{t+1}\left[\tau_t f(x_{t})+ \langle g(x_t), x_t - z\rangle - \tau_t f(x_{t-1}) \revision{}{+ h(z_{t}) - h(z)} \right] + \tfrac{\eta_2}{2\eta_{1}} \left[\|z_1 - y_1\|^2 + \|z_2 - z_1\|^2 \right]\nn\\
& \leq \tfrac{1}{2\beta}\|y_1 - z\|^2 - \tfrac{1}{2\beta}\|y_{k+1}-z\|^2 + \tsum_{t=2}^{k+1}\tilde \Delta_{t} + \tsum_{t=3}^{k+1}\tfrac{\eta_t \epsilon}{2},
\end{align*}
where $\tilde \Delta_t$ is defined in the statement of this proposition, and we complete the proof.
\end{proof}
\vgap

In the following theorem, we provide the convergence guarantees for AC-FGM in the H\"{o}lder continuous setting. Notice that we will focus on the convergence of the averaged iterate $\bar x_k$ rather than the last iterate $x_k$ in order to provide the optimal rate of convergence.


\begin{theorem}\label{theorem_universal_optimal}
For any given accuracy $\epsilon > 0$, suppose the parameters $\{\tau_t\}$, $\{\eta_t\}$ and $\{\beta_t\}$ satisfies conditions \eqref{cond_3}, \eqref{cond_2_2_0} and \eqref{cond_2_2}.
Then we have for $\bar x_k$ defined in \eqref{def_bar_x_k},
\begin{align*}
\revision{f}{\Psi}(\bar x_{k}) - \revision{f}{\Psi}(x^*) \leq \tfrac{1}{\sum_{t=2}^{k+1} \eta_t} \cdot \left[\tfrac{1}{2\beta} \|z_0 - x^*\|^2 + \left(\tfrac{\revision{3}{5}\eta_2 \tilde L_1(\epsilon)}{\revision{2}{4}} - \tfrac{\eta_2}{2\eta_1} \right)\|z_1 - z_0\|^2 \right] + \tfrac{\epsilon}{2}.
\end{align*}
\end{theorem}

\begin{proof}
In this proof, we first upper bound the terms $\tilde \Delta_t$ in Ineq.~\eqref{theorem_2_eq4} of Proposition \ref{prop_2}.
For $t\geq 3$, the bound $\tilde \Delta_t \leq \tfrac{\revision{\beta}{1}}{4}\|z_{t-1}-y_{t-2}\|^2 - \tfrac{\revision{\beta}{1}}{4}\|z_t - y_{t-1}\|^2$  holds for the same reason as Ineq.~\eqref{eq10} with $L_t$ replaced by $\tilde L_t(\epsilon)$. 
For $t = 2$, we have
\begin{align*}
\tilde \Delta_2 
& \overset{(i)}\leq \eta_2 \|g(x_1) - g(x_0)\|\|z_1 - z_2\| - \tfrac{\revision{\beta}{1}}{2}\|z_2 - y_1\|^2\\
&\overset{(ii)}\leq \eta_2 \|g(z_1) - g(z_0)\| \|z_1 - z_0\|+ \eta_2 \|g(z_1) - g(z_0)\| \|z_2 - z_0\| - \tfrac{\revision{\beta}{1}}{2}\|z_2 - z_0\|^2\\
&\overset{(iii)}\leq \tfrac{\eta_2\|g(z_1)-g(z_0)\|^2}{2\tilde L_1(\epsilon)} + \tfrac{\eta_2 \tilde L_1(\epsilon)\|z_1 - z_0\|^2}{2} + \tfrac{\eta_2\|g(z_1)-g(z_0)\|^2}{4\tilde L_1(\epsilon)}  +\eta_2 \tilde L_1(\epsilon)\|z_2 - z_0\|^2 - \tfrac{\revision{\beta}{1}}{2}\| z_2 - z_0\|^2\\
&\overset{(iv)}\leq \tfrac{3\eta_2\|g(z_1)-g(z_0)\|^2}{4\tilde L_1(\epsilon)} + \tfrac{\eta_2 \tilde L_1(\epsilon)\|z_1 - z_0\|^2}{2}  - \tfrac{1}{4}\|z_2 - z_0\|^2\\
&\overset{(v)}\leq \tfrac{5\eta_2 \tilde L_1(\epsilon)\|z_1 - z_0\|^2}{4} + \tfrac{3\eta_2\epsilon}{8} - \tfrac{1}{4}\|z_2 - y_1\|^2,
\end{align*}
where step (i) follows from Cauchy-Schwarz inequality; step (ii) follows from triangle inequality and \revision{}{$x_1 = z_1$, $x_0 = y_0 = y_1=z_0$ due to $\tau_1 = 0$ and $\beta_1 = 0$}; step (iii) follows from Young's inequality; \revision{}{step (iv) follows from $\eta_2 \leq \tfrac{\beta}{4\tilde L_1(\epsilon)}$; step (v) follows from $y_1=z_0$ and the fact that $\tfrac{\|g(z_1)-g(z_0)\|^2}{\tilde L_1(\epsilon)} = \tilde L_1(\epsilon)\|z_1-z_0\|^2 + \tfrac{\epsilon}{2}$ due to the definition of $\tilde L_1(\epsilon)$.} Then, we use the bound for $\tilde \Delta_t$ in \eqref{theorem_2_eq4} and obtain
\begin{align}\label{theorem_2_eq5}
&\tsum_{t=1}^k \eta_{t+1}\left[\tau_t f(x_{t})+ \langle g(x_t), x_t - z\rangle - \tau_t f(x_{t-1}) \revision{}{+ h(z_{t}) - h(z)} \right] + \tfrac{1}{2\beta}\|y_{k+1}-z\|^2 \nn\\
& \leq \tfrac{1}{2\beta}\|y_1 - z\|^2  + \left(\tfrac{5\eta_2L_1(\epsilon)}{4} - \tfrac{\eta_2}{2\eta_1}\right)\|z_1-z_0\|^2 - \tfrac{1}{4}\|z_{k+1}-y_k\|^2 + \tsum_{t=2}^{k+1}\tfrac{\eta_t \epsilon}{2},
\end{align}
Then, we set $z=x^*$. \revision{}{Recall the definition $\Psi(x) := f(x) + h(x)$. By utilizing the relationships \eqref{label_convex_h} and $\langle g(x_t), x_t - x^*\rangle \geq f(x_t) - f(x^*)$ due to the convexity of $f$}, and invoking the conditions \eqref{cond_4} and $\tau_1=0$ in \eqref{cond_3}, we obtain
\begin{align}\label{theorem_2_eq6}
&\tsum_{t=1}^{k-1}[(\tau_{t}+1)\eta_{t+1} - \tau_{t+1} \eta_{t+2}][\revision{f}{\Psi}(x_{t})-\revision{f}{\Psi}(x^*)] + (\tau_{k}+1)\eta_{k+1} [\revision{f}{\Psi}(x_{k})-\revision{f}{\Psi}(x^*)]\nn\\
&\leq  \tfrac{1}{2\beta} \|z_0 - x^*\|^2 + \left(\tfrac{5\eta_2L_1(\epsilon) }{4} - \tfrac{\eta_2}{2\eta_1} \right)\|z_1 - z_0\|^2 + \tsum_{t=2}^{k+1}\tfrac{\eta_t \epsilon}{2}.
\end{align}
Then, by using the definition of $\bar x_k$ in \eqref{def_bar_x_k}, we arrive at
\begin{align*}
    \revision{f}{\Psi}(\bar x_{k}) - \revision{f}{\Psi}(x^*) \leq  \tfrac{1}{\sum_{t=2}^{k+1}\eta_t }\cdot \left[\tfrac{1}{2\beta} \|z_0 - x^*\|^2 + \left(\tfrac{5\eta_2 L_1(\epsilon)}{4} - \tfrac{\eta_2}{2\eta_1} \right)\|z_1 - z_0\|^2 \right] + \tfrac{\epsilon}{2},
\end{align*}
which completes the proof.
\end{proof}
\vgap

\revision{}{It should be noted that we do not provide the convergence guarantee for the last iterate $x_k$ since the error introduced by the input accuracy $\epsilon$ can blow up when $\sum_{t=2}^{k+1}\eta_t$ is much larger than $(\tau_k+1)\eta_{k+1}$.} With the help of Theorem~\ref{theorem_universal_optimal}, we establish in the following corollary the uniformly optimal convergence rate for AC-FGM.
\begin{corollary}\label{corollary_uniform_optimal}
    Assume the algorithmic parameters of AC-FGM satisfy the conditions in Corollary~\ref{main_corollary_1} with $L_t$ replaced by $\tilde L_t(\epsilon)$. \revision{$\eta_1 \in [\tfrac{\beta}{4(1-\beta)\tilde L_1(\epsilon)}, \tfrac{1}{3 \tilde L_1(\epsilon)}]$}{Suppose the initial stepsize $\eta_1$ satisfies $\eta_1^{-1} \leq \mathcal{O}\big(\big(\tfrac{1}{\epsilon}\big)^{\frac{1-\nu}{1+\nu}} \cdot L_\nu^{\frac{2}{1+\nu}}\big)$ and $\eta_1 \leq \tfrac{2}{5\tilde L_1(\epsilon)}$.} Then we have
    \begin{align*}
\revision{f}{\Psi}(\bar x_{k}) - \revision{f}{\Psi}(x^*) \leq \mathcal{O}\left( \left(\tfrac{L_\nu^2}{\epsilon^{1-\nu} k ^{1+3\nu}}\right)^{\frac{1}{1+\nu}} \cdot \|z_0 - x^*\|^2\right) + \tfrac{\epsilon}{2}.
\end{align*}
The same argument holds for taking the stepsize rule in Corollary~\ref{main_corollary_2} with $\alpha \in (0,1]$.
\end{corollary}
\begin{proof}
First, due to the definition of $\tilde L_1(\epsilon)$, we have
\begin{align*}
\tfrac{\|g(x_1) - g(x_0)\|^2}{2\tilde L_1(\epsilon)} - \tfrac{\tilde L_1(\epsilon) \|x_1 - x_0\|^2}{2} \leq \tfrac{\epsilon}{4}.
\end{align*}
Together with \revision{Lemma~\ref{lemma:holder_2}}{Ineq.~\eqref{eq:holder_4} in Lemma~\ref{lemma:holder_1}}, it is easy to see that 
\begin{align*}
\tilde L_1(\epsilon) = \mathcal{O}\left(\left(\tfrac{1}{\epsilon}\right)^{\frac{1-\nu}{1+\nu}} \cdot L_\nu^{\frac{2}{1+\nu}}\right).
\end{align*}
Similarly, by combining the definition of $\tilde L_t(\epsilon), ~t\geq 2$ and \revision{}{Ineq.~\eqref{eq:holder_3} in Lemma~\ref{lemma:holder_1}}, we have that 
\begin{align*}
\tilde L_t(\epsilon) = \mathcal{O}\left(\left(\tfrac{\tau_t}{\epsilon}\right)^{\frac{1-\nu}{1+\nu}} \cdot L_\nu^{\frac{2}{1+\nu}}\right)= \mathcal{O}\left(\left(\tfrac{t}{\epsilon}\right)^{\frac{1-\nu}{1+\nu}} \cdot L_\nu^{\frac{2}{1+\nu}}\right), ~\forall t \geq 2.
\end{align*}
Therefore, we obtain that $\tilde L_t(\epsilon)= \mathcal{O}\left(\left(\tfrac{t}{\epsilon}\right)^{\frac{1-\nu}{1+\nu}} \cdot L_\nu^{\frac{2}{1+\nu}}\right)$. \revision{Recalling Theorem~\ref{theorem_universal_optimal}}{Similar to \eqref{eta_lower_bound} in Corollary~\ref{main_corollary_1}}, we have that
\begin{align*}
\revision{}{\eta_t \geq \tfrac{t}{12 \hat L_{t-1}(\epsilon)}~~\text{where} ~~\hat L_{t}(\epsilon) := \max\{\tfrac{\revision{\beta}{1}}{4(1-\beta)\eta_1}, \tilde L_1(\epsilon),...,\tilde L_t(\epsilon)\}.}
\end{align*}
\revision{when $\eta_1\in [\tfrac{\beta}{4(1-\beta)\tilde L_1(\epsilon)}, \tfrac{1}{3 \tilde L_1(\epsilon)}]$, }{Thus by recalling Theorem~\ref{theorem_universal_optimal} and the condition that $\eta_1^{-1} \leq \mathcal{O}\big(\big(\tfrac{1}{\epsilon}\big)^{\frac{1-\nu}{1+\nu}} \cdot L_\nu^{\frac{2}{1+\nu}}\big)$ and $\eta_1 \leq \tfrac{2}{5\tilde L_1(\epsilon)}$, we have}
\begin{align*}
\revision{f}{\Psi}(\bar x_{k}) - \revision{f}{\Psi}(x^*) &\leq \mathcal{O}\left(\tfrac{\max\{\revision{}{\eta_1^{-1}}, \tilde L_1(\epsilon), ..., \tilde L_k(\epsilon)\}}{k^2} \|z_0-x^*\|^2\right) + \tfrac{\epsilon}{2}\\
& \leq \mathcal{O}\left(\tfrac{1}{k^2} \cdot \left(\tfrac{k}{\epsilon}\right)^{\frac{1-\nu}{1+\nu}} \cdot L_\nu^{\frac{2}{1+\nu}}\cdot \|z_0-x^*\|^2\right) + \tfrac{\epsilon}{2}\\
&=\mathcal{O}\left( \left(\tfrac{L_\nu^2}{\epsilon^{1-\nu} k ^{1+3\nu}}\right)^{\frac{1}{1+\nu}} \cdot \|z_0 - x^*\|^2\right) + \tfrac{\epsilon}{2},
\end{align*}
which completes the proof.
\end{proof}
\vgap

\revision{}{Notice that the condition $\eta_1^{-1} \leq \mathcal{O}\big(\big(\tfrac{1}{\epsilon}\big)^{\frac{1-\nu}{1+\nu}} \cdot L_\nu^{\frac{2}{1+\nu}}\big)$ can be easily ensured by finding a $z_{-1}\in X$ and choosing $\eta_1$ in the order of $\tilde L_0^{-1}(\epsilon)$ where 
\begin{align}\label{def_tilde_L_0}
\tilde L_0(\epsilon):=\tfrac{\sqrt{\|z_{-1} - z_0\|^2\|g(z_{-1})-g(z_0)\|^2 + (\epsilon/4)^2} - \epsilon/4}{\|z_{-1} - z_0\|^2}.
\end{align}
To ensure the condition $\eta_1 \leq \tfrac{2}{5\tilde L_1(\epsilon)}$ at the same time, one can perform a similar line search procedure for the first iteration as in Section~\ref{sec_smooth} after Corollary~\ref{main_corollary_1}.
}

As a consequence of Corollary~\ref{corollary_uniform_optimal}, for any $\nu \in [0,1]$, the AC-FGM algorithm requires at most 
\begin{align*}
    \mathcal{O}\left( \left( \tfrac{L_\nu \|z_0-x^*\|^{1+\nu}}{\epsilon}\right)^{\frac{2}{1+3\nu}}\right)
\end{align*}
iterations to obtain an $\epsilon$-optimal solution for any convex objective with H\"{o}lder continuous gradients, and this complexity is optimal, supported by the general complexity theory of convex optimization in \cite{nemirovski1983problem}. \revision{}{This uniformly optimal guarantee is obtained without any prior knowledge of $L_\nu$ or $\nu$, but with a pre-specified accuracy $\epsilon$.}

\section{Numerical experiments}\label{sec_numerical}
In this section, we report a few encouraging experimental results that can demonstrate the advantages of the AC-FGM algorithm over several other well-known parameter-free methods for convex \revision{}{(composite)} optimization. Specifically, \revision{}{for smooth problems,} we compare AC-FGM against Nesterov's fast gradient method (NS-FGM) proposed in \cite{nesterov2015universal}, and the line search free adaptive gradient descent (AdGD) proposed in \cite{malitsky2023adaptive}. Meanwhile, we take Nesterov's accelerated gradient descent method (NS-AGD) \cite{nesterov1983method} as a benchmark. 
\revision{}{In the nonsmooth problem, we compare AC-FGM against NS-FGM and Nesterov's primal gradient method (NS-PGM), also proposed in \cite{nesterov2015universal}.} 
\revision{}{In all sets of experiments, we report the convergence results in terms of iterations (in plots), CPU time, and oracle calls (in tables), where the latter means the number of queries of the first-order oracle, which returns $f(x)$ and $g(x) \in \partial f(x)$. Specifically, each iteration of AC-FGM, AdGD, and NS-AGD requires one oracle call. The oracle calls used in each iteration of NS-PGD equals the number of line search steps within the iteration, while the one required by NS-FGM equals two times the number of line search steps within the iteration. Moreover, in the instances where the optimal objective $\Psi(x^*)$ is unknown, we run the algorithms long enough until at least two methods converge to a high enough accuracy, e.g., $10^{-10}$. Then we take the lowest objective values among all solvers as the approximation of $\Psi(x^*)$. }

Notice that in our AC-FGM method, one can choose any $\alpha \in [0, 1]$ in Corollary~\ref{main_corollary_2}, \revision{}{allowing different levels of adaptivity.}
\revision{To distinguish the performance between AC-FGM and other methods}{To demonstrate this}, we implement AC-FGM with three different choices of $\alpha$ in $[0, 0.5]$, namely AC-FGM:0.5, AC-FGM:0.1, and AC-FGM:0.0, in all sets of experiments. \revision{}{For simplicity, we set $\beta = 1 - \tfrac{\sqrt{6}}{3}$ as suggested in Corollary~\ref{main_corollary_2} in the reported experiments, but more refined choices of $\beta$ may lead to superior convergence in practice. For the initial stepsize $\eta_1$, we first calculate $L_0$ defined in \eqref{def_L_0} and set $\eta_1:=\tfrac{2}{5 L_0}$, thus our implementation is fully line search free. When implementing the NS-FGM method, we take the line search constant $\gamma=2$ and take $L_0$ as the initial guess for line search. When implementing the AdGD method, we use $\gamma=1.5$ and $L_0$ for searching the initial stepsize. Our source code for reproducing the experiments is publicly available.\footnote{\url{https://github.com/tli432/AC-FGM-Implementation}}}

\revision{}{
In Section~\ref{subsec_num_2}, we report the numerical results for solving the classic quadratic programming problems and lease square regression with Lasso regularization. 
Section~\ref{sec_num_square_root_Lasso} presents the experiments for solving the nonsmooth square root Lasso problems with uniform/universal methods.
In Section~\ref{subsec_num_3}, we report the experiments for the sparse logistic regression with $\ell_1$-regularization. Finally, in Section~\ref{sec_ablation_analysis}, we provide ablation analysis to compare the performance of AC-FGM with and without line search in the first iteration, showing that AC-FGM is not sensitive to the initial line search.}

\subsection{Quadratic programming and Lasso regularization}\label{subsec_num_2}
In this subsection, we first investigate the performance of AC-FGM on the classical quadratic programming (QP) problem
\begin{align}\label{QP_prob}
\min_{\revision{\|x\|\leq1}{x \in \bbr^n}} f(x):= \tfrac{1}{m}\|Ax-b\|_2^2,
\end{align}
where $A \in \bbr^{m\times n}$ and $b \in \bbr^m$. Clearly, this formulation covers the well-known least square linear regression when $A$ is the design matrix and $b$ is the response vector.

\revision{}{In the first set of experiments, we test the performance of the algorithms on both randomly generated instances and real-world datasets.} For the random instances, we first randomly choose an optimal solution $x^*\in \bbr^n$ in the unit Euclidean ball. Then we generate a matrix $A \in \bbr^{m\times n}$, where each entry is uniformly distributed in $[0,1]$. At last, we set $b = A x^*$, thus $f(x^*) = 0$. \revision{}{Clearly, problem \eqref{QP_prob} reduces to solving the linear system $A x = b$. 
For the real-world datasets, we consider two instances from LIBSVM with relatively large $m$ and small $n$, e.g., \emph{bodyfat (m=252, n=14)} and \emph{cadata (m=20640, n=8)}. Thus there is no need for feature selection. 
We report the convergence results in terms of iterations in Figure~\ref{fig_random_QP} and in terms of CPU time and oracle calls in Table \ref{tableQP}. Notice that the average CPU time of AC-FGM in the table is calculated as the average of AC-FGM with the three different choices of $\alpha$. }

\begin{figure}[h]
	\centering
	\begin{minipage}[t]{0.45\linewidth}
		\centering
		\includegraphics[width=5cm]{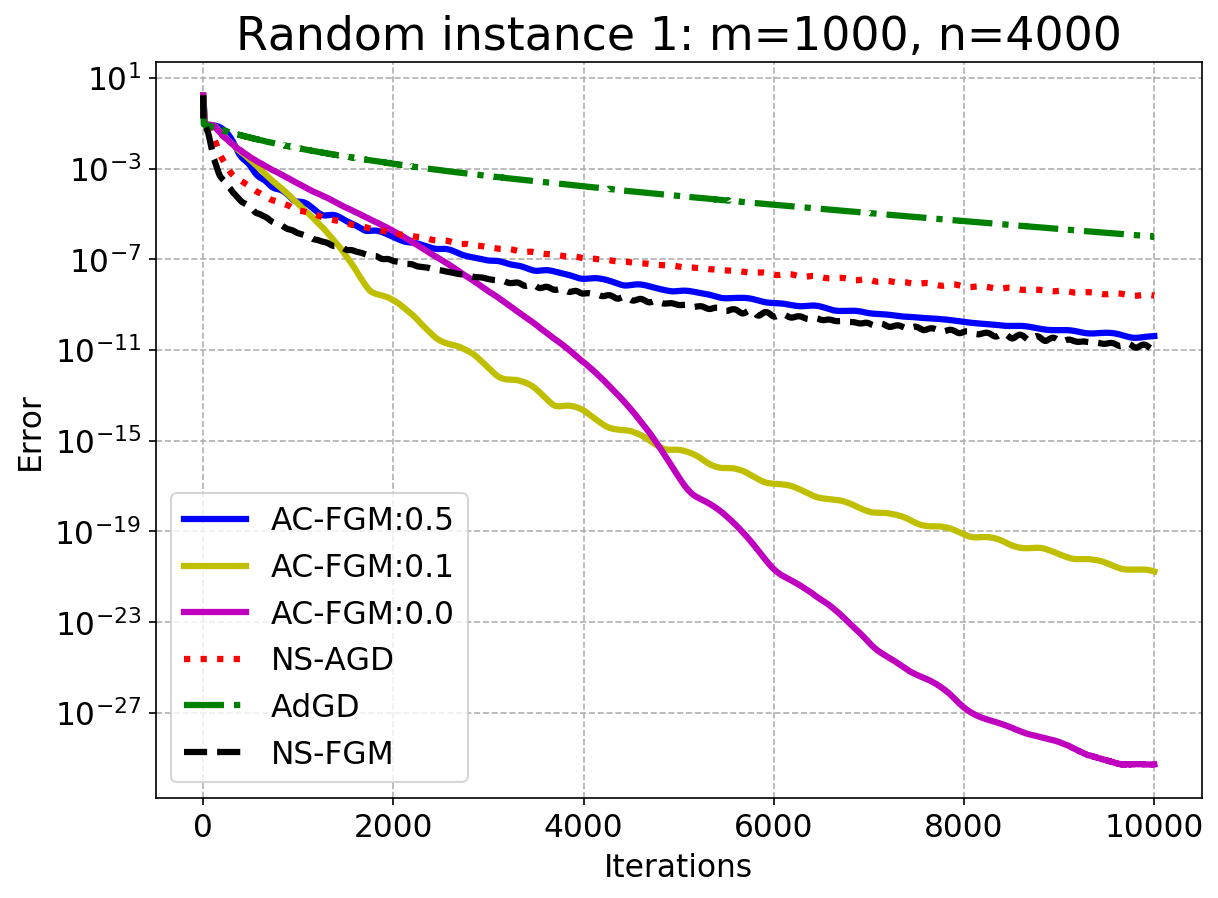}
	\end{minipage}
         \begin{minipage}[t]{0.45\linewidth}
		\centering
		\includegraphics[width=5cm]{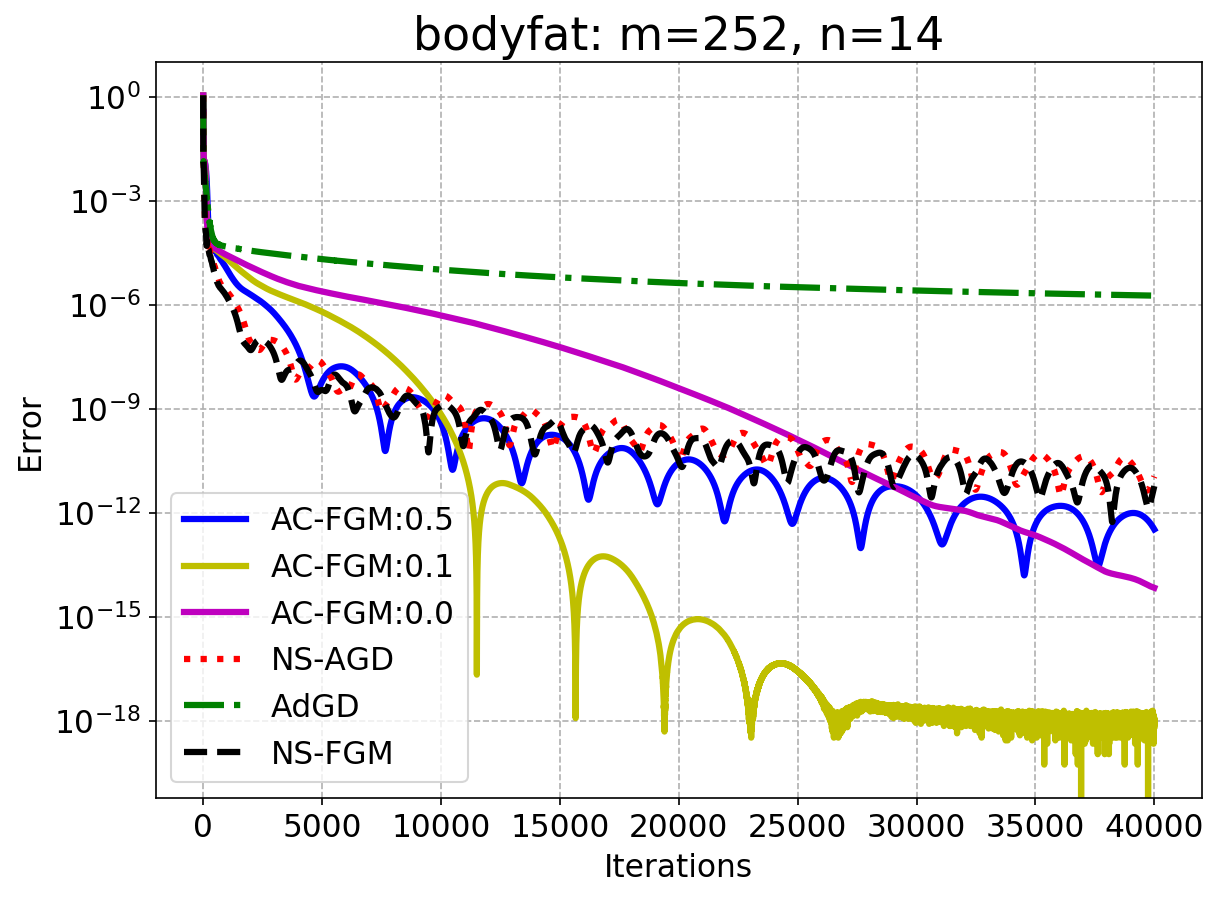}
	\end{minipage}
 \begin{minipage}[t]{0.45\linewidth}
		\centering
		\includegraphics[width=5cm]{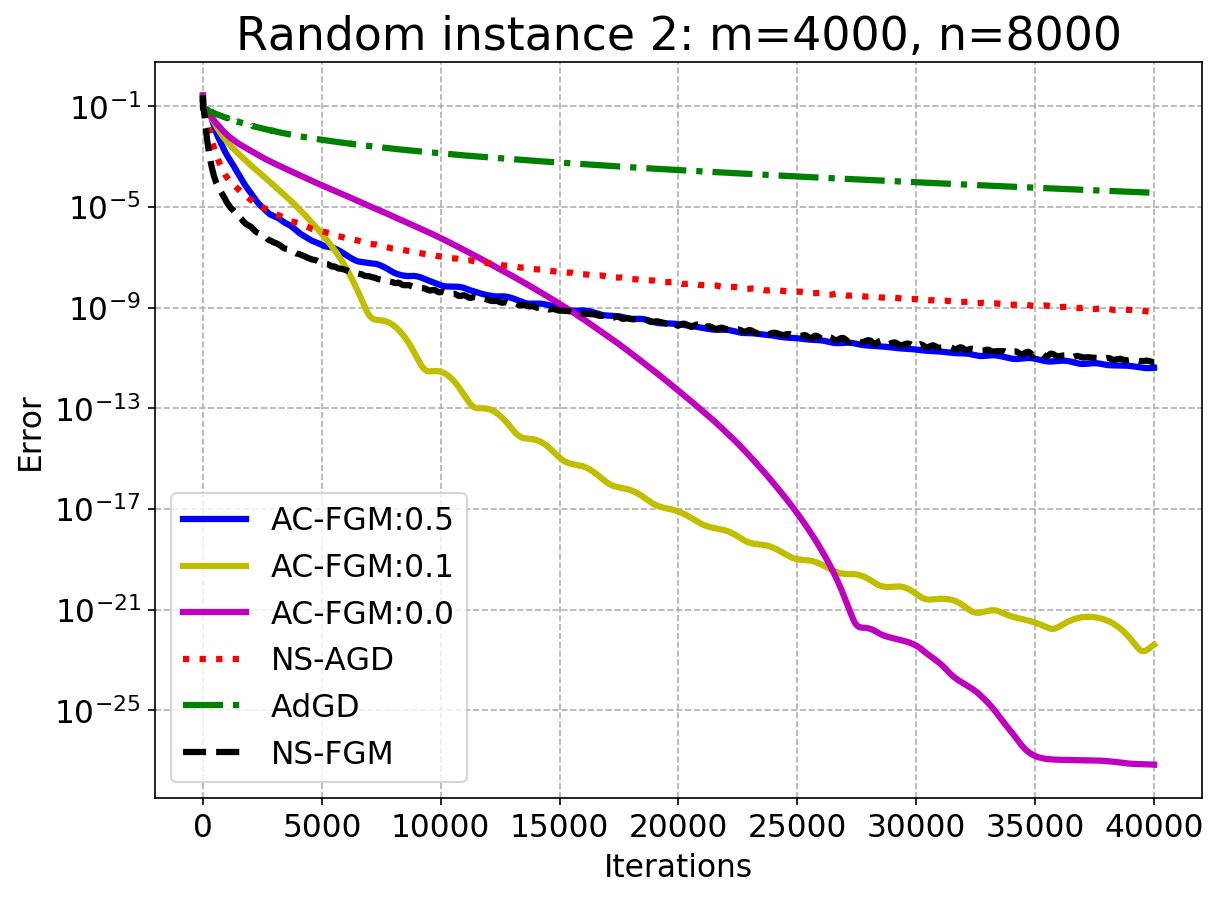}
	\end{minipage}
	\begin{minipage}[t]{0.45\linewidth}
		\centering
		\includegraphics[width=5cm]{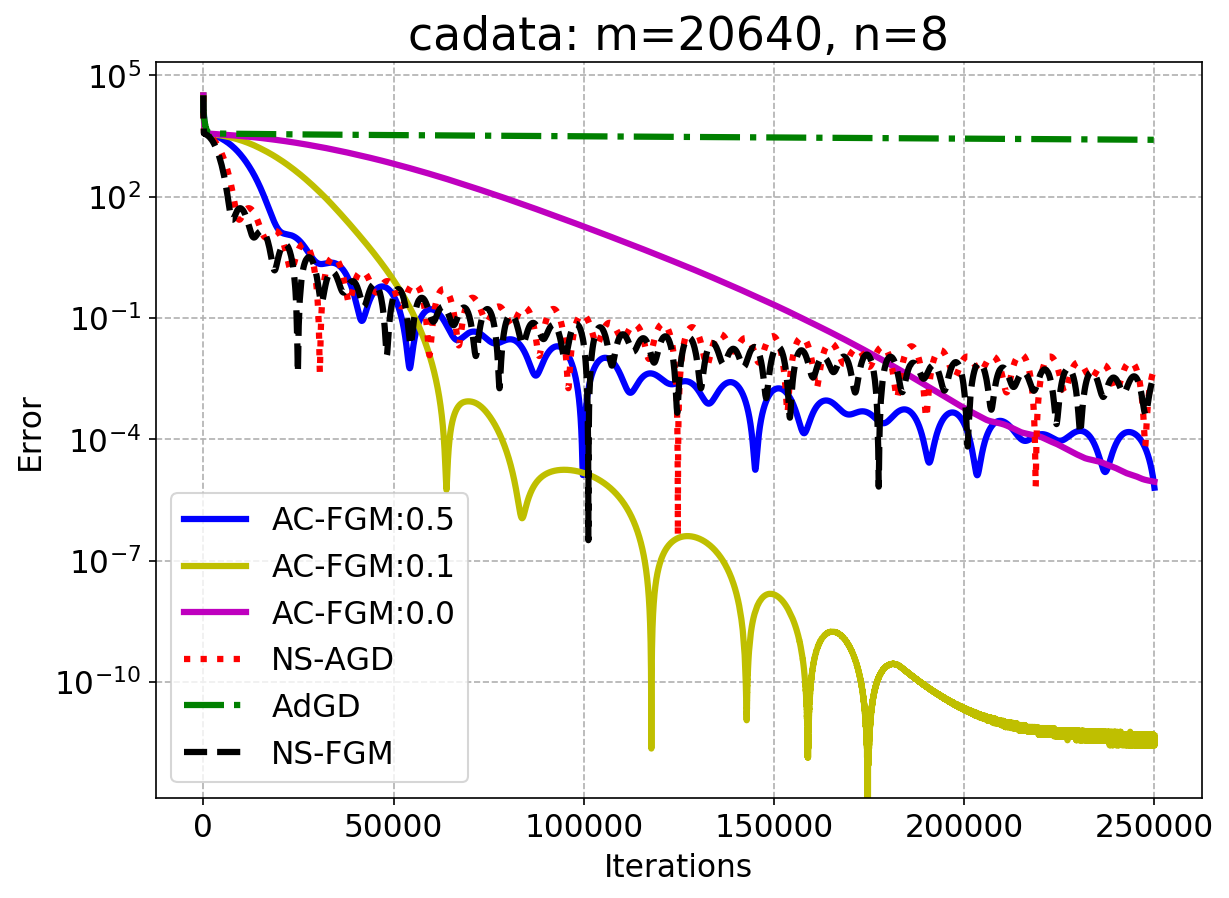}
	\end{minipage}
	\caption{Quadratic programming~\eqref{QP_prob}: Comparison between AC-FGM, NS-AGD, NS-FGM, and AdGD in terms of the number of iterations for solving randomly generated instances (column 1) and datasets \emph{bodyfat} and \emph{cadata} (column 2).}
	\label{fig_random_QP}
\end{figure}

\begin{table}[h]\scriptsize  
	\centering
 \renewcommand\arraystretch{1.3}
 \tabcolsep=0.15cm
	\begin{tabular}{|c|c|c|ccc|ccc|}
 \hline
 \multirow{3}*{Instances} & \multirow{3}*{m} & \multirow{3}*{n}&\multicolumn{3}{c|}{\multirow{2}*{Avg CPU time (per 1000 iterations)}}& \multicolumn{3}{c|}{\multirow{2}*{Avg \# oracle calls (per iteration)}} \\
 & ~ & ~ & ~ & ~ & ~& ~ &~ & ~\\
	& ~ &	~ & AC-FGM & AdGD & NS-FGM & AC-FGM & AdGD & NS-FGM\\\hline
  Random 1& 1000 & 4000 & 11.45 & 11.68 & 46.68& 1 &1 & 4.0\\
  Random 2& 4000 & 8000& 57.75 & 56.30 & 250.37& 1 &1 & 4.0\\
  bodyfat & 252 & 14 & 0.047 & 0.046 & 0.110 & 1 &1 & 4.0\\
  cadata& 20640 & 8 &  0.90 & 0.90  & 3.29 & 1 & 1 & 4.0
  \\\hline
	\end{tabular}
	\caption{Quadratic programming~\eqref{QP_prob}: Comparison of the averaged CPU time in [sec] and averaged oracle calls per iteration.}\label{tableQP}
\end{table}

\revision{}{
We make the following observations from the results in Figure~\ref{fig_random_QP} and Table~\ref{tableQP}. First, the accelerated approaches, AC-FGM and NS-FGM, significantly outperform the non-accelerated AdGD method in all 4 instances. Second, the convergence behavior of NS-FGM closely follows the NS-AGD method, while the AC-FGM with $\alpha=0.1$ largely outperforms the other two accelerated methods, especially in the high accuracy regime. In terms of the computational cost, NS-FGM takes around 4 oracle calls per iteration and the average CPU time is also roughly 4 times of AC-FGM and AdGD.}

\revision{}{
Next, we consider the well-known composite optimization problem, namely quadratic programming (QP)  with $\ell_1$-regularization:
\begin{align}\label{QP_prob_with_lasso}
\min_{\revision{\|x\|\leq1}{x \in \bbr^n}} f(x):= \tfrac{1}{m}\|Ax-b\|_2^2 \revision{}{+ \lambda \|x\|_1},
\end{align}
where $A \in \bbr^{m\times n}$ and $b \in \bbr^m$. This problem is also known as the ``Lasso'' regression in statistical learning, which is widely used for variable selection and model regularization. 
}
\revision{}{We test the performance of the algorithms for solving \eqref{QP_prob_with_lasso} on 2 large-scale datasets in LIBSVM, e.g., \emph{gisette (m=6000, n=5000)}, and \emph{rcv1.binary (m=20242, n=47236)}. We choose the penalty parameter $\lambda = \tfrac{c}{m} \|A^\top b \|_\infty$ with $c=0.01$ and $0.001$ respectively, suggested by \cite{moreau2022benchopt}. The experiment results are reported in Figure~\ref{fig_QP_libsvm_1} and Table~\ref{tableQP_lasso}.}

\begin{figure}[h]
	\centering
	\begin{minipage}[t]{0.45\linewidth}
		\centering
		\includegraphics[width=5cm]{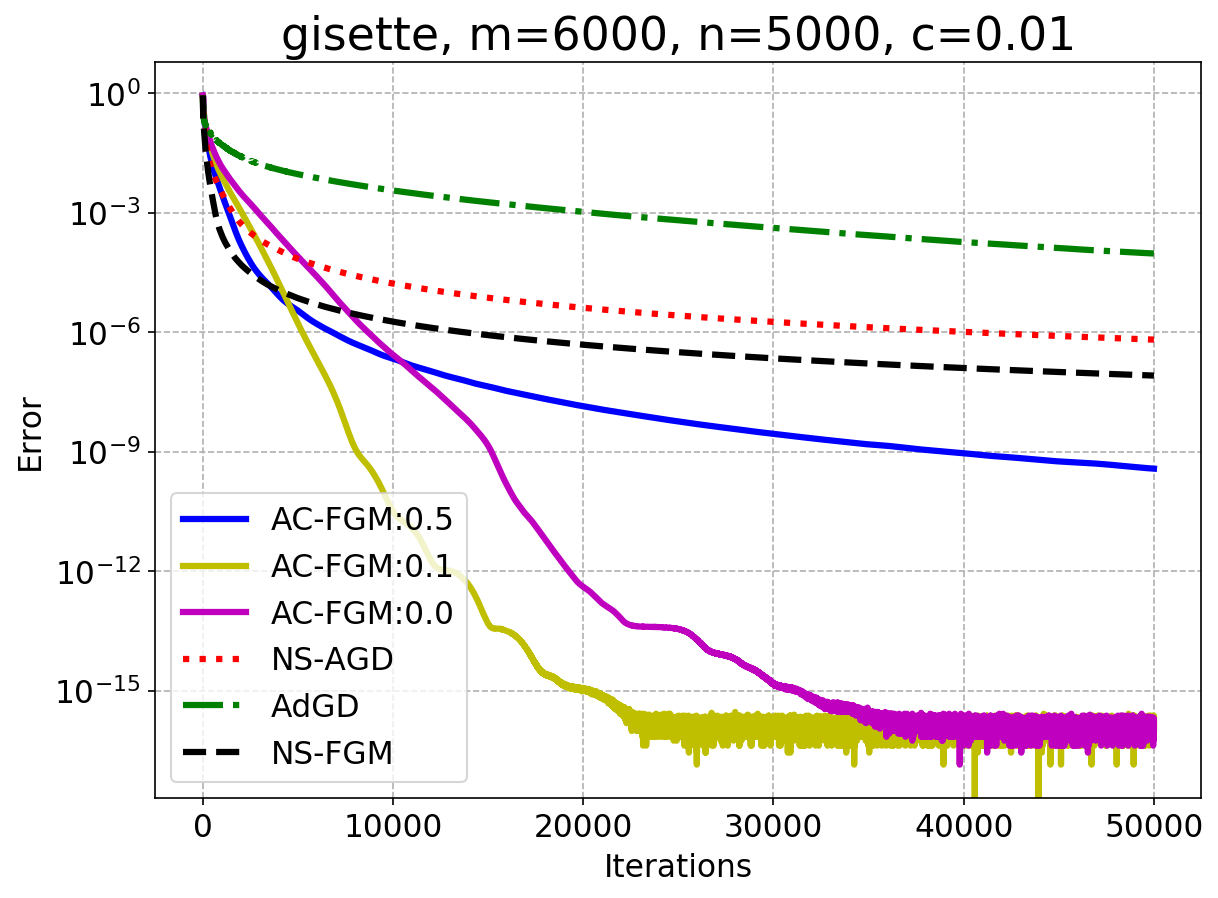}
	\end{minipage}
 \begin{minipage}[t]{0.45\linewidth}
		\centering
		\includegraphics[width=5cm]{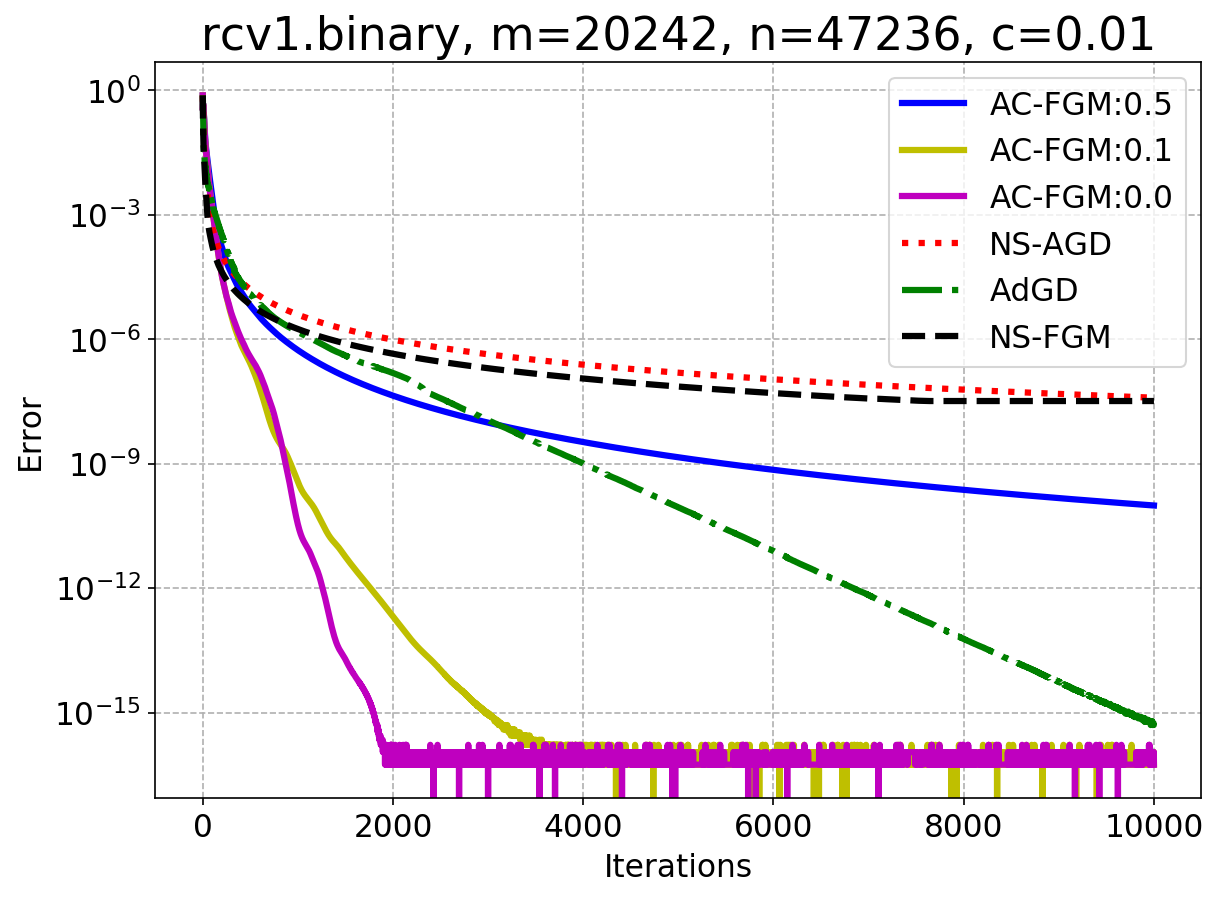}
	\end{minipage}
	\begin{minipage}[t]{0.45\linewidth}
		\centering
		\includegraphics[width=5cm]{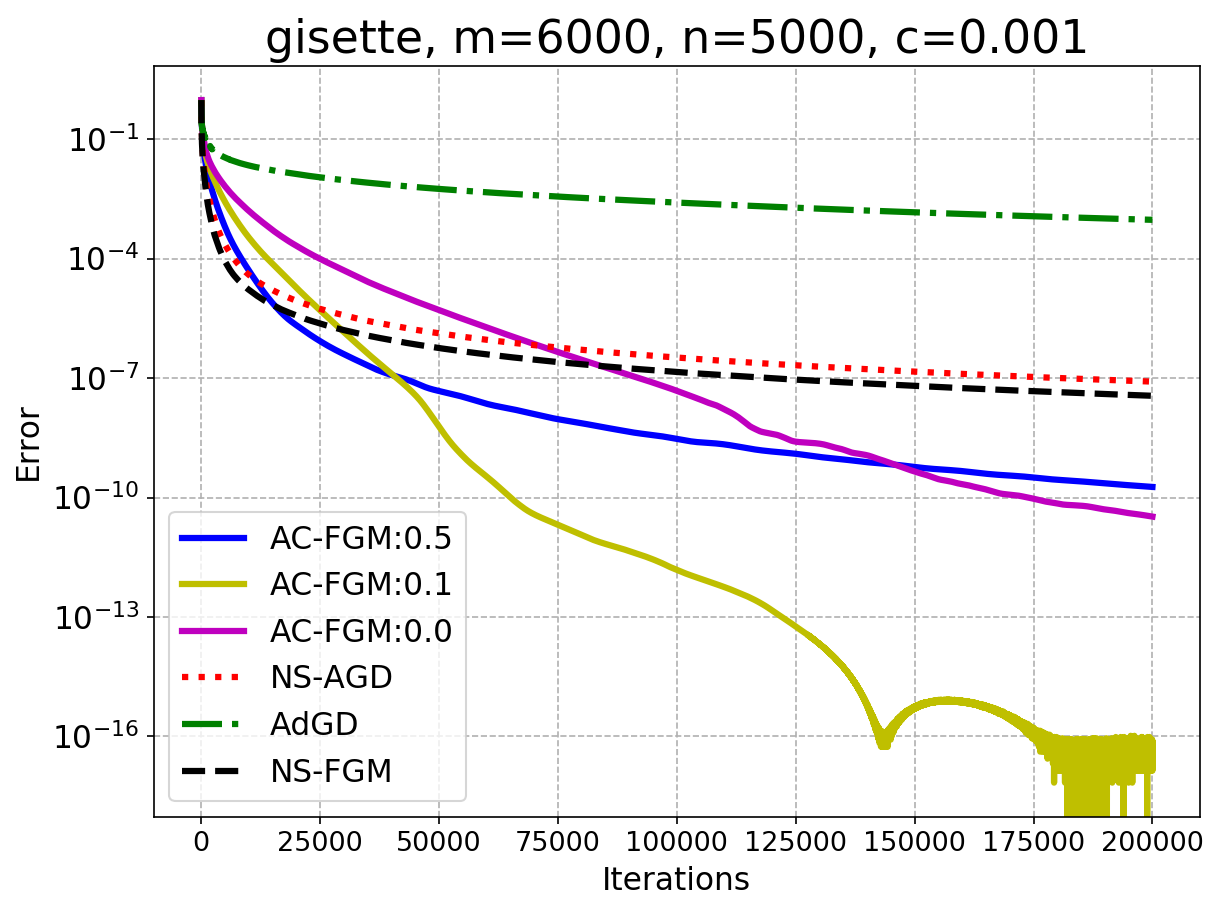}
	\end{minipage}
 \begin{minipage}[t]{0.45\linewidth}
		\centering
		\includegraphics[width=5cm]{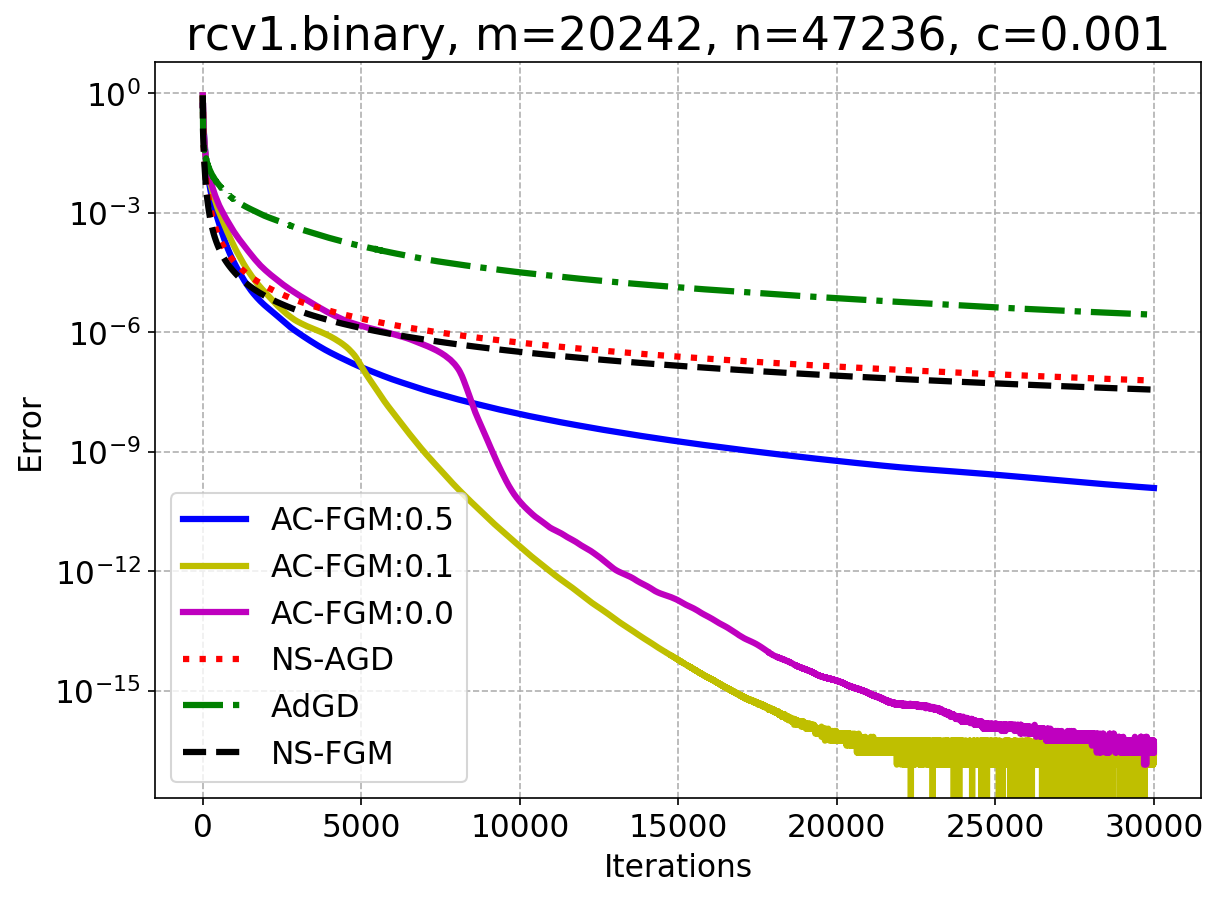}
	\end{minipage}
	\caption{Lasso \eqref{QP_prob_with_lasso}. Comparison between AC-FGM, NS-AGD, NS-FGM, and AdGD in terms of the number of iterations for datasets \emph{gisette} (column 1) and \emph{rcv1.binary} (column 2). Row 1 takes $\lambda = \tfrac{0.01}{m} \|A^\top b \|_\infty$ and Row 2 takes $\lambda = \tfrac{0.001}{m} \|A^\top b \|_\infty$.}
	\label{fig_QP_libsvm_1}
\end{figure}

\begin{table}[h]\scriptsize  
	\centering
 \renewcommand\arraystretch{1.3}
 \tabcolsep=0.15cm
	\begin{tabular}{|c|c|c|ccc|ccc|}
 \hline
 \multirow{3}*{Instances} & \multirow{3}*{m} & \multirow{3}*{n}&\multicolumn{3}{c|}{\multirow{2}*{Avg CPU time (per 1000 iterations)}}& \multicolumn{3}{c|}{\multirow{2}*{Avg \# oracle calls (per iteration)}} \\
 & ~ & ~ & ~ & ~ & ~& ~ &~ & ~\\
	& ~ &	~ & AC-FGM & AdGD & NS-FGM & AC-FGM & AdGD & NS-FGM\\\hline
  gisette ($c=0.01$)& 6000 & 5000 & 56.61 & 56.11 & 213.62 & 1 &1 & 4.0\\
  gisette ($c=0.001$)& 6000 & 5000&  57.80 & 62.30 & 199.74 & 1 &1 & 4.0\\
  rcv1.binary ($c=0.01$)  & 20242 & 47236 & 19.63 & 22.38 & 63.33 & 1 &1 & 4.0\\
  rcv1.binary ($c=0.001$)& 20242 & 47236 &  21.77 & 19.89  & 59.93 & 1 & 1 & 4.0
  \\\hline
	\end{tabular}
	\caption{Lasso~\eqref{QP_prob_with_lasso}: Comparison of the averaged CPU time in [sec] and averaged oracle calls per iteration.}\label{tableQP_lasso}
\end{table}

\revision{}{From Figure~\ref{fig_QP_libsvm_1}, we observe that AC-FGM methods with a smaller choice of $\alpha$, like $0.1$ and $0.0$, exhibit faster convergence than $\alpha=0.5$ in most of the experiments. Meanwhile, compared with NS-AGD and NS-FGM, our AC-FGM with smaller $\alpha$ can take more advantage of the larger penalty parameter $\lambda$ to achieve faster convergence. }

\subsection{Square root Lasso}\label{sec_num_square_root_Lasso}
In this subsection, we consider another well-known problem arisen from statistical learning, known as ``square root Lasso'' \cite{belloni2011square}. 
\begin{align}\label{square_root_lasso_prob}
\min_{\revision{\|x\|\leq1}{x \in \bbr^n}} f(x):= \tfrac{1}{\sqrt{m}}\|Ax-b\|_2 \revision{}{+ \lambda \|x\|_1}.
\end{align}
The theoretical advantage of the square root Lasso problem over the ``least square Lasso'' in \eqref{QP_prob_with_lasso} lies in the more robust choice of the penalty parameter and fewer assumptions on the model; see detailed discussion in \cite{belloni2011square}. Notice that this 
problem is nonsmooth, so we consider solving it with the following uniform/universal methods, i.e., AC-FGM, NS-PGM, and NS-FGM.

We test the performance of the algorithms on two real-word datasets from LIBSVM, i.e., \emph{gisette (m=6000, n=5000)}, and \emph{YearPredictionMSD.test (m=51630, n=90)}. 
The theoretical guarantees in \cite{belloni2011square} suggest the choice of penalty parameter $\lambda = c\cdot m^{-1/2}\Phi^{-1}(1-0.01/n)$, where $\Phi$ is the CDF of the standard normal distribution.
For all datasets, we report the algorithm performances under the choices $c=10$ and $100$, respectively. Moreover, we set $\epsilon=10^{-8}$ when implementing all the algorithms. The results are reported in Figure~\ref{fig_SRL_1} and Table~\ref{table_square_root_lasso}.

\begin{figure}[h]
	\centering
	\begin{minipage}[t]{0.45\linewidth}
		\centering
		\includegraphics[width=5cm]{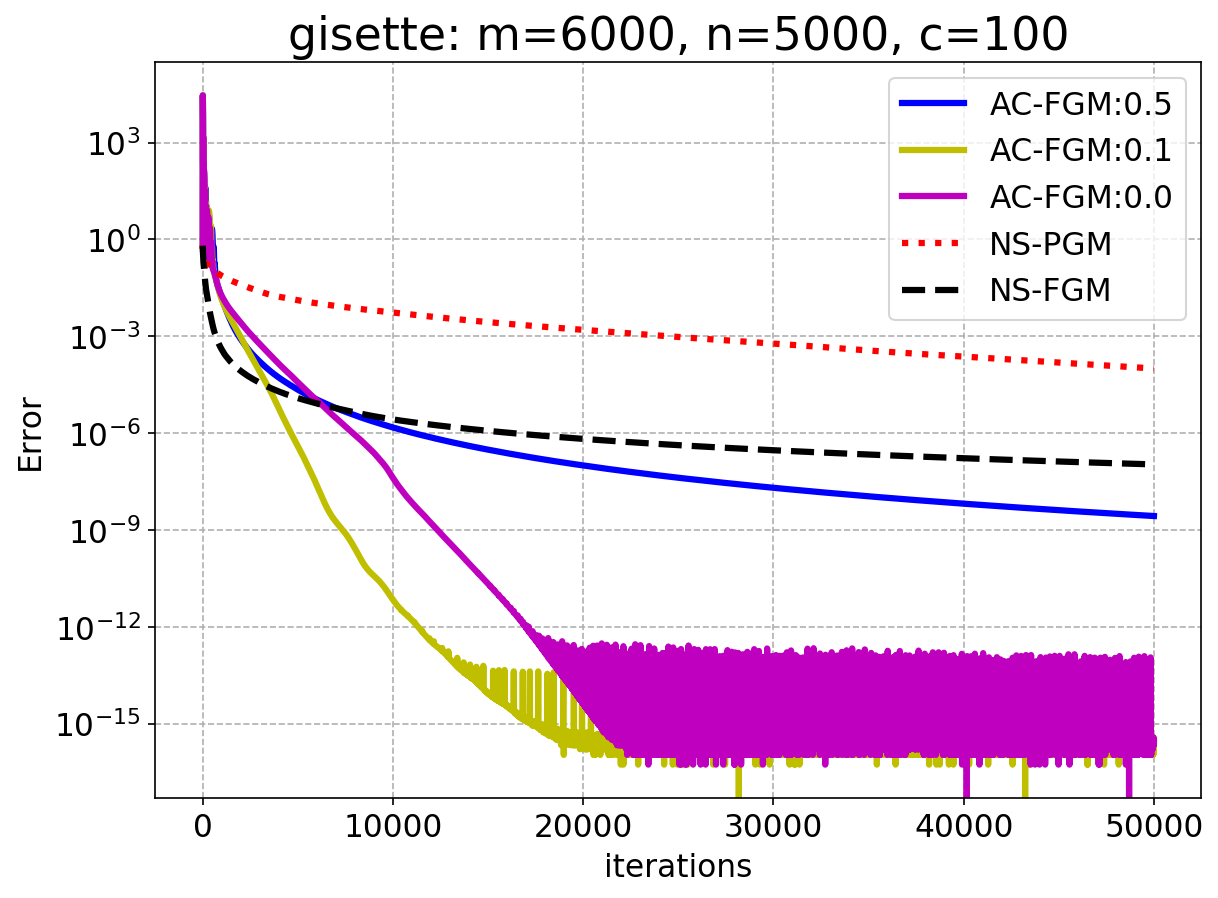}
	\end{minipage}
	\begin{minipage}[t]{0.45\linewidth}
		\centering
		\includegraphics[width=5.1cm]{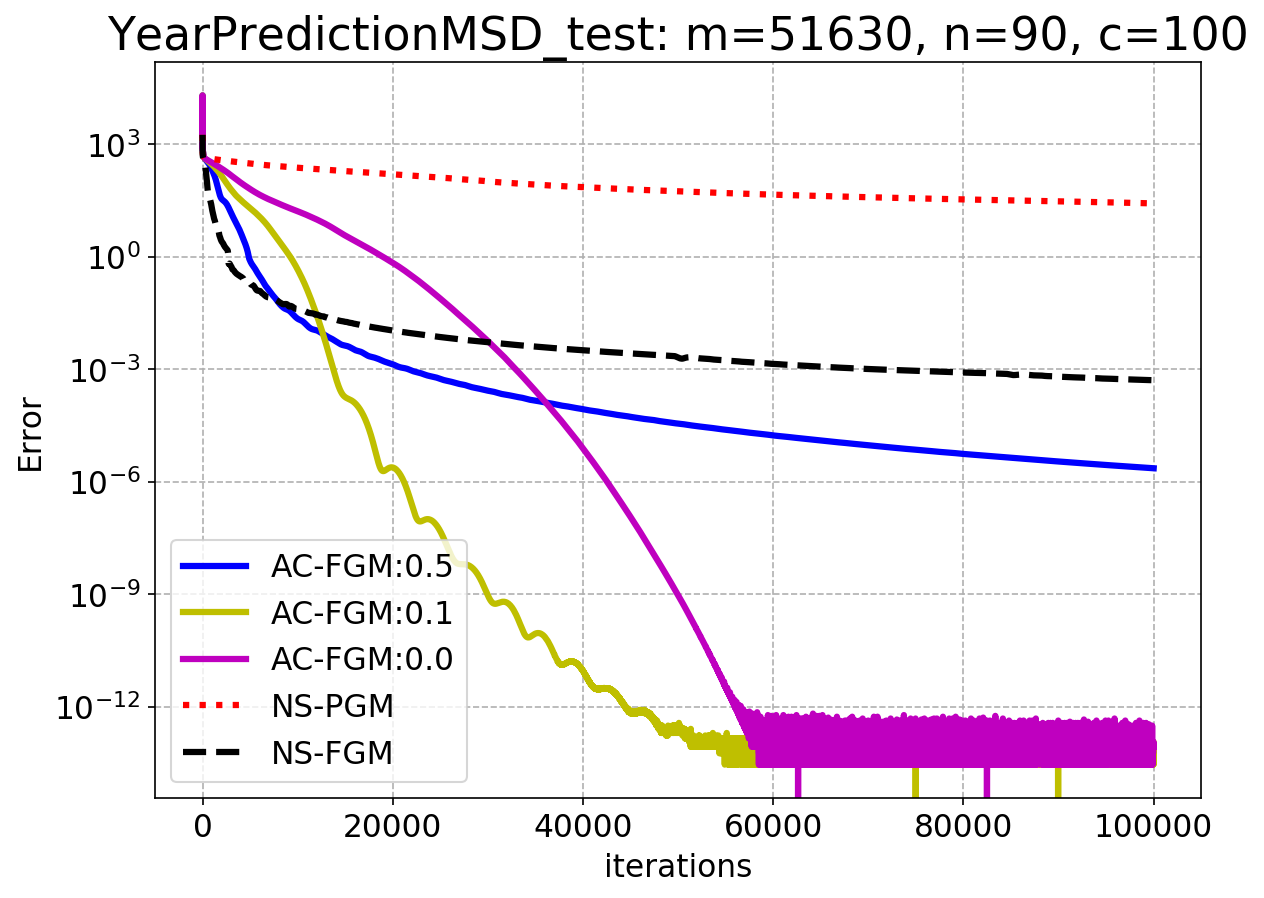}
	\end{minipage}
	\begin{minipage}[t]{0.45\linewidth}
		\centering
		\includegraphics[width=5cm]{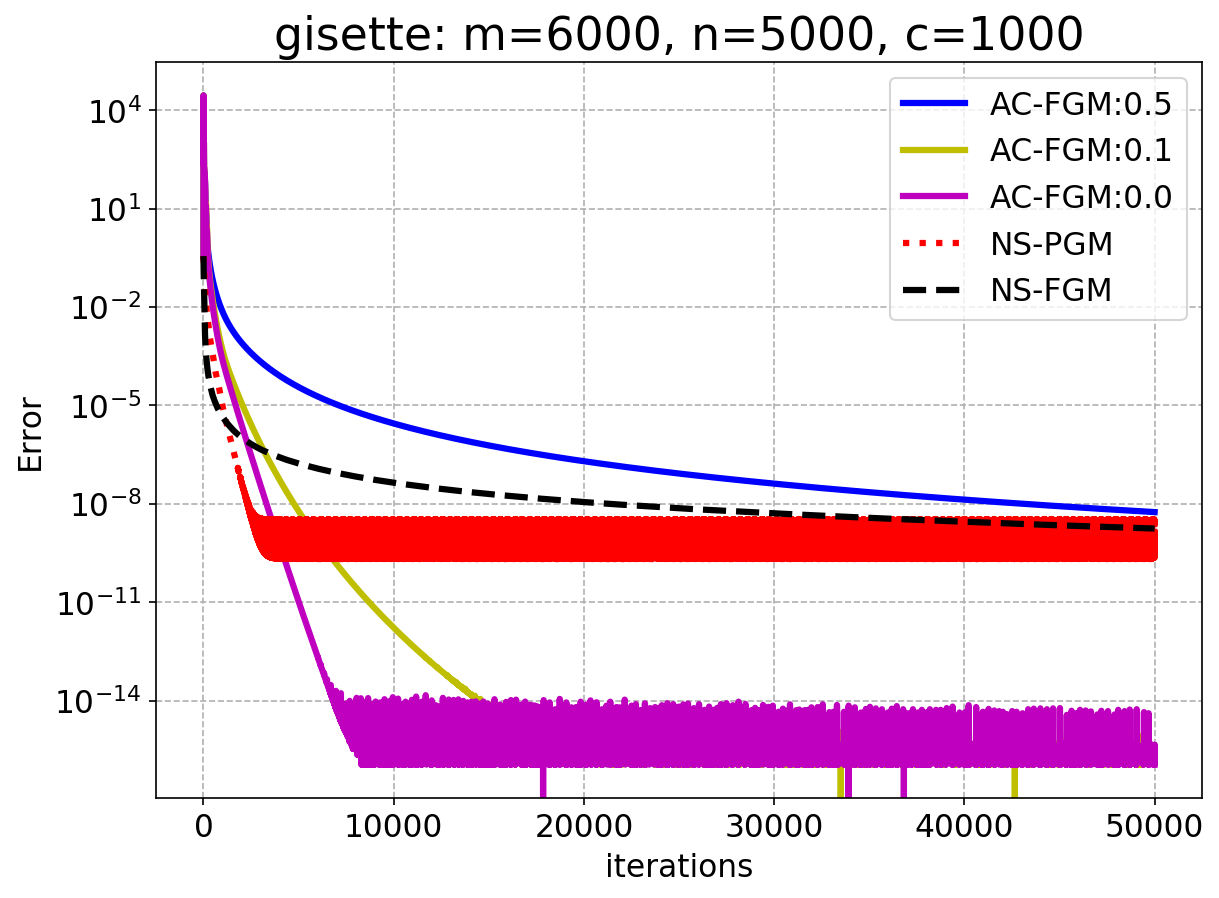}
	\end{minipage}
	\begin{minipage}[t]{0.45\linewidth}
		\centering
		\includegraphics[width=5.2cm]{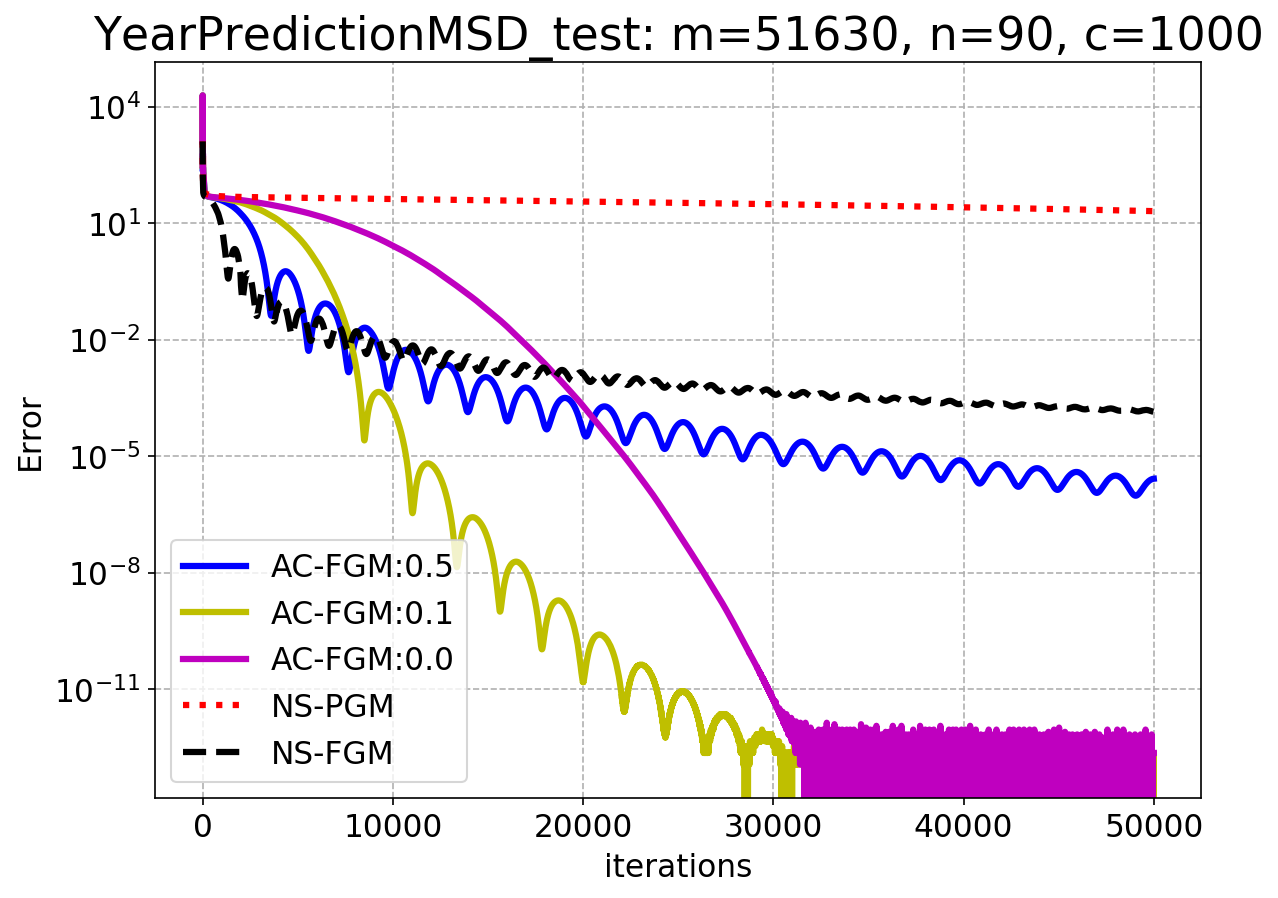}
	\end{minipage}
	\caption{Square root Lasso~\eqref{square_root_lasso_prob}: Comparison between AC-FGM, NS-AGD, NS-FGM, and AdGD in terms of the number of iterations for datasets \emph{gisette} (column 1) and \emph{YearPredictionMSD.test} (column 2). Row 1 takes $\lambda = 100\cdot m^{-1/2}\Phi^{-1}(1-0.01/n)$ and Row 2 takes $\lambda = 1000\cdot m^{-1/2}\Phi^{-1}(1-0.01/n)$. We set $\epsilon=10^{-8}$ when implementing the algorithms.}
	\label{fig_SRL_1}
\end{figure}

\begin{table}[h]\scriptsize  
	\centering
 \renewcommand\arraystretch{1.3}
 \tabcolsep=0.15cm
	\begin{tabular}{|c|c|c|ccc|ccc|}
 \hline
 \multirow{3}*{Instances} & \multirow{3}*{m} & \multirow{3}*{n}&\multicolumn{3}{c|}{\multirow{2}*{Avg CPU time (per 1000 iterations)}}& \multicolumn{3}{c|}{\multirow{2}*{Avg \# oracle calls (per iteration)}} \\
 & ~ & ~ & ~ & ~ & ~& ~ &~ & ~\\
	& ~ &	~ & AC-FGM & NS-PGM & NS-FGM & AC-FGM & NS-PGM & NS-FGM\\\hline
  gisette ($c=100$)& 6000 & 5000 & 56.61 & 56.11 & 213.62 & 1 &2.0 & 4.0\\
  gisette ($c=1000$)& 6000 & 5000& 59.16 & 129.11 & 228.05& 1 &2.0 & 4.0\\
  YearPredictMSD ($c=100$)  & 51630 & 90 & 12.63 & 26.02 & 42.31 & 1 &2.0 & 4.0\\
  YearPredictMSD ($c=1000$)& 51630 & 90 &  12.52 & 25.27  & 43.52 & 1 & 2.0 & 4.0
  \\\hline
	\end{tabular}
	\caption{Square root Lasso~\eqref{square_root_lasso_prob}: Comparison of the averaged CPU time in [sec] and averaged oracle calls per iteration.}\label{table_square_root_lasso}
\end{table}

\revision{}{From Figure~\ref{fig_SRL_1}, we observe that AC-FGM consistently outperforms NS-FGM in the experiments in terms of the number of iterations. Also notice that, in this set of experiments, the error may oscillate at a relatively high accuracy, which is expected given the design of these universal or uniform methods. In \emph{gistte} with $c=1000$, NS-PGM achieves slightly faster convergence than AC-FGM in terms of the number of iterations. However, due to the line search steps, NS-PGM's computational cost is roughly two times of AC-FGM (see Table~\ref{table_square_root_lasso}). Therefore, AC-FGM with $\alpha=0.0$ and $0.1$ still outperforms NS-PGM in terms of the overall efficiency.}

\subsection{Sparse logistic regression}\label{subsec_num_3}
In this subsection, we consider another classical convex composite optimization problem, known as the logistic regression with $\ell_1$-regularization, 
\begin{align}\label{logistic_regression}
\min_{x\in \bbr^n}~\Psi(x) := \tsum_{i=1}^m \log\left( 1 + \exp(-b_i \langle a_i,x\rangle)\right) + \lambda \|x\|_1,
\end{align}
where $\lambda >0$ is the regularization parameter, and $a_i \in \bbr^n$ and $b_i \in \{-1, 1\}$ represent the feature vector and the binary index of each sample, respectively. \revision{This problem can be viewed as a composite optimization of a smooth function and a nonsmooth function, i.e.,}{This problem can be equivalently written as}
\begin{align*}
\Psi(x) = f(x) + h(x) = \tilde f(Kx) + h(x),
\end{align*}
where $\tilde f(y) = \tsum_{i=1}^m \log (1+\exp(y_i))$, $h(x) = \lambda \|x\|_1$, and $K\in \bbr^{m \times n}$ satisfies $K_{i,j}=-a_{i}^{(j)}b_i$. Clearly, the Lipschitz constant of $\tilde f$ is bounded by $1/4$, thus the Lipschitz constant of $f$ can be upper bounded by $L = \lambda_{\max}(K^\top K)/4$. It is also noteworthy that the local smoothness constant of $\tilde f$ can be much smaller than $1/4$ as $y$ gets far away from the origin. Therefore, if we set the origin as the initial point $x_0$, the adaptive algorithms that can take advantage of the local smoothness level may have much faster convergence than the non-adaptive algorithms. 

\revision{}{We test the performance of the algorithms for solving \eqref{logistic_regression} on 4 large-scale datasets in LIBSVM, e.g., \emph{gisette (m=6000, n=5000)}, \emph{rcv1.binary (m=20242, n=47236)}, \emph{real-sim (m=72309, n=20958)}, and \emph{covtype.binary (m=581012, n=54)}, and take two choices of penalty parameters, $\lambda = 0.001\|A^\top b\|_\infty$ and $0.005\|A^\top b\|_\infty$, with $A=[a_1^T;\ldots, a_m^T]$ and $b=(b_1;\ldots,b_m)$. The results are reported in Figure~\ref{fig:logistic_regression}, Figure~\ref{fig:logistic_regression_2} and Table~\ref{table_logistic_regression}.\footnote{The performance of the algorithms on more problem instances in LIBSVM can be found in our publicly available repository (\url{https://github.com/tli432/AC-FGM-Implementation}).}
}

\begin{figure}[h]
	\centering
 \begin{minipage}[t]{0.45\linewidth}
		\centering
		\includegraphics[width=5cm]{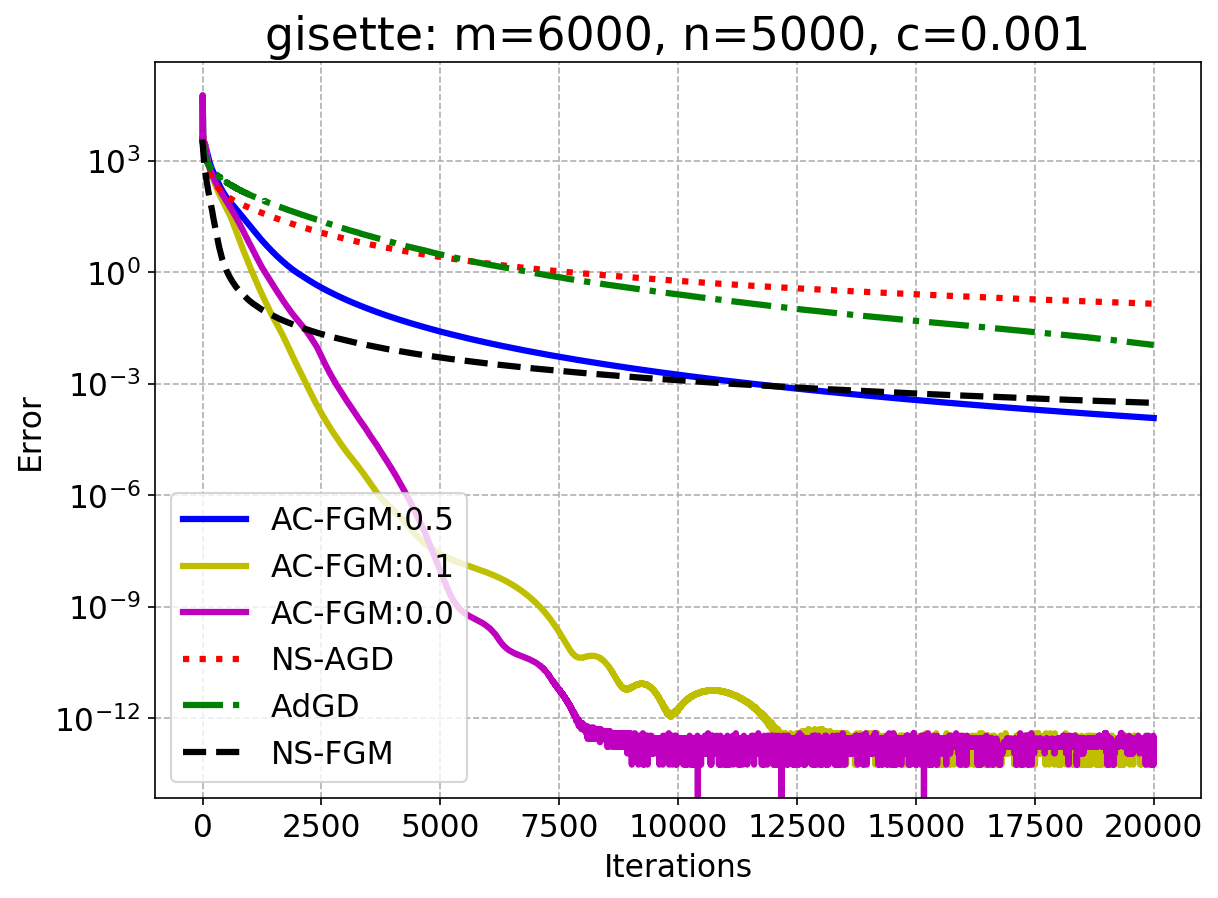}
	\end{minipage}
 \begin{minipage}[t]{0.45\linewidth}
		\centering
		\includegraphics[width=5cm]{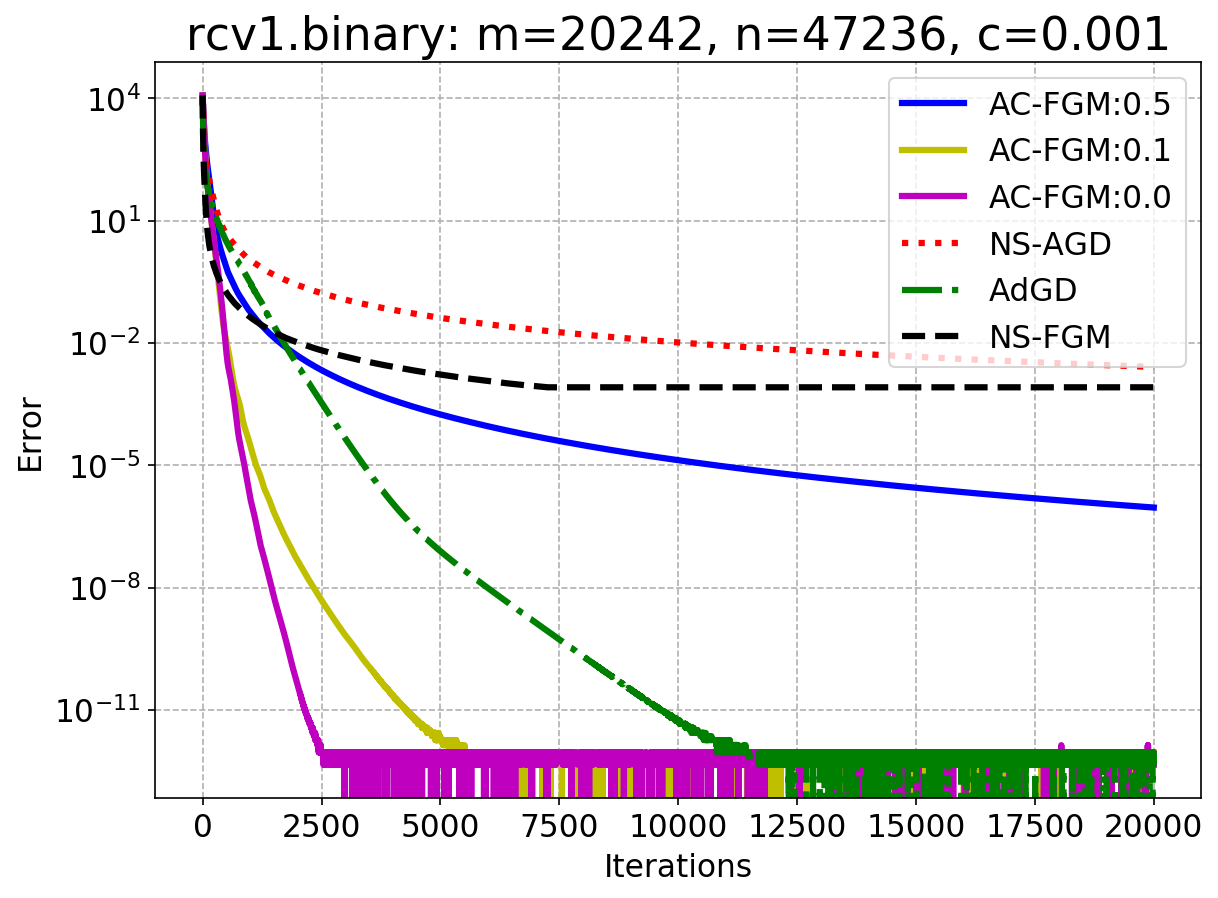}
	\end{minipage}
 \begin{minipage}[t]{0.45\linewidth}
		\centering
		\includegraphics[width=5cm]{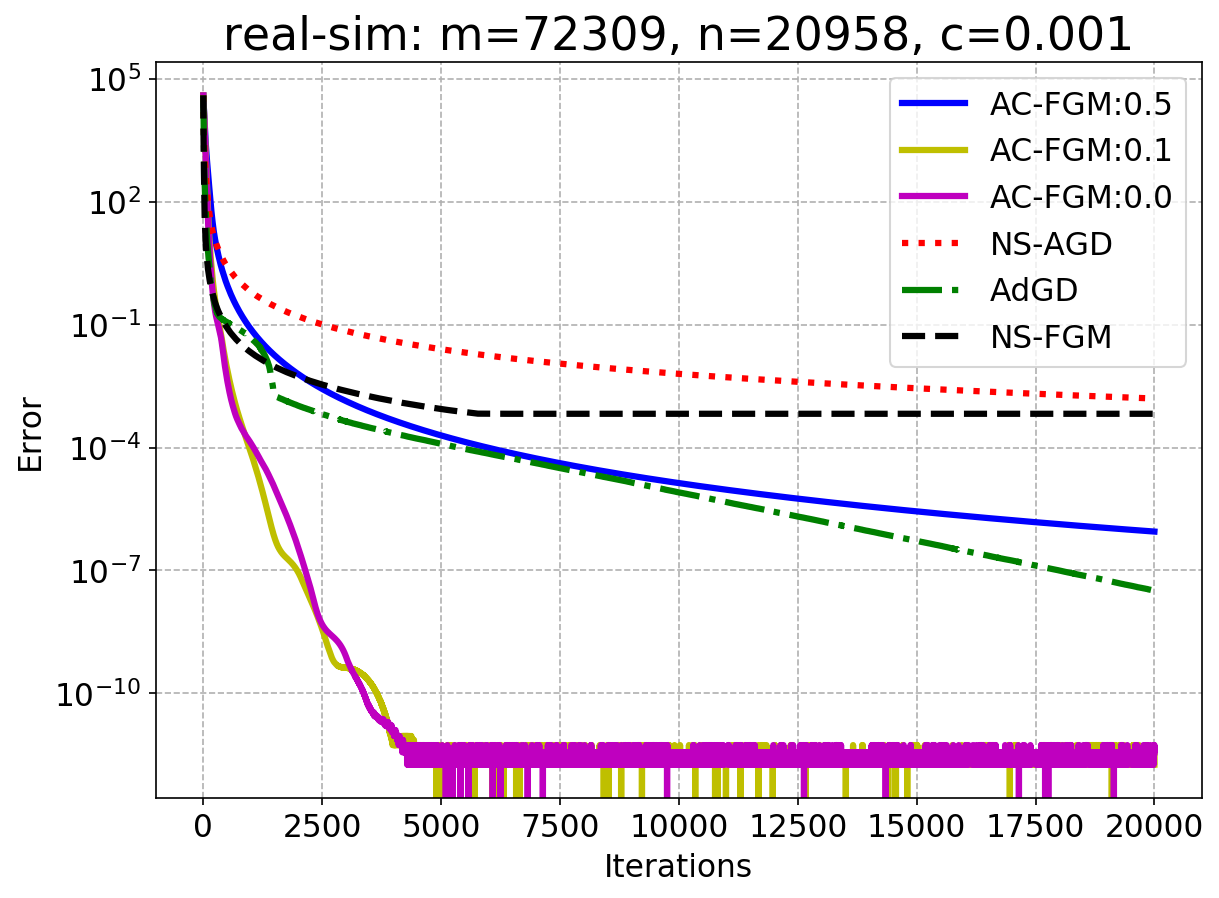}
	\end{minipage}
 \begin{minipage}[t]{0.45\linewidth}
		\centering
		\includegraphics[width=5cm]{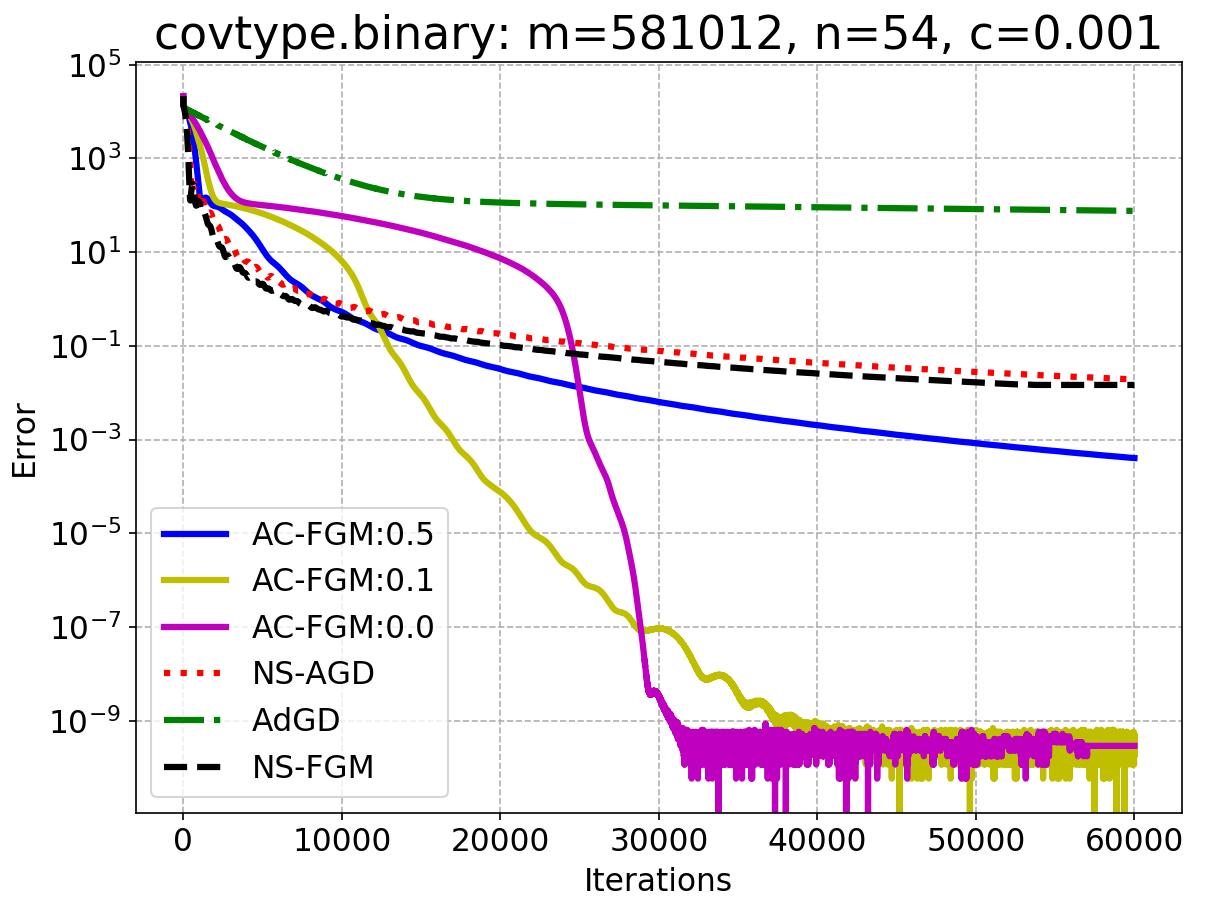}
	\end{minipage}
	\caption{Sparse logistic regression~\eqref{logistic_regression}: Comparison between AC-FGM, NS-AGD, NS-FGM, and AdGD in terms of the number of iterations for datasets \emph{gisette}, \emph{rcv1.binary}, \emph{real-sim}, and \emph{covtype.binary}. Penalty parameter: $\lambda = 0.001\|A^\top b\|_\infty$.}
	\label{fig:logistic_regression}
\end{figure}

\begin{figure}[h]
	\centering
 \begin{minipage}[t]{0.45\linewidth}
		\centering
		\includegraphics[width=5cm]{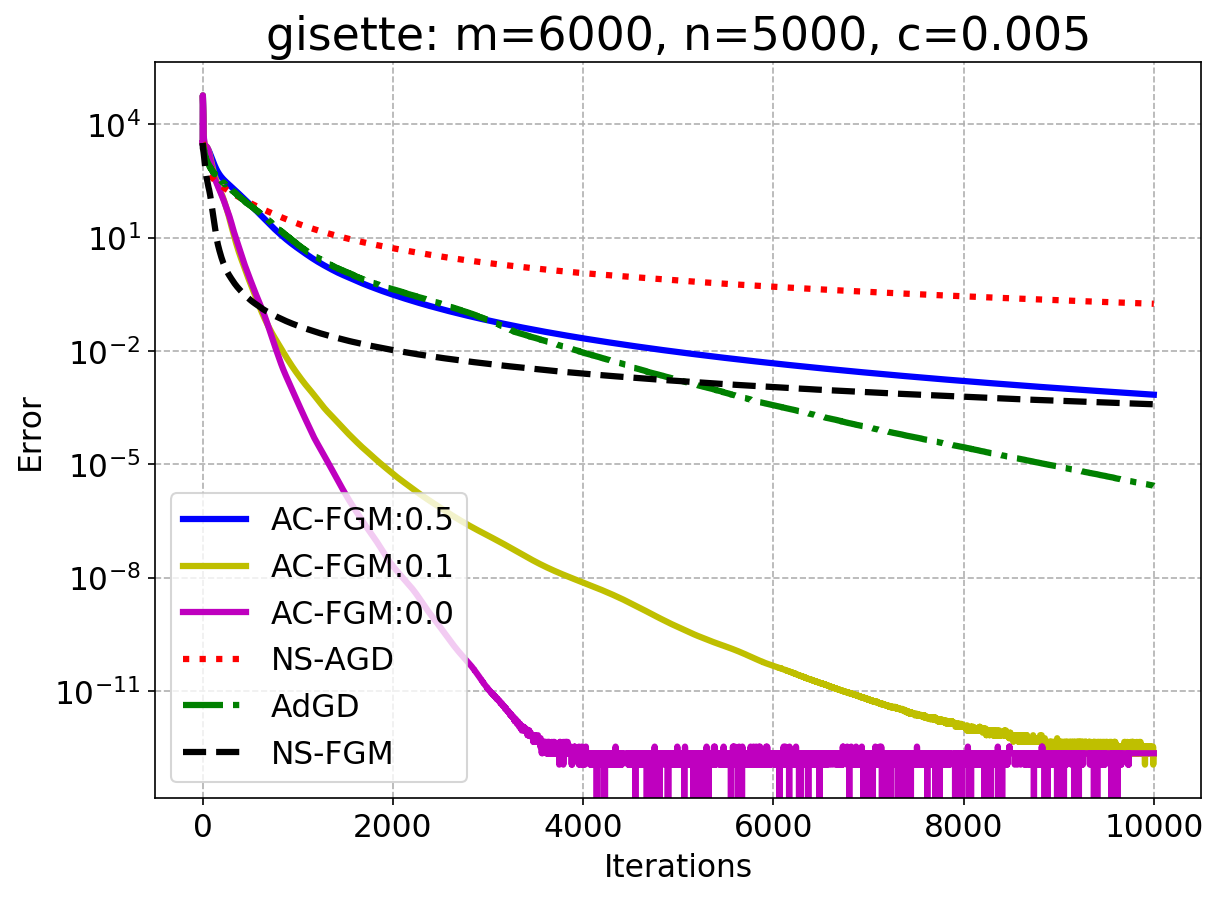}
	\end{minipage}
 \begin{minipage}[t]{0.45\linewidth}
		\centering
		\includegraphics[width=5cm]{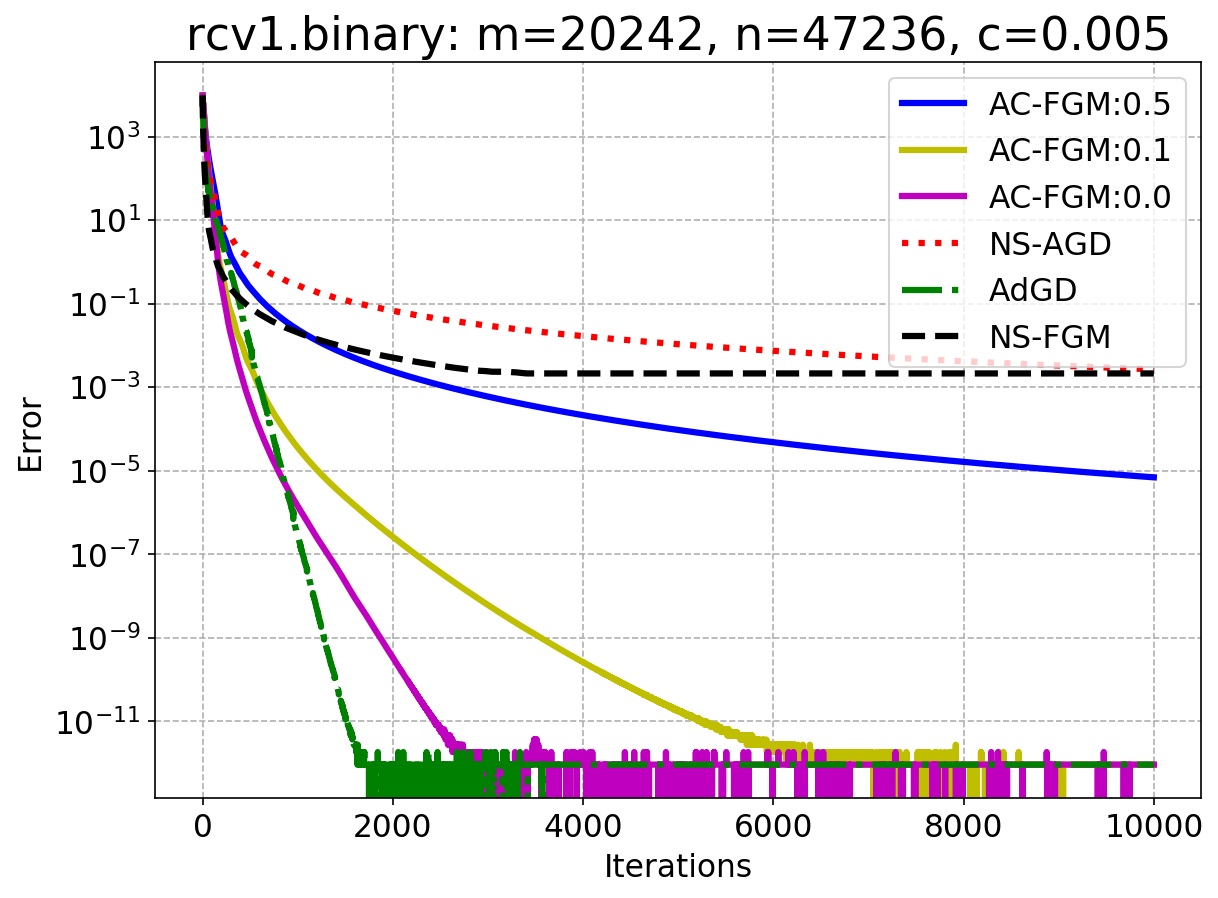}
	\end{minipage}
 \begin{minipage}[t]{0.45\linewidth}
		\centering
		\includegraphics[width=5cm]{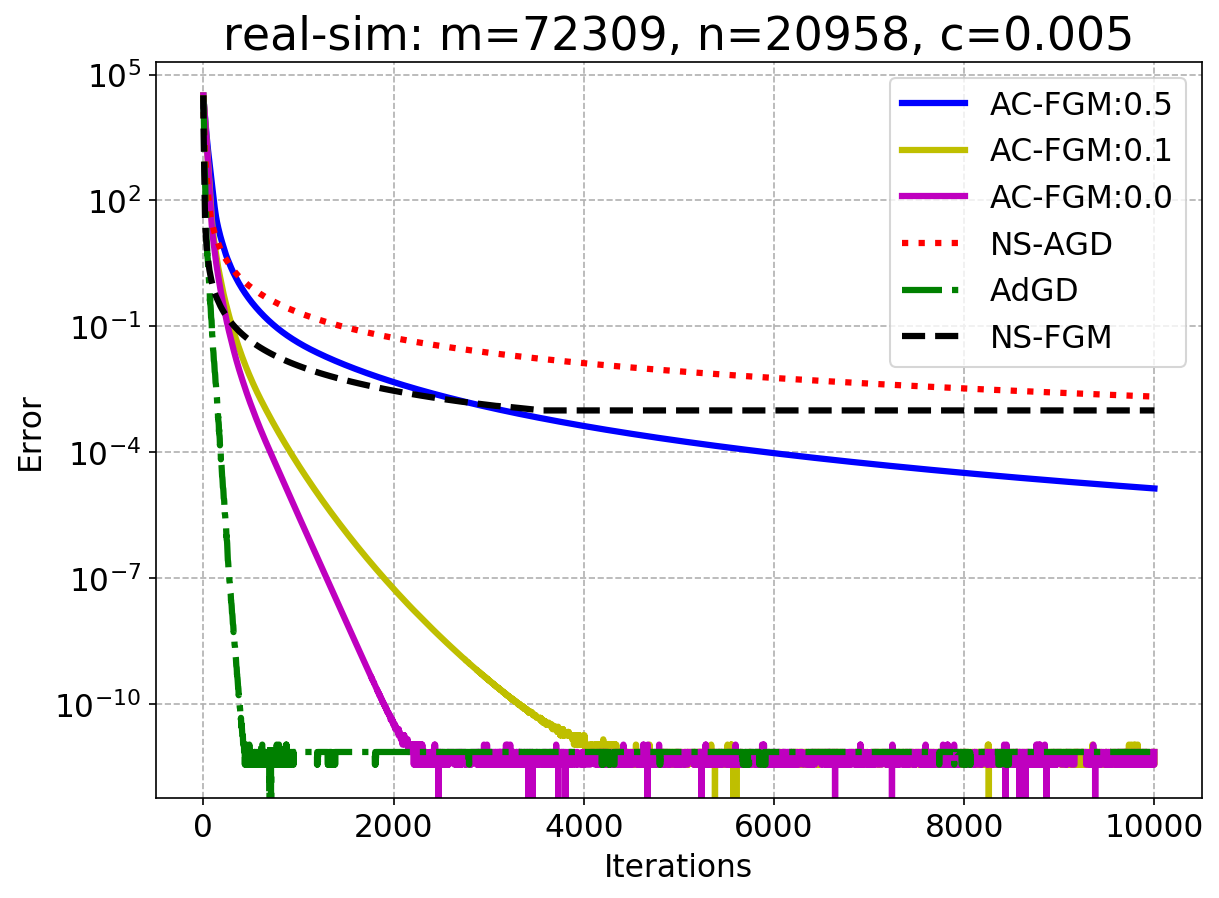}
	\end{minipage}
 \begin{minipage}[t]{0.45\linewidth}
		\centering
		\includegraphics[width=5cm]{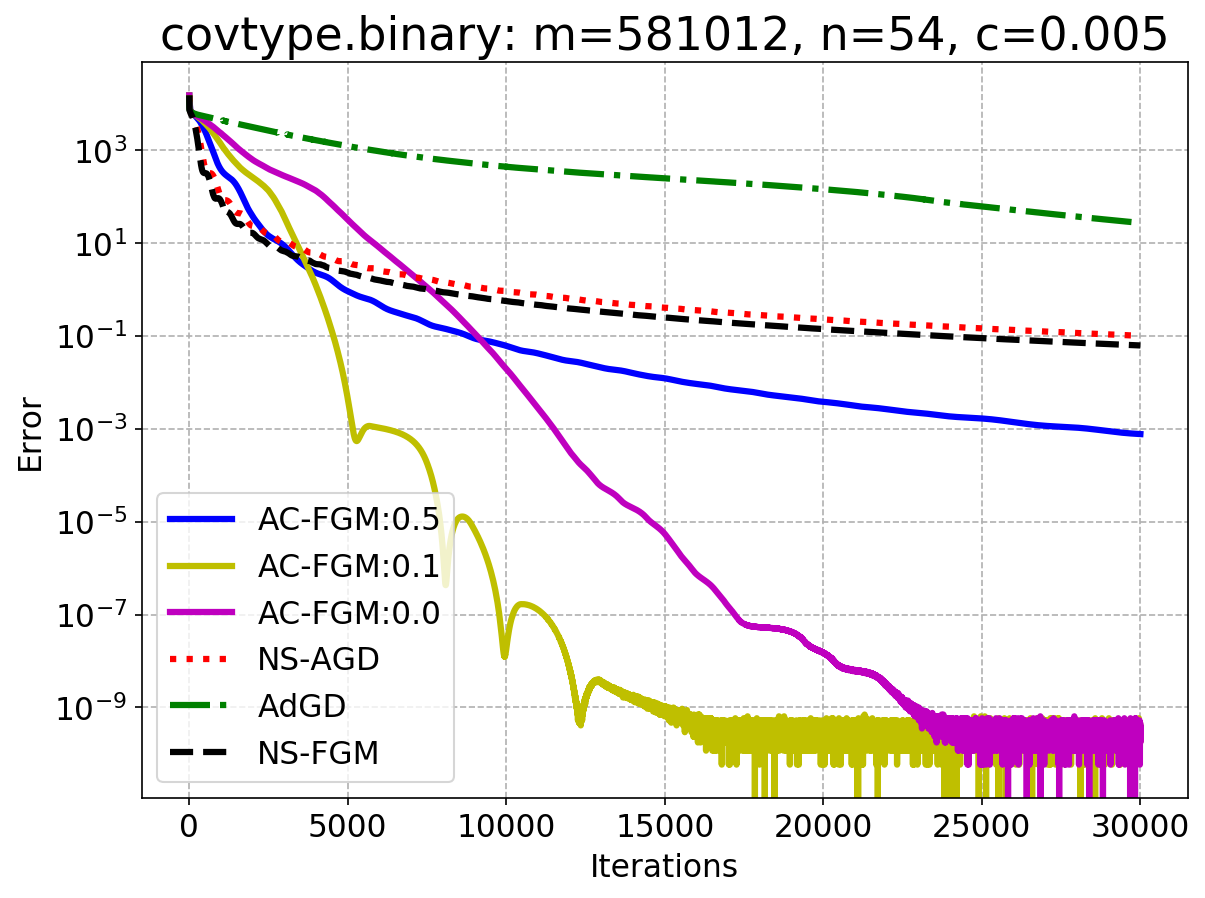}
	\end{minipage}
	\caption{Sparse logistic regression~\eqref{logistic_regression}: Comparison between AC-FGM, NS-AGD, NS-FGM, and AdGD in terms of the number of iterations for datasets \emph{gisette}, \emph{rcv1.binary}, \emph{real-sim}, and \emph{covtype.binary}. Penalty parameter: $\lambda = 0.005\|A^\top b\|_\infty$.}
	\label{fig:logistic_regression_2}
\end{figure}

\begin{table}[h]\scriptsize  
	\centering
 \renewcommand\arraystretch{1.3}
 \tabcolsep=0.15cm
	\begin{tabular}{|c|c|c|ccc|ccc|}
 \hline
 \multirow{3}*{Instances} & \multirow{3}*{m} & \multirow{3}*{n}&\multicolumn{3}{c|}{\multirow{2}*{Avg CPU time (per 1000 iterations)}}& \multicolumn{3}{c|}{\multirow{2}*{Avg \# oracle calls (per iteration)}} \\
 & ~ & ~ & ~ & ~ & ~& ~ &~ & ~\\
	& ~ &	~ & AC-FGM & AdGD & NS-FGM & AC-FGM & AdGD & NS-FGM\\\hline
  gisette ($c=0.001$)& 6000 & 5000 & 61.26 & 61.86 & 198.53 & 1 &1 & 4.0\\
  gisette ($c=0.005$)& 6000 & 5000& 61.07 & 62.12 & 195.15 & 1 &1 & 4.0\\
  rcv1.binary ($c=0.001$)  & 20242 & 47236 & 16.10 & 15.90 & 56.11 & 1 &1 & 4.0\\
  rcv1.binary ($c=0.005$)& 20242 & 47236 &  16.43 & 15.65  & 54.84 & 1 & 1 & 4.0\\
  real-sim ($c=0.001$)& 72309 & 20958 & 39.42 & 39.19 & 149.78 & 1 &1 & 4.0\\
  real-sim ($c=0.005$)& 72309 & 20958 & 37.16 & 38.62 & 148.25 & 1 &1 & 4.0\\
  covtype.binary ($c=0.001$)  & 581012 & 54 & 47.94  & 42.27 & 172.68 & 1 &1 & 4.0\\
  covtype.binary ($c=0.005$)& 581012 & 54 &  46.98 & 40.32  & 173.01 & 1 & 1 & 4.0
  \\\hline
	\end{tabular}
	\caption{Sparse logistic regression~\eqref{logistic_regression}: Comparison of the averaged CPU time in [sec] and averaged oracle calls per iteration.}\label{table_logistic_regression}
\end{table}

\revision{}{Figure~\ref{fig:logistic_regression} and Figure~\ref{fig:logistic_regression_2} indicate that AC-FGM with $\alpha=0.0$ and $0.1$ outperform NS-FGM in almost all test cases. We also notice that AdGD converges fast in \emph{real-sim} and \emph{rcv.binary} when $\lambda = 0.005\|A^\top b\|_\infty$. However, AC-FGM appears to be more robust to different choices of penalty parameters.}

\subsection{Ablation analysis for line search in first iteration}\label{sec_ablation_analysis}
\revision{}{In this subsection, we conduct ablation analysis for the line search procedure in the first iteration of AC-FGM. Specifically, we run the experimental group with line search in the first iteration and the control group without line search in the first iteration. For the control group, the initial stepsize is chosen in the same way as the previous subsections, i.e., $\eta_1:=\tfrac{2}{5 L_0}$ where $L_0$ is defined in \eqref{def_L_0}. For the experimental group, we set the line search constant $\gamma = 1.5$ and start with a relatively large stepsize to ensure that the line search terminates after at least 2 steps. 
We re-use three problem instances from the previous sections to conduct the experiments, covering Lasso \eqref{QP_prob_with_lasso}, square root Lasso \eqref{square_root_lasso_prob}, and sparse logistic regression \eqref{logistic_regression}. From Figure~\ref{ablation_analysis}, we observe that there is no significant performance difference between the experimental group and control group, demonstrating that the long-term performance of AC-FGM is not sensitive to the line search step in the initial iteration. }

\begin{figure}[H]
	\centering
	\begin{minipage}[t]{0.3\linewidth}
		\centering
		\includegraphics[width=4.3cm]{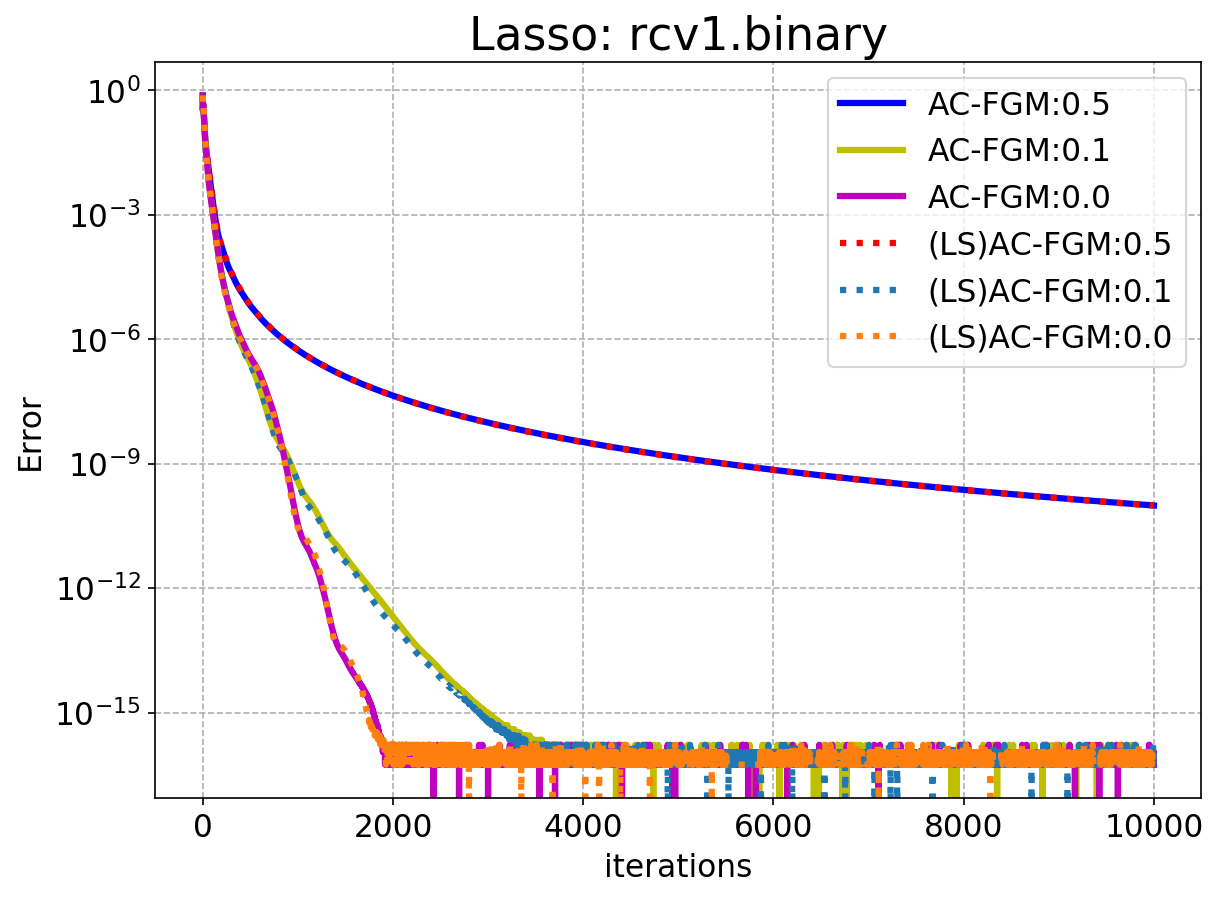}
	\end{minipage}
	\begin{minipage}[t]{0.3\linewidth}
		\centering
		\includegraphics[width=4.3cm]{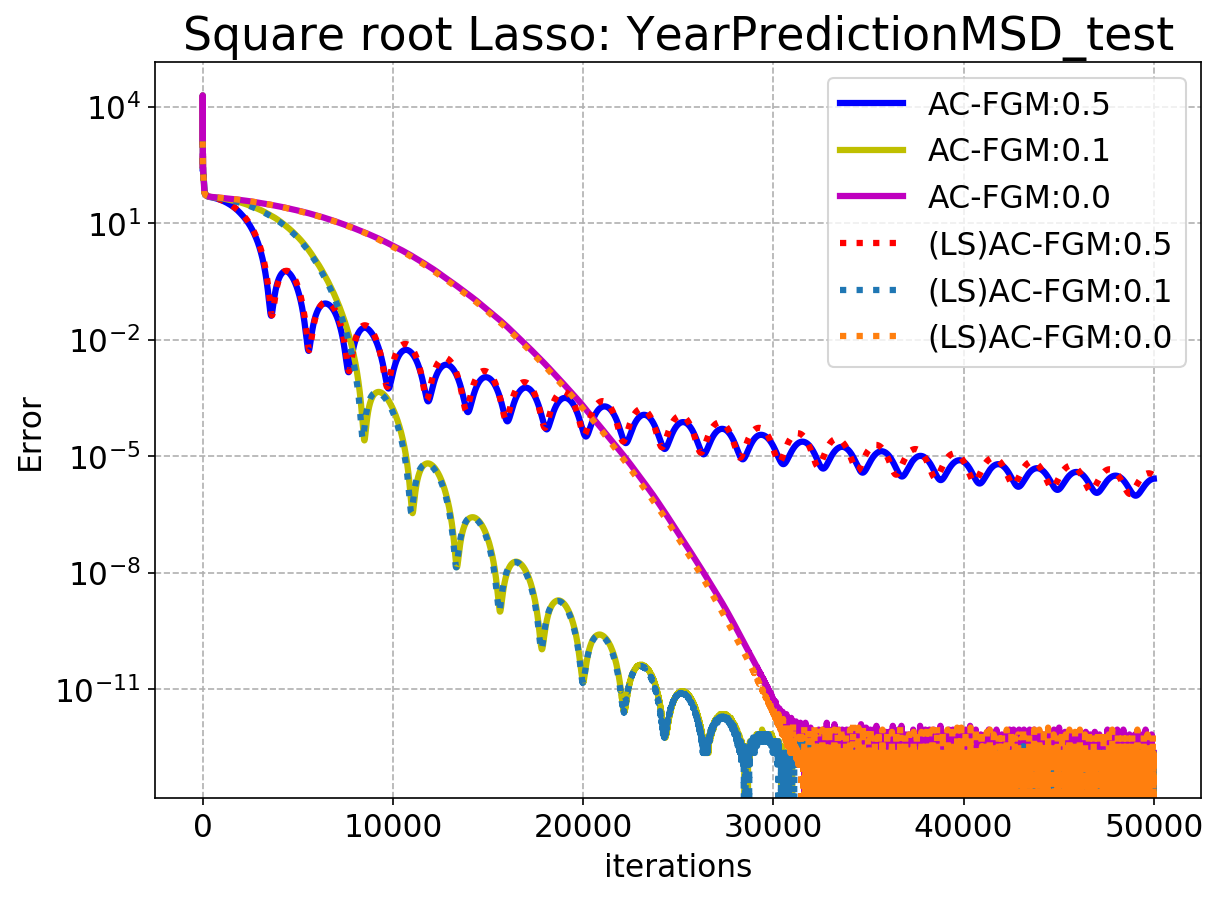}
	\end{minipage}
 \begin{minipage}[t]{0.3\linewidth}
		\centering
		\includegraphics[width=4.3cm]{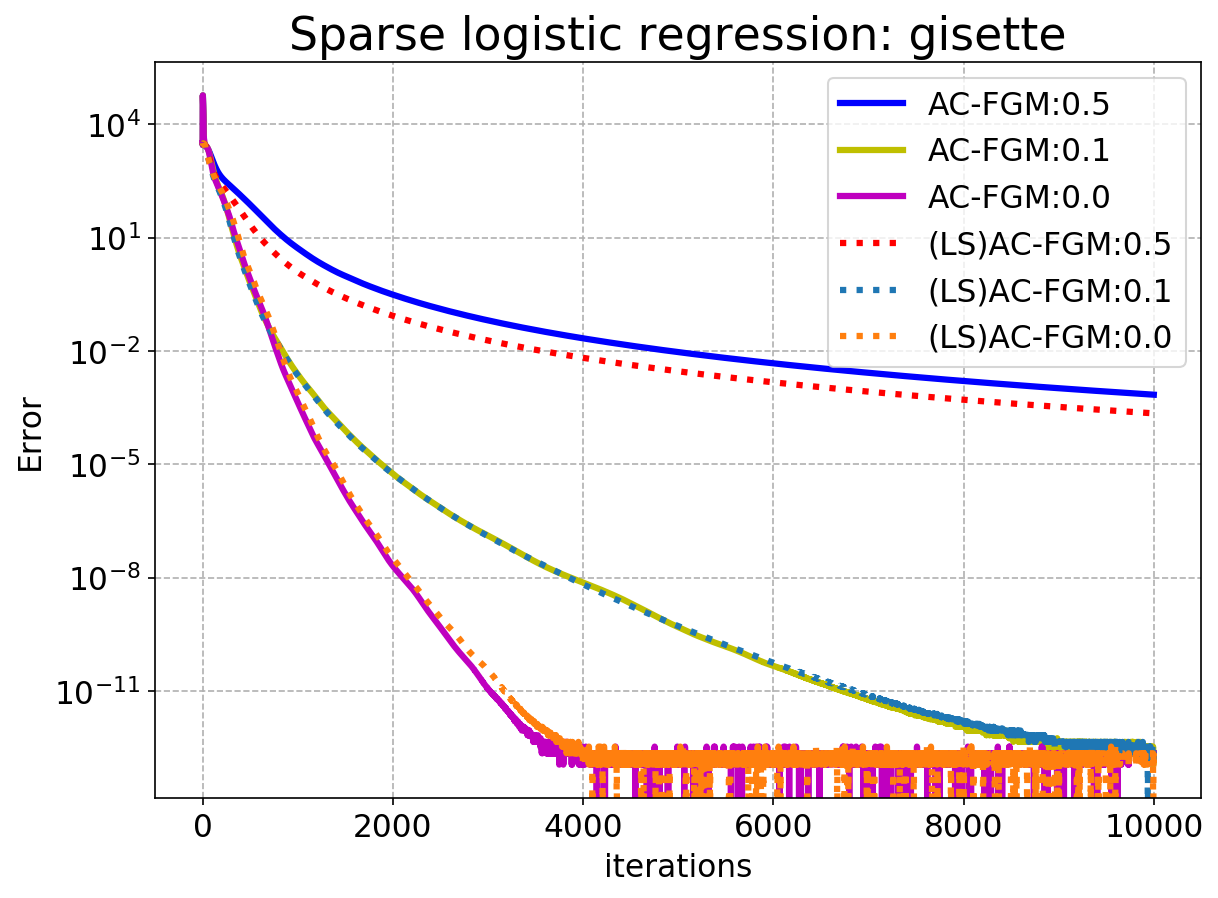}
	\end{minipage}
	\caption{Ablation analysis on initial line search: Test on \emph{rcv1.binary (c=0.01)} in Lasso \eqref{QP_prob_with_lasso}, \emph{YearPredictionMSD.test (c=1000)} in square root Lasso \eqref{square_root_lasso_prob}, and \emph{gisette (c=0.005)} in sparse logistic regression \eqref{logistic_regression}. ``LS'' represents the experimental group with initial line search. }
	\label{ablation_analysis}
\end{figure}

\section{Concluding remarks}
In summary, we design a novel accelerated gradient decent type algorithm, AC-FGM, for convex composite optimization. We first show that AC-FGM is fully problem parameter free, line search free, and optimal for smooth convex problems. Then we extend AC-FGM to solve convex problems with H\"{o}lder continuous gradients and demonstrate its uniform optimality for all smooth, weakly smooth, and nonsmooth objectives, \revision{}{with the desired accuracy of the solution as the only input.} 
At last, we present some numerical experiments, where AC-FGM appears to be more advantageous over the previously developed parameter-free methods on a wide range of testing problems.

\renewcommand\refname{Reference}

\bibliographystyle{abbrv}
\bibliography{arxiv_20240816}

\end{document}